\documentclass[11pt,fleqn]{article}
\pdfoutput=1
\usepackage[utf8]{inputenc}
\usepackage[ngerman,english]{babel}
\usepackage{amsmath,amsfonts,amssymb}
\usepackage{dsfont}
\usepackage{amsthm}
\usepackage{stmaryrd}
\usepackage{lmodern}
\usepackage{fancybox}
\usepackage{a4wide}
\usepackage[colorlinks,citecolor=blue,linkcolor=blue]{hyperref}
\usepackage{enumerate}
\usepackage{color}
\usepackage{bigints}
\usepackage{relsize}
\usepackage{graphicx}
\graphicspath{ {./Grafiken/} }
\usepackage{float}
\usepackage{etoolbox}
\usepackage{mathtools}
\usepackage[dvipsnames]{xcolor}
\usepackage{caption}
\usepackage{booktabs}
\usepackage{chngcntr}
\usepackage{todonotes}
\usepackage{nameref}
\usepackage[bottom=3cm,top=2.5cm]{geometry}
\usepackage{tikzsymbols}
\usepackage{accents}
\usepackage[para]{footmisc}
\usepackage{authblk}
\usepackage{blindtext}
\usepackage{footmisc}
\usepackage{subcaption}
\usepackage{longtable}
\usepackage[toc,page,header]{appendix}
\usepackage{xr}

\usepackage{natbib}
\setcitestyle{authoryear} 
\usepackage{fancyhdr}

\fancypagestyle{plain}{% emulate the plain style
	\fancyhf{}% clear all fields
	\fancyfoot[C]{\thepage}%
}
\counterwithin{figure}{section}
\counterwithin{equation}{section}
\counterwithin{table}{section}

\allowdisplaybreaks

\hypersetup{pdfstartview={XYZ null null 1.15}}

\DeclareMathOperator{\cD}{\mathcal{D}}
\DeclareMathOperator{\N}{\mathbb{N}}
\DeclareMathOperator{\Sprocess}{\mathbb{S}}

\DeclareMathOperator{\Z}{\mathbb{Z}}

\DeclareMathOperator{\R}{\mathbb{R}}

\DeclareMathOperator{\E}{\mathbb{E}}
\DeclareMathOperator{\Wkeit}{P}%{\mathbb{P}}

\DeclareMathOperator{\Varianz}{Var}
\DeclareMathOperator{\Covarianz}{Cov}

\DeclareMathOperator{\tdist}{t}
\DeclareMathOperator{\Exp}{Exp}

\DeclareMathOperator{\vol}{vol}

\renewcommand{\le}{\leqslant}

\renewcommand{\ge}{\geqslant}

\newtheorem{Theorem}{Theorem}[section]
\newtheorem{Lemma}[Theorem]{Lemma}

\newtheorem{Assumption}{Assumption}[section]

\theoremstyle{definition}

\newtheorem{Definition}{Definition}[section]
\newtheorem{Remark}{Remark}[section]

\addto\captionsenglish{%

}

\newcommand*{\fullref}[1]{\hyperref[{#1}]{\hyperref{#1}(\ref*{#1})}}

\newcommand{\closure}[1]{\overline{#1}}
\newcommand{\inner}[1]{{#1}^{\circ}}

\newcommand{\balpha}{\boldsymbol{\alpha}}
\newcommand{\bz}{\boldsymbol{z}}
\newcommand{\bx}{\boldsymbol{x}}
\newcommand{\by}{\boldsymbol{y}}

\newcommand{\bt}{\boldsymbol{t}}
\newcommand{\bs}{\boldsymbol{s}}
\newcommand{\ba}{\boldsymbol{a}}

\newcommand{\bk}{\boldsymbol{k}}
\newcommand{\bl}{\boldsymbol{l}}

\newcommand{\bd}{\mathbf{d}}

\newcommand{\Tstar}{\ceil{T^{1/2}}}

\newcommand{\geschweift}[1]{\left\lbrace #1 \right\rbrace}
\newcommand{\ceil}[1]{\left\lceil #1 \right\rceil}
\newcommand{\floor}[1]{\left\lfloor #1 \right\rfloor}
\newcommand{\linksoffen}[1]{\left( #1 \right]}
\newcommand{\rechtsoffen}[1]{\left[ #1 \right)}

\newcommand{\nB}{\operatorname{n\kern-0.12em B}}
\newcommand{\F}[1]{\operatorname{{}_{#1}\kern-0.14em F}}
\newcommand{\FnB}[1]{\operatorname{{}_{#1}\kern-0.14em F\kern-0.12em n\kern-0.12em B}}

\newcommand{\dto}{\overset{\cD}{\rightarrow}}
\newcommand{\pto}{\overset{P}{\rightarrow}}
\newcommand{\wto}{\overset{w}{\rightarrow}}

\newcommand{\EW}[1]{\E\left[#1\right]}
\newcommand{\Var}[1]{\Varianz\left[#1\right]}
\newcommand{\Cov}[2]{\Covarianz\left[#1,#2\right]}
\newcommand{\Prob}[1]{\Wkeit\left(#1\right)}
\newcommand{\supp}[1]{\underset{#1}{\sup}}

\newcommand{\inff}[1]{\underset{#1}{\inf}}

\newcommand{\maxx}[1]{\underset{#1}{\max}}

\newcommand{\abs}[1]{\left|#1\right|}

\newcommand{\thr}{\beta}

\newcommand{\danobipixel}{Y}

\newcommand{\danobisum}{S}
\newcommand{\danobimean}{\overline{S}}

\newcommand{\scanset}{A}
\newcommand{\tscanset}{\widetilde{A}}

\newcommand{\anomaly}{F}
\newcommand{\Hoelder}{G}
\newcommand{\largeblock}{B}
\newcommand{\bigblock}{L}
\newcommand{\smallblock}{N}
\newcommand{\inside}{I}

\title{\textbf{Scan statistics for the detection of anomalies in $M$-dependent random fields with applications to image data}}
\author{Claudia Kirch$^1$ \quad Philipp Klein$^2$\quad Marco Meyer$^3$}
\date{\today}

\begin{document}
	
	\maketitle
	
	\footnotetext[1]{Department of Mathematics, Otto-von-Guericke University; Center for Behavioral Brain Sciences (CBBS); Magdeburg, Germany.
		Email: \url{claudia.kirch@ovgu.de}.
	}
	
	\footnotetext[2]{Department of Mathematics, Otto-von-Guericke University; Magdeburg, Germany.
		Email: \url{philipp.klein@ovgu.de}.
	}
	
	\footnotetext[3]{Department of Mathematics, Leibniz University; Hannover, Germany.
		Email: \url{marco.meyer@math.uni-hannover.de}}

	\begin{abstract}
Anomaly detection in random fields is an important problem in many applications including the detection of cancerous cells in medicine, obstacles in autonomous driving and cracks in the construction material of buildings. Such anomalies are often visible as areas with different expected values compared to the background noise. Scan statistics based on local means have the potential to detect such local anomalies by enhancing relevant features. We derive limit theorems for a general class of such statistics over $M$-dependent random fields of arbitrary but fixed dimension. By allowing for a variety of combinations and contrasts of sample means over differently-shaped local windows, this yields a flexible class of scan statistics that can be tailored to the particular application of interest. The latter is demonstrated for crack detection in 2D-images of different types of concrete. Together with a simulation study this indicates the potential of the proposed methodology for the detection of anomalies in a variety of situations. 
	\end{abstract}

	\section{Introduction}

	Scan statistics in spatial data were first introduced by \cite{Kulldorf97} using a Likelihood-ratio test in order to detect clusters.
	Under the assumption of independence of the noise, there exist methods using 2D windows (see \cite{HaimanPreda}), hypercubes with arbitrary dimension (see \cite{Kabluchko} for a single-window approach and \cite{Jiang,SharpnackAriasCastro} for multiple window approaches) and spatial scan statistics for point processes (\cite{Glaz}). Under the assumption of a weak invariance principle, \cite{Bucchia} introduced a test for an epidemic change. \cite{ProkschWernerMunk} and \cite{MunkProksch} introduced tests on inverse regression models with independent, but not necessarily identically distributed noise. Other approaches include the use of EM algorithms (see \cite{Moonetal}),  spatial scan statistics (see \cite{Kulldorf}) and deformable models (see \cite{McInerneyTerzopoulos}).
	 \cite{Glaz} provides an overview of scan statistics in general. 

  Scan statistics based on a suitable combination of windows can be used to detect local structure (which will be called anomaly from now on) by enhancing relevant features taking their geometric properties into account. The family-wise error rate to detect voxels close to an anomaly can be controlled
  using a threshold based on quantiles of the corresponding limit distribution for stationary $M$-dependent random fields.
Such anomaly detection has been shown to be an important problem in many applications including the detection of cancerous cells in medicine (see  \cite{JamesClymerSchmalbrock}, \cite{Moonetal}), obstacles in autonomous driving (see \cite{Koopman}) and cracks in the construction material of buildings (see \cite{ItoAoki}, \cite{TangGu}).
  
	We will concentrate on the specific task of the detection of anomalies (e.g.\ cracks or fissures) in building material such as concrete. In this situation  CT scans have been used in order to analyze the structure of materials, compare e.g.\ \cite{Weiseetal} or \cite{Baranowskietal}. Early methods  focus on the analysis of the gray value distribution (see e.g.\ \cite{AcostaFigueroa}), other methods use algorithms from image processing for the anomaly detection in 2D  data including Template Matching (\cite{Roseman}), algorithms to detect minimal paths (\cite{Amhaz}) and percolation based on the Hessian matrix (\cite{YamaguchiHashimoto}). For three-dimensional data, filtering methods like Frangi filters (see \cite{Frangi}, \cite{Frangi2}), sheet filters (\cite{Satoetal}) and algorithms to find minimal paths (\cite{Musebeck}) are used. See \cite{Ehrigetal}, \cite{Ehrigetal2} for a comparison of methods as well as \cite{KL} for a comprehensive overview over methods for both 2D and 3D data.
	Furthermore, Machine and Deep Learning methods like convolutional neural networks (see \cite{UNet}, \cite{UNet3D} and \cite{SegNet}) and random forests (\cite{Furat}, \cite{Shietal}) are frequently used for the segmentation of both 2D  and 3D image data.
	However, for large image data, these methods are computationally too expensive to be applied on the whole data set.
	On the other hand, they are computationally feasible if they are applied only locally in areas including cracks. Such areas can be determined by computationally fast scan statistics that come with certain statistical guarantees which are the focus of this paper. By construction, such scan statistics only use local information to determine pixels/voxels that are likely to contain an anomaly, which is computationally advantageous  in large image data.

	\subsection{Outline}
	
In Section \ref{sec_stat_mod}, we formulate a suitable model for the problem.
 This is followed in Section \ref{danobi_motivation} by examples of 2D slices of  concrete scans containing fissures (both artificial and non-artificial). These examples will serve as motivation for the definition of the general class of scanning statistics considered in the theoretic part.  In Section \ref{danobi_theory}, we develop asymptotic theory for this large class of local sums of stationary random fields. Our main result delivers a functional central limit theorem (FCLT) for combinations (via uniformly continuous functions) of a finite number of such local sums of potentially different shapes. This result extends clearly beyond existing theory on local sums of random fields (which is basically restricted to correctly oriented hypercubes) and is of its own interest. The main application for our FCLT is the derivation of thresholds for scan statistics based on convex sets with linear size.
 Coming back to the motivating examples, in Section \ref{sec_danobi_statistics}, we propose specific scanning procedures from the general class considered theoretically for the detection of fissures in concrete image data on 2D slices. Based on the geometric properties of fissures, we use a combination of scan statistics on rectangles and circular segments in order to find areas that potentially contain fissures.
	Since concrete contains aggregates with different geometric properties than fissures, we introduce alternative statistics in Sections \ref{sec_fnb2} and \ref{sec_fnb1} that aim at the detection of (potentially dangerous) fissures while simultaneously ignoring other structures which are not safety-related. In particular, this illustrates how to tailor scan statistics to the specific application at hand: Because the class of statistics investigated theoretically is quite large and allows for a lot of flexibility in choosing a specific scan statistic for a given application, this is important to transfer the methodology to other fields where anomaly detection plays a role.
	In Section \ref{danobi_simstudy} we analyze the performance of our procedure in a simulation study and furthermore illustrate the procedure based on 2D images with fissures.

	\subsection{Statistical model}\label{sec_stat_mod}
	To set the statistical framework for the theoretic analysis consider the following (sequence of) random fields, with spatial dimension $p\in \N$ and taking values in $\mathbb{R}$:
	\begin{align}\label{eq_model}
		Y_{\bk,T}=\mu_{\bk,T}+\epsilon_{\bk}, \quad \bk=(k_1,\ldots,k_p)^{\prime}, \quad k_i=1,\ldots,T\quad\text{ for } i=1,\ldots,p\,.
	\end{align}
	Here, $ \left(\epsilon_{\bk}\right)_{\bk\in\Z^p} $, is a stationary and centered sequence of $ M $-dependent random variables  fulfilling certain moment conditions. To elaborate,  $\left(\epsilon_{\bk}\right)_{\bk\in\Z^p} $ are called $ M $-dependent w.r.t.~a norm $ \left\|\cdot\right\| $ on $\R^p$ if for all $ \bs,\bt\in \Z^p $ with $ \left\|\bs-\bt\right\|>M $, $ \epsilon_{\bs} $ and $ \epsilon_{\bt} $ are independent. In the following, we always consider $M$-dependence w.r.t.\ to the supremum norm. By Lemma~\ref{lem_equivalence_M_dependence} in the appendix this is equivalent to $\widetilde{M}$-dependence for any equivalent norm, with  $\widetilde{M}$ depending on that norm.

	\begin{Assumption}\label{ass_errors}
		Let $\left(\epsilon_{\bk}\right)_{\bk\in\Z^p} $ be a sequence of (strictly) stationary $ M $-dependent random variables with $ \EW{\epsilon_{\bk}}=0 $, $\EW{\epsilon_{\bk}^2}\in(0,\infty)  $ and  $ \EW{\abs{\epsilon_{\bk}}^r}<\infty $ for some $ r>2p $.
		Define
		\begin{align}
			\sigma^2=\Var{\epsilon_{\boldsymbol{0}}}+\sum\limits_{\bk\ne 0}\Cov{\epsilon_{\boldsymbol{0}}}{\epsilon_{\bk}}, \label{long_run_variance}
		\end{align}
		where the latter sum is finite due to  the $M$-dependence.
	\end{Assumption}
	
	The $ \bk $ in \eqref{eq_model} constitute a regular $ p $-dimensional grid of $ T^p $ locations within $ \Z^p $. The asymptotic results in this paper will be for $ T \rightarrow \infty $. Effectively, this corresponds to so-called infill asymptotics: We can switch to the corresponding \emph{rescaled} locations $ \bk/T:=(k_1/T,\ldots,k_p/T) $ which form a regular grid of $ T^p $ points within the fixed set $ [0,1]^p $. Increasing $ T $ corresponds to a finer sampling within $ [0,1]^p $, and asymptotics via $ T \rightarrow \infty $ entail scanning on an arbitrarily fine grid.\\
	
	The goal is to test the null hypothesis of no anomaly
	\begin{align*}
		H_0:  \mu_{\bk,T}=\mu_{0}  \quad \text{for all }  \bk
	\end{align*}
	for some $\mu_0\in\R$
	against the alternative of an anomaly (in expectation)
	\begin{align*}
		H_1:\quad &\text{There exists } \anomaly\subset[0,1]^p \text{ with Lebesgue measure }0<\lambda(\anomaly)<1 \\
		&\text{and }  \mu_{\bk,T}\begin{cases}
			=\mu_0, &\text{ for } \bk/T\not\in \anomaly,\\
			\ne\mu_{0}, & \text{ for } \bk/T\in \anomaly.
		\end{cases}
	\end{align*}
	
	Here, $ \anomaly\subset[0,1]^p $ denotes an anomaly, that is, an area within the sample space $ [0,1]^p $ with different expectation than the default value $ \mu_0 $. These hypotheses are plausible in the specific context of anomaly detection in concrete, since in practice, it can be observed that anomalies like cracks and natural anomalies such as gravel particles, air bubbles etc.~have different average gray values than areas without anomalies. In this context, Assumption~\ref{ass_errors} on the errors is plausible as well, since it can be observed that the error of a pixel usually is correlated with the errors of pixels in a small neighboring environment.
	
	To this end, we consider scan windows $ \scanset\subset[0,1]^p $ with $ 0<\lambda(\scanset)<1$, where $\lambda$ denotes the Lebesgue measure, and define the corresponding scan statistics as local means taken over all elements of the random field that fall within the shifted scan window $ \scanset\left(\floor{\bs}_T\right) $ with anchoring point $\bs$; these are determined by the local sums
	\begin{align} \label{eq_def_sums}
		&\danobisum_{\scanset}\left(\floor{\bs}_T\right)=	S_{\scanset}\left(Y;\floor{\bs}_T\right)= \sum_{\frac{\bk}{T}\in \scanset\left(\floor{\bs}_T\right)}Y_{\bk,T},\\
		&\text{where }\floor{\bs}_T= \left(\frac{\floor{s_1T}}{T},\ldots,\frac{\floor{s_pT}}{T}\right)^{\prime}, \qquad 	\frac{\bk}{T}=\ \left(\frac{k_1}{T},\ldots,\frac{k_p}{T}\right)^{\prime},\notag\\
		&\phantom{\text{where }}	\scanset(\bs)= \scanset+\bs=\geschweift{\bx\in\R^p| \bx-\bs\in \scanset}\qquad \text{for }\bs\in[0,1]^p.\notag
	\end{align}
	
	In the following,
	we drop the dependency on $ T $ for ease of notation except in situations where it helps to clarify the argument.

Previous	literature related to this problem  focuses on oriented hypercubes as scan windows. Indeed, for scan statistics based on $ p $-dimensional hyperrectangles, there exists some theory on the distribution of scan statistics:  \cite{Kabluchko}, \cite{AriasCastroDonoho} and \cite{SharpnackAriasCastro} analyze various multiscale procedures with sublinear bandwidths for i.i.d.~Gaussian noise. \cite{HaimanPreda} analyze the behavior for positive integer-valued noise while \cite{Jaruskova}, \cite{Zemlys} assume independent centered noise with existing variance. \cite{Bucchia} studies a CUSUM-type scan statistic under the assumption of an existing functional central limit theorem on the noise.
	
	However,  using only oriented hypercubes as scan windows is not sufficient for many applications. Instead, combinations of more general window shapes, as discussed in this paper, are necessary for a good practical performance. In the next section, we will illustrate this using 2D-slices of concrete data with cracks as motivating examples.

	\subsection{Motivating example for the use of scan statistics based on the mean}\label{danobi_motivation}
	
	\begin{figure}
		\includegraphics*[width=0.32\textwidth]{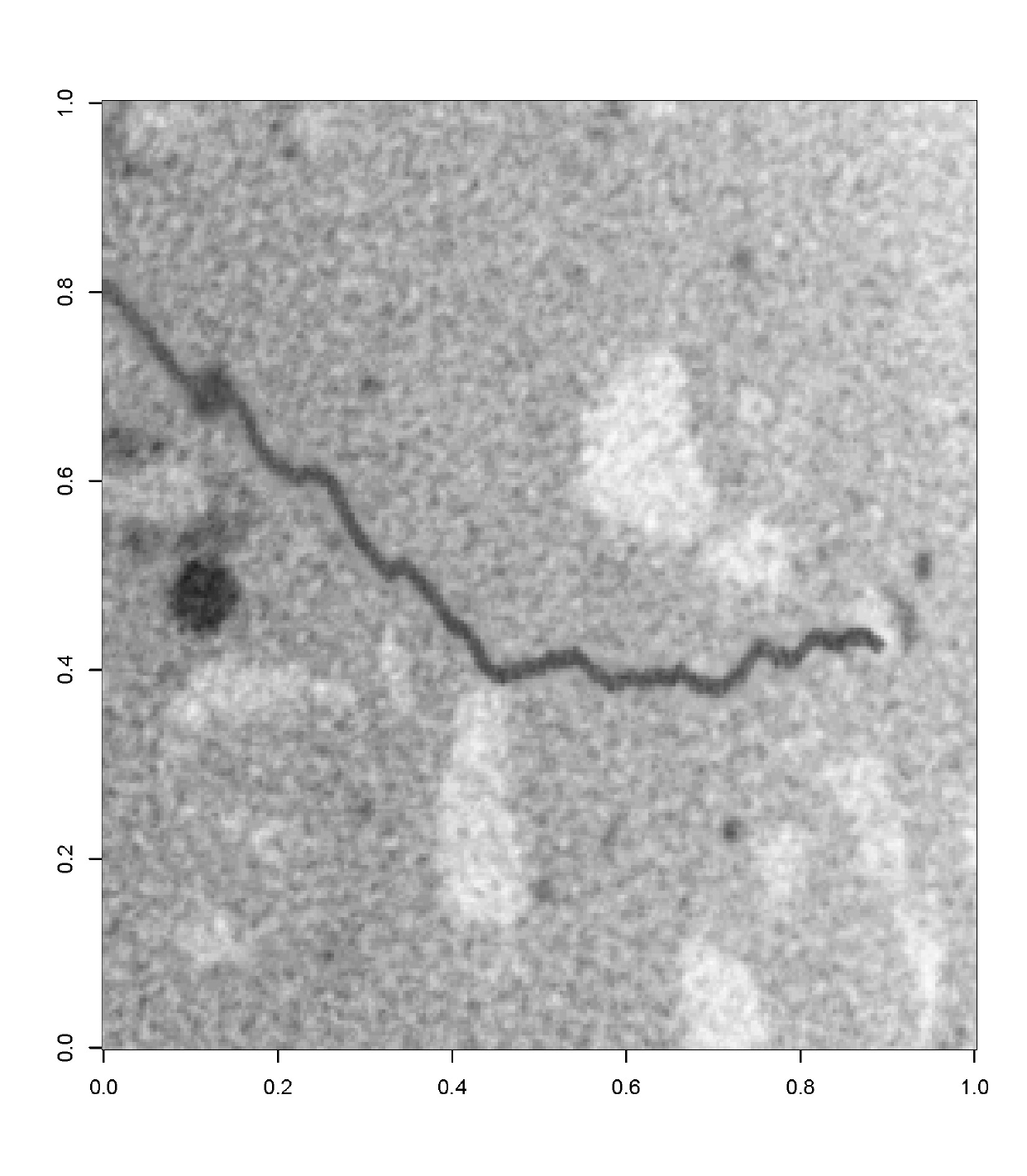}
		\includegraphics*[width=0.32\textwidth]{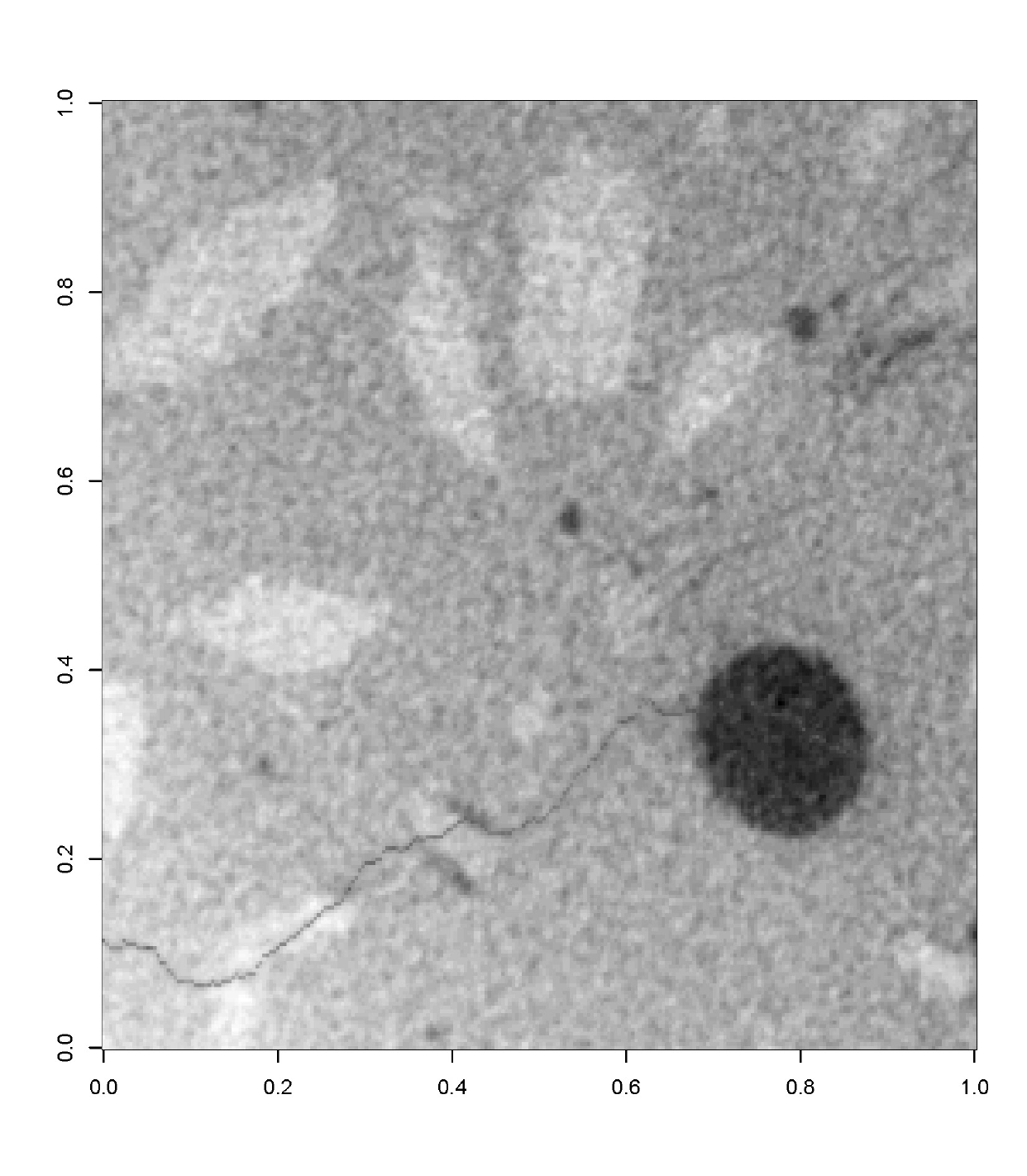}
		\includegraphics*[width=0.32\textwidth]{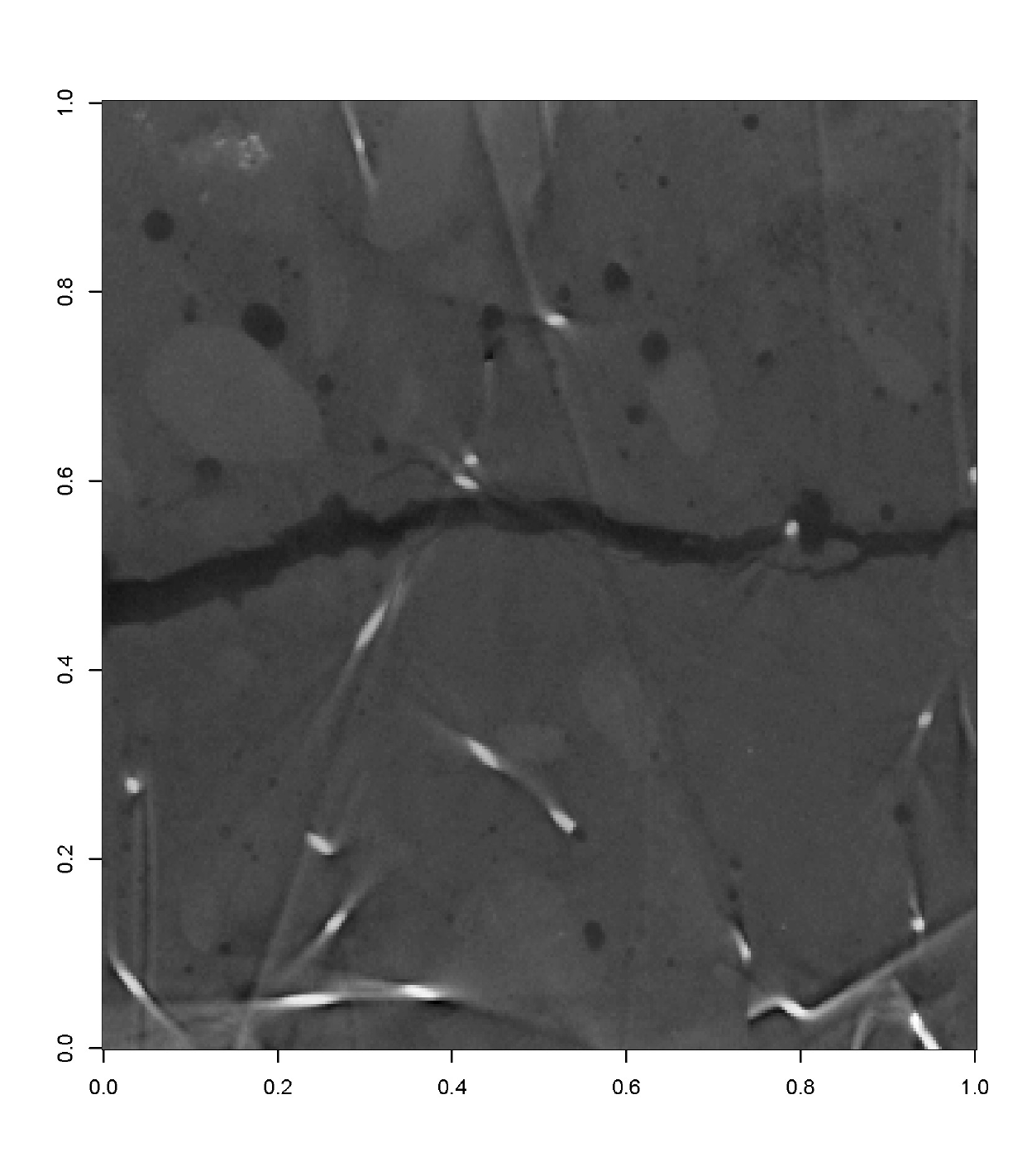}
		\caption[Motivational example based on concrete with an artificial fissure for Part \ref{Part_2}]{The left panel and the middle panel show 2D slices of concrete blocks that contain artificial fissures of width $ 5 $ and $ 1 $ pixels, respectively. The right panel shows a 2D slice of steel fiber-reinforced concrete with an actual (not artificial) fissure.
		}
		\label{figure_motivation}	 	
	\end{figure}
	
	Anomalies are frequently visible as areas with different gray values (and thus, different expected values in \eqref{eq_model}) compared to the background noise. The geometric properties of these areas typically depend on the type of anomaly. Such geometric properties can be taken into account by  combinations and contrasts of sample means over differently-shaped local windows.

	As an illustrative example consider (semi-artificial) 2D image data of concrete as in Figure~\ref{figure_motivation}. For the first two images, an artificial crack of constant width, which is generated by a fractional Brownian motion, is added to a 3D image of uncracked concrete. This was done by Franziska M\"usebeck (TU Kaiserslautern). The last image contains an actual crack that occurred after a bending test. We have changed the gray scale for better visibility as compared to the original CT-image which was mainly black. The images contain both potentially dangerous anomalies such as the cracks that we aim to detect, as well as natural anomalies which are an integral part of the material (obtained by various aggregates such as gravel or sand and often contain small pores of air). In all cases, anomalies can be modeled as  areas with different expected values in the image.

\begin{figure}[b]
		\includegraphics*[width=0.24\textwidth]{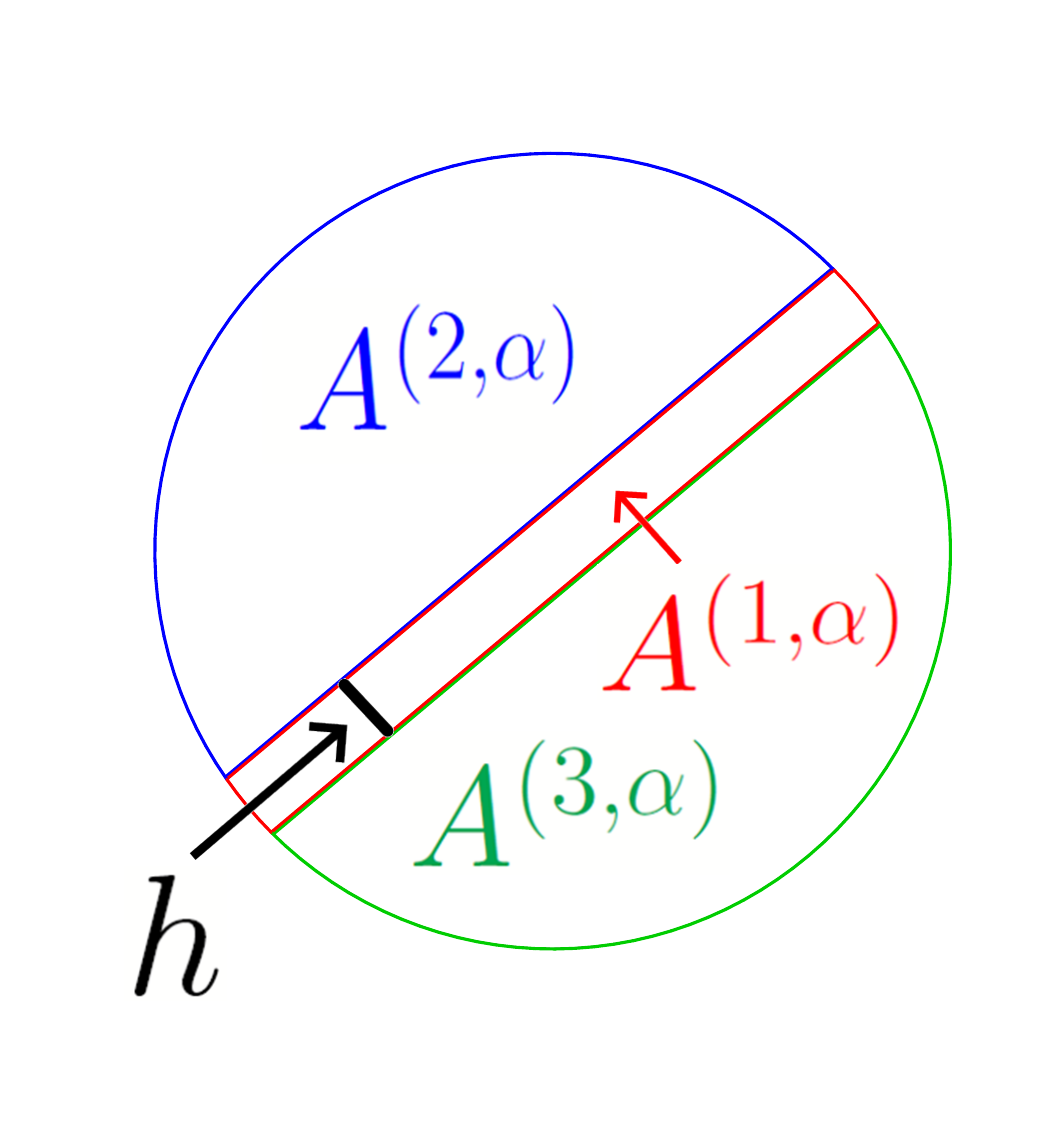}
		\includegraphics*[width=0.37\textwidth]{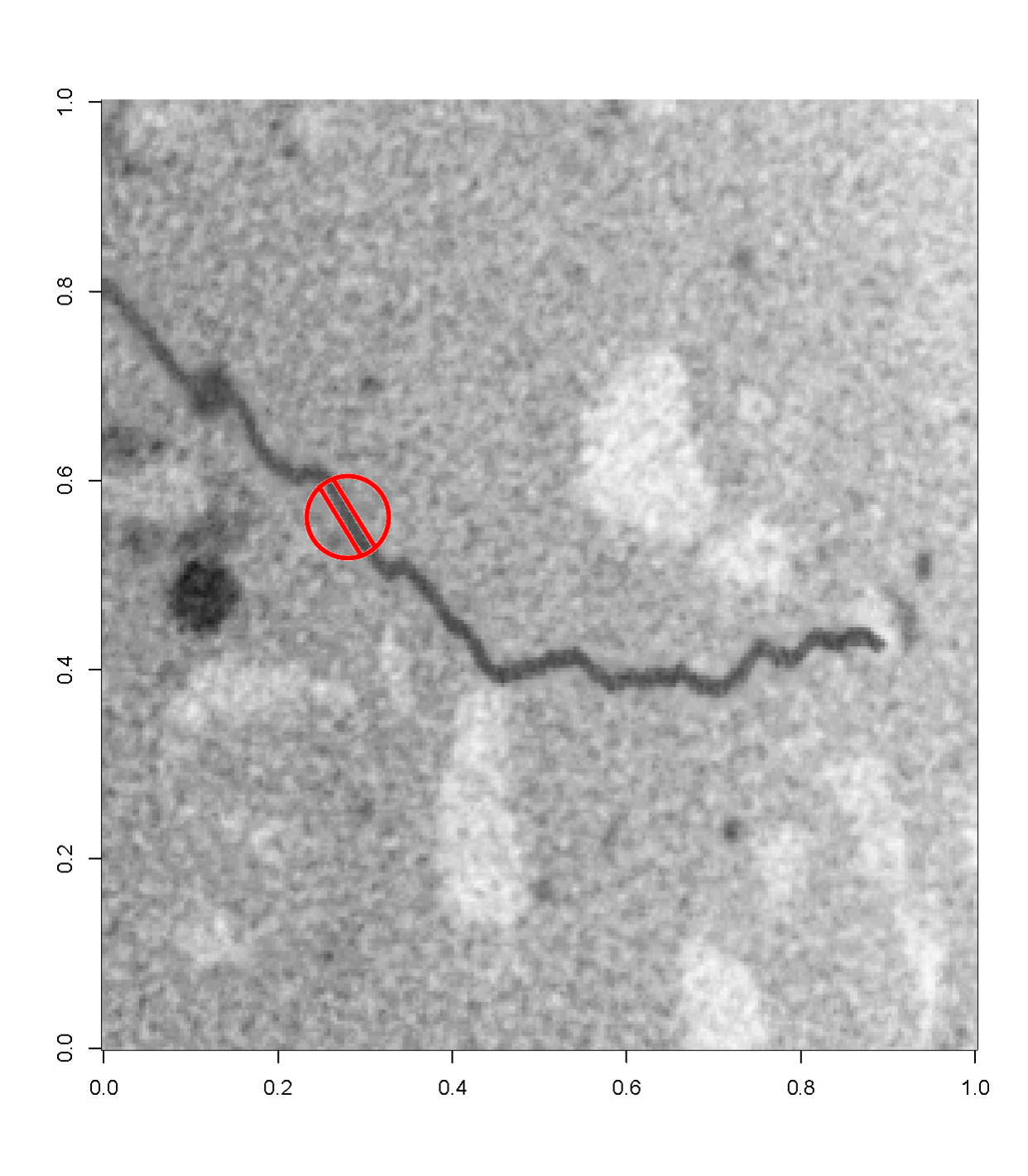}
		\includegraphics*[width=0.37\textwidth]{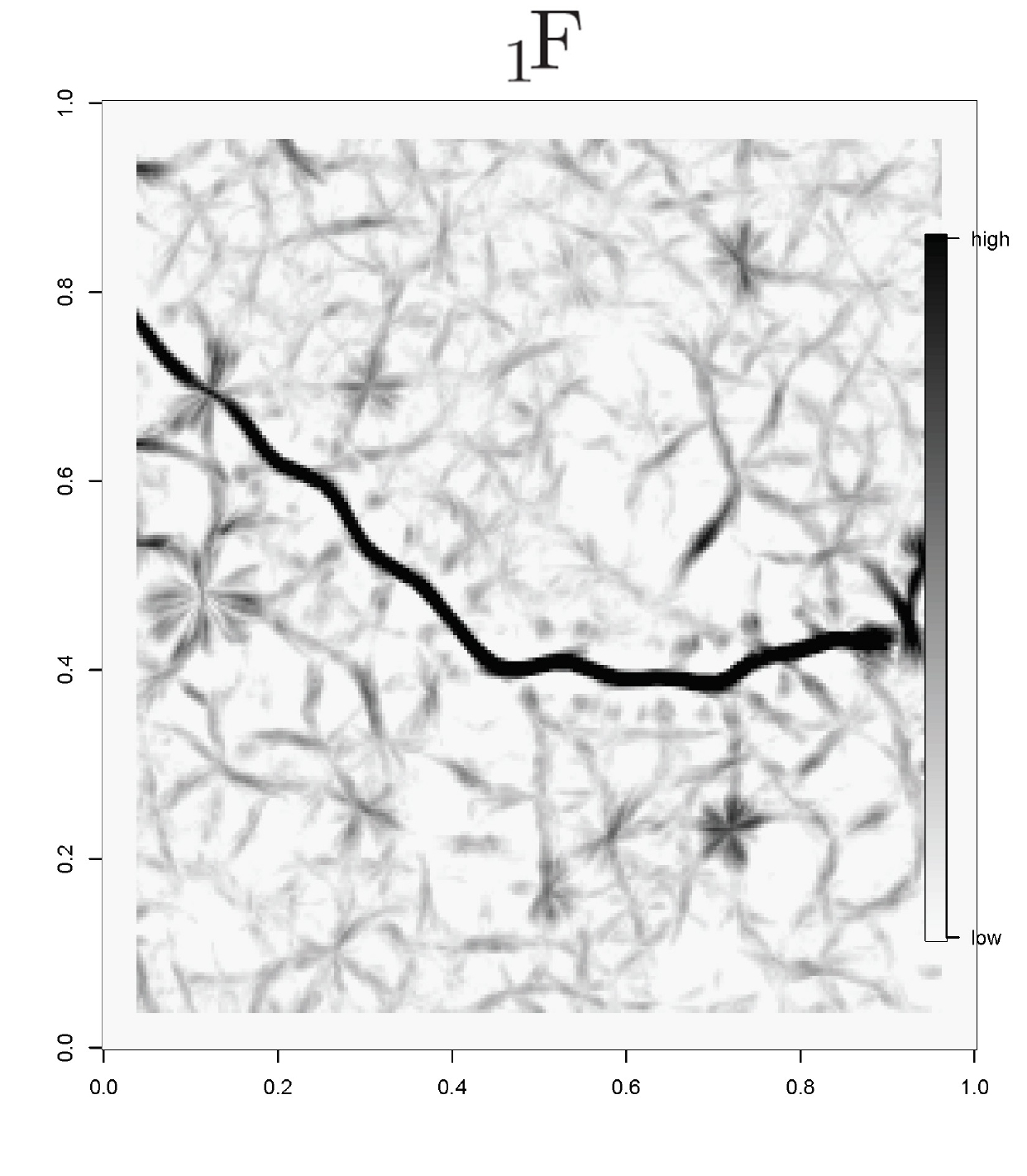}
		\caption[Geometric properties of fissures and bubbles together with scan statistics.]{
			The left panel shows the scan window used by the scan statistic $ \F1 $ in \eqref{f1}.
			The middle panel illustrates the local geometric properties of fissures.
			In the right panel, a heat map of  $ \F1 $ is displayed, showing that in this example, \eqref{f1} enhances the fissure.
		}\label{figure_danobi_basesetting}
	\end{figure}

	On an abstract level cracks, which we will henceforth call fissures, are lines with small width and potentially changing direction such that, locally, they resemble rectangles with small width. We illustrate this fact   in the  middle panel of Figure \ref{figure_danobi_basesetting} using the left-most image of Figure~\ref{figure_motivation}. Furthermore, the gray values of the fissures are lower than the gray values of their neighboring environment. Consequently, the mean over the inner segment $\scanset^{(1,\alpha)}$ of a split circle as depicted in the left-most panel of Figure~\ref{figure_danobi_basesetting} is smaller than the mean over each of the outer segments $\scanset^{(j,\alpha)}$, $j=2,3$, if the circle lies properly (and correctly oriented) above a fissure. This is in contrast to what happens away from such anomalies, since then the aforementioned means produce similar values. This shows the need for combining different scan windows with different orientation and motivates the use of the maximum over different angles of such contrasts as scan statistics, where we refer to Section \ref{sec_danobi_statistics} for the mathematical details. The right panel of Figure~\ref{figure_danobi_basesetting} illustrates that such a scan statistic does indeed enhance the fissures.
 
	In conclusion, the statistical theory needs to be flexible enough to (a) allow for the use of multiple scan windows (different angles, different shapes) and (b) include the option to combine the information from several scan windows (contrasting, maximizing).

	\section{Asymptotic results}

	In this section we develop asymptotic theory for a large class of local sums of stationary random fields. As a first result we obtain a general central limit theorem for sums over Jordan-measurable sets of $M$-dependent random variables. Our main result delivers a functional central limit theorem (FCLT) for combinations (via uniformly continuous functions) of a finite number of  local sums over  sets that fulfill slightly stronger assumptions. As a main application of this theory we then derive null asymptotics for scan statistics such as the ones  introduced in Sections \ref{sec_stat_mod} and \ref{danobi_motivation}, from which thresholds can be derived that go along with statistical guarantees. However, our theory is much more general than this. As we will point out in Remark \ref{rem_rectangles}, previous theory on local sums of random fields is restricted to correctly oriented (that is, parallel to a fixed set of coordinate axes) hypercubes. Since our results hold for a much larger class of shapes, we present the asymptotic theory in this subsection in a very general way, before looking at specific examples of scan statistics in Section \ref{sec_danobi_statistics}.

 \subsection{Universal asymptotic theory for combined local sums} \label{danobi_theory}
	We start with a general central limit theorem over bounded Jordan-measurable sets in random fields.
	A bounded set $A$ is called Jordan-measurable if and only if the function $\mathds{1}_A$ is Riemann-integrable. This allows for an arbitrarily close approximation of the sums  by sums over finite unions of rectangles. For a detailed treatment of Jordan-measurability we refer to ~\cite{DuistermattKolk}, Chapter 6. In particular, bounded convex sets $A$ with $0<\lambda(A)<\infty$ are Jordan-measurable, compare Lemma \ref{convex_jordan} in the appendix. In the following, $\lambda(\cdot)$ will always be the Lebesgue measure as all sets considered are Lebesgue measurable. In particular, Jordan-measurable sets are Lebesgue measurable, with the Jordan and the Lebesgue measure coinciding (compare e.g.\ \cite{Salamon}, Theorem 2.24).
	
	\begin{Theorem}
		\label{clt_m_dependence}\ \\
		Let model \eqref{eq_model} be given under the null hypothesis with $\mu_{\bk,T}=0$ for all $ \bk,T $, and  errors fulfilling Assumption~\ref{ass_errors}.
		%	with $ \EW{\epsilon_{\bk}}=0 $, $ \EW{\epsilon_{\bk}^2}\in(0,\infty) $ and $ \EW{\abs{\epsilon_{\bk}}^r}<\infty $ for some $ r>2p $. 
		Let $ A\subset[0,1]^p $ be a compact Jordan-measurable set with $\lambda(A)>0$ and let $ \sigma^2 $ be the long-run variance defined in \eqref{long_run_variance}. Then it holds that
		\[ 
		\frac{1}{T^{p/2}}S_A\dto\mathcal{N}(0,\sigma^2\lambda(A))
		\]
		as $ T\to\infty $ with $S_A=S_A(Y;0)$ as in \eqref{eq_def_sums}.
	\end{Theorem}

	We impose the following assumption on the sets that define the local sums in our FCLT. This assumption is stronger than Jordan-measurability (compare Lemma \ref{ass21_implies_jordan} in the appendix) and required to prove a modulus of continuity (see Lemma \ref{modulus_of_continuity} in the appendix) while the fidis-convergence follows from Theorem~\ref{clt_m_dependence}.
	
	\begin{Assumption}\label{ass_sets}
		Let $A\subset[0,1]^p$ be a compact set such that there exists $K\in\N$ and Lipschitz continuous functions $f_1,\ldots,f_K: [0,1]^{p-1}\to\R^p$ such that
		\[ 
		\partial A\subset\bigcup\limits_{i=1}^K f_i\left([0,1]^{p-1}\right),
		\]
		i.e.\ the boundary of $A$ is a subset of the image of finitely many Lipschitz continuous functions from $[0,1]^{p-1}$ to $\R^p$.
	\end{Assumption}

	The assumption is fulfilled for all bounded convex sets and in particular for the scan windows as discussed in Section~\ref{danobi_motivation} including rectangles, circles and circle segments, as the following Lemma shows:

	\begin{Lemma}(Compare e.g.\ \cite{Widmer}, Lemma 3.1) \label{convex_sets}
		Convex sets $ \scanset\subset[0,1]^p $ with $ 0<\lambda(\scanset)<\infty $  fulfill Assumption \ref{ass_sets}.
	\end{Lemma}
	
	Furthermore, as motivated in Section~\ref{danobi_motivation}, it is desirable to be able to consider statistics that combine information obtained from several scanning windows (e.g.\ for the purpose of local contrasting). As a mathematical framework we assume that this combination occurs via uniformly continuous functions.
	
	As a main result, which is applicable beyond the specific application discussed in Section~\ref{danobi_motivation}, allowing for a large range of local scan statistics based on varying scan windows, we derive the following functional central limit theorem (over the rescaled image):

	\begin{Theorem} \label{danobi_convergence_map}
		Let model \eqref{eq_model} be given under the null hypothesis with $\mu_{\bk,T}=0$, errors fulfilling Assumption~\ref{ass_errors}, and let $ \sigma^2>0 $ be the long-run variance as in \eqref{long_run_variance}.
		Let $ \scanset_1,\ldots,\scanset_P\subset[0,1]^p $ fulfill Assumption \ref{ass_sets} and  let $ \Hoelder:\R^P\to\R $ be a uniformly continuous function.\\
		Denote \[
		\Sprocess_T(\bs)=\left(\frac{1}{T^{p/2}}\danobisum_{\scanset_1}(\floor{\bs}_T),\ldots,\frac{1}{T^{p/2}}\danobisum_{\scanset_P}(\floor{\bs}_T)\right)^{\prime},
		\]
		where $ \danobisum_{\scanset}(\floor{\bs}_T)$ is as in \eqref{eq_def_sums}.
		
		Then there exists a $ P $-dimensional centered Gaussian process \\ $ (Z(\bs))_{\bs\in[0,1]^p}=((Z_1(\bs),\ldots,Z_{P}(\bs))^{\prime})_{\bs\in[0,1]^p} $ with
		\[
		\Cov{Z_i(\bs)}{Z_j(\bt)}=\sigma^2\lambda\left(\scanset_i(\bs)\cap\scanset_j(\bt)\right)
		\]
		such that
		\[
		\left(\Hoelder\left(\Sprocess_T(\bs)\right)\right)_{\bs\in[0,1]^p}\wto \left(\Hoelder\left(Z(\bs)\right)\right)_{\bs\in[0,1]^p}
		\]
		on $ \cD([0,1]^p)  $ (as in Definition \ref{Wichura_def} in the appendix) equipped with the topology induced by the maximum norm on $ \mathcal{D}([0,1]^p) $ (called \emph{U-topology} by \cite{Wichura}).
	\end{Theorem}
	The proof of Theorem \ref{danobi_convergence_map} is quite involved requiring tools from real analysis, convex geometry and probability theory. It is postponed to Section \ref{sec_proofs} in the appendix and requires auxiliary results developed in Sections \ref{sec_limit_thms} - \ref{sec_algebraic_properties}. In Section \ref{sec_danobi_statistics}, we illustrate how to choose both $G$ and $A_1,\ldots,A_P$ for a given application. The corresponding limit results are then a mere corollary of the above general result (see Section \ref{sec_size_power} in the appendix) providing the tools necessary for statistical tests. Additionally, we prove that these tests have asymptotic power one in a variety of situations including misspecified situations (see Theorem  \ref{power_one} and Remark \ref{rem_power_one} as well as Theorems \ref{power_one_fnb2} and \ref{power_one_fnb1} in the appendix).    
	\begin{Remark}\label{domain_fclt}
		For ease of notation, we consider all anchoring points $\bs\in[0,1]^p$ 	in Theorem~\ref{danobi_convergence_map}, which requires observations  $\geschweift{Y_{\bk,T}}$ in a slightly larger hypercube (depending on the sets $A_1,\ldots,A_P$).
		On the other hand, in the practical application in Section \ref{sec_danobi_statistics}, we use a slightly different scaling, which is more convenient for implementation purposes: We assume that we observe $\geschweift{Y_{\bk,T}}$ only on the grid $\bk/T\in[0,1]^p$ with the scan statistics $S_{A_i}(\bs)$ calculated only for anchoring points $\bs\in [d/2,1-d/2]^p$ (see also Section  \ref{sec_danobi_statistics}), so that all grid points within the shifted scan windows used in the application lie within $[0,1]^p$. Nevertheless, this is inline with the theory above by an appropriate rescaling of the grid and scan windows.
	\end{Remark}
	
	\begin{Remark}\label{rem_rectangles}
		Previous literature concerned with functional central limit theorems for random fields focuses on the case of oriented hypercubes (see e.g.\ \cite{Kabluchko}, \cite{AriasCastroDonoho}, \cite{SharpnackAriasCastro}, \cite{Bucchia}, \cite{HaimanPreda}, \cite{Jaruskova}, \cite{Zemlys}), i.e.\ they obtained the following functional central limit theorem
		\begin{align}
			\left(\frac{1}{T^{p/2}}\sum_{k_1=1}^{\floor{s_1T}}\ldots\sum_{k_p=1}^{\floor{s_pT}}\epsilon_{k_1,\ldots,k_p}\right)_{\bs\in[0,1]^p}\wto\sigma\left(W_{\bs}\right)_{\bs\in[0,1]^p} \label{FCLT}
		\end{align}	
		on $ \cD\left([0,1]^p\right) $, for some $ p $-parameter Wiener process $ \left(W_{\bs}\right)_{\bs\in[0,1]^p} $ (compare Definition \ref{Wiener_ppar} in the appendix) and some $ \sigma>0 $.
		
		In such a case, and if additionally it holds for any bounded hyperrectangle $ I\subset\R^p $ and any $ (c_1,\ldots,c_p)^{\prime}\in\N_0^p $  that
		\begin{align}
			\frac{1}{T^{p/2}}\ \supp{\bs\in I}\abs{\sum_{k_1=1}^{\floor{s_1T}}\ldots\sum_{k_p=1}^{\floor{s_pT}}\epsilon_{k_1,\ldots,k_p}-\sum_{k_1=1}^{\floor{s_1T}+c_1}\ldots\sum_{k_p=1}^{\floor{s_pT}+c_p}\epsilon_{k_1,\ldots,k_p}}=o_P(1), \label{FCLT2}
		\end{align}
		then the limit process of Theorem~\ref{danobi_convergence_map} for (correctly oriented) rectangular scanning windows $ \scanset=\linksoffen{-a_1,a_1}\times\ldots\times\linksoffen{-a_p,a_p} $, $a_j\in (0,1/2)$ for $j=1,\ldots,p$, can be written obtained explicitly in terms of the above limit  process $ \left(W_{\bs}\right)_{\bs\in[0,1]^p} $. 
To elaborate, it holds,
		\begin{align*}
			\left(\frac{1}{T^{p/2}}\danobisum_{\scanset}\left(\epsilon;\floor{\bs}_T\right)\right)_{\bs\in[0,1]^p}\wto\left(\Lambda_{\bs}\right)_{\bs\in[0,1]^p}
		\end{align*}	
		on $ \cD\left([0,1]^p\right) $ as $ T\to\infty $, where (with $ (-1)^{\bs}:=\big((-1)^{s_1},\ldots,(-1)^{s_p}\big) $ and $\bs\odot\bt=(s_1t_1,\ldots,s_pt_p)^{\prime} $ for $\bs=(s_1,\ldots,s_p)^{\prime}$, $\bt=(t_1,\ldots,t_p)^{\prime}$)
		\begin{align*}
			\Lambda_{\bs}=\sigma\left(\sum_{\bd=(d_1,\ldots,d_p)^{\prime}\in\geschweift{0,1}^p} (-1)^{\sum_i d_i} W_{\bs+(-1)^{\bd}\odot\ba}\right), 
		\end{align*}
  
		However, even the case of rotated rectangular windows is not covered by this case, such that the joint asymptotic for several rotated rectangular windows cannot be obtained from a functional central limit theorem as in \eqref{FCLT}.
	\end{Remark}

	\subsection{Scan statistic for the detection of fissures in concrete}\label{sec_danobi_statistics}
 While the results of the previous section provide a flexible framework for the detection of anomalies based on scan statistics, in this section we illustrate the choice of window shapes and how to combine them for the motivating examples in Section~\ref{danobi_motivation}.
	Here, the main goal is the identification of areas that potentially contain fissures (without too many false alarms due to natural anomalies). These image data are in 2D (to be precise, 2D slices of 3D images). Therefore we will restrict ourselves to the case of $ p=2 $ in this section, while the previous theory is valid in arbitrary (but fixed) dimension.
	\begin{figure}
		\begin{center}
			\includegraphics*[width=0.25\textwidth]{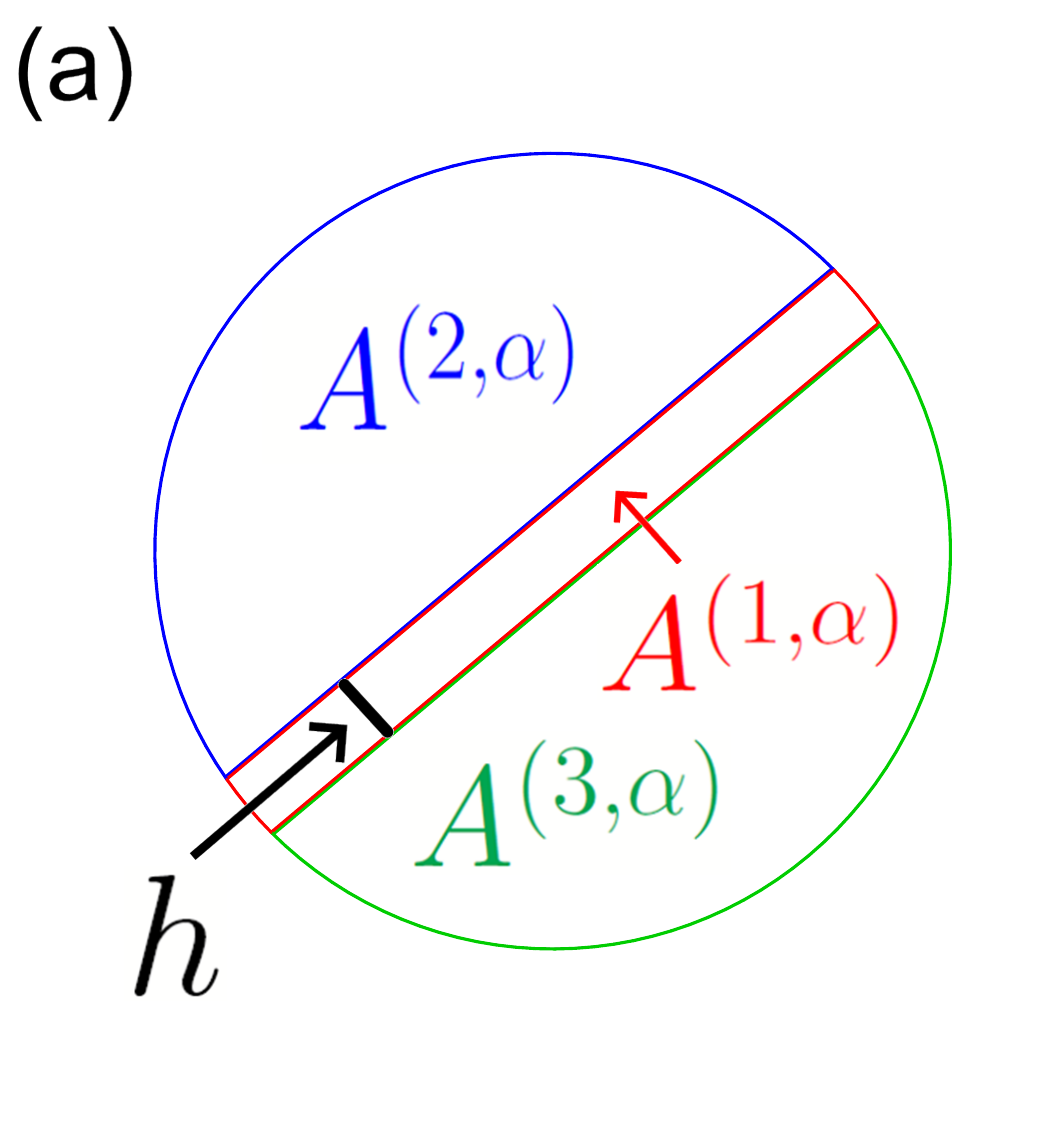}
			\includegraphics*[width=0.25\textwidth]{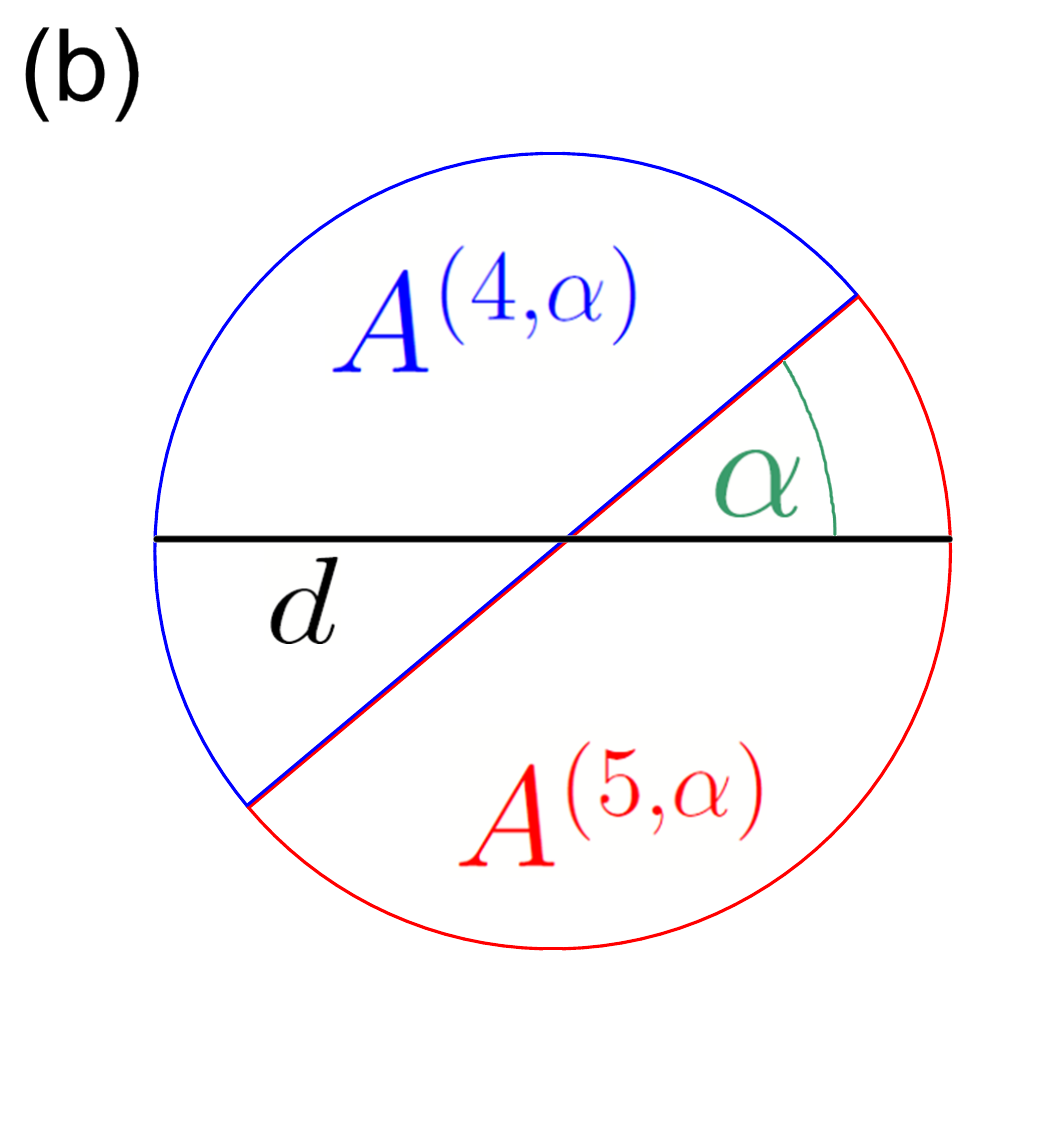}
		\end{center}
		\caption{Panels (a)-(b) illustrate the sets that we use in our scan statistics introduced in \eqref{f1} and \eqref{nB}, respectively. }
		\label{figure_danobi_method}
	\end{figure}
	
	\subsubsection{$\F1$-statistic}\label{sec_f1}
	As already indicated such an anomaly detection can be achieved by local  contrasts between suitable circle segments.
	To elaborate, let $ \scanset $ be a circle of diameter $ d\in(0,1) $ centered around 0. For some angle $ \alpha\in\rechtsoffen{0,\pi} $, we split this circle by having a "strip" of width $ h\in(0,d) $ turned by $ \alpha $ (see Figure \ref{figure_danobi_method} (a)). We denote this inner "strip" by $ \scanset^{(1,\alpha)} $ and the two remaining segments by $ \scanset^{(2,\alpha)} $ and $ \scanset^{(3,\alpha)} $, respectively.  Denote by $ (Y_{\bk,T})_{\bk\in\geschweift{1,\ldots,T}^2} $ the process of gray values and let $ \danobisum_{\scanset}\left(\floor{\bs}_T\right)= \danobisum_{\scanset}\left(Y;\floor{\bs}_T\right) $ be as in \eqref{eq_def_sums}. The sets $\scanset^{(j,\alpha)}$, $j=4,5$, will play a role later in Sections~\ref{sec_fnb2} and \ref{sec_fnb1} and are already included here for completeness.

	While the sums over the scanning windows in the universal result as given in Theorem~\ref{danobi_convergence_map} are implicitly normalized with their Lebesgue measure, in practice, we want to use sample means over the scanning windows, such that we normalize with the actual number of points within the scanning windows. As can be seen in Lemma~\ref{lem_order} in the appendix, both approaches are asymptotically equivalent.
	To elaborate, define
	for $ i=1,\ldots,5 $, $ s\in[d/2,1-d/2]^2 $ (recall that $ d $ denotes the diameter of the scan window)
	\begin{align*}
		\danobimean_{\scanset^{(i,\alpha)}}\left(\floor{\bs}_T\right)=\danobimean_{\scanset^{(i,\alpha)}}\left(Y;\floor{\bs}_T\right)=&\ \frac{T}{\abs{\geschweift{\frac{\bk}{T}\in \scanset^{(i,\alpha)}\left(\floor{\bs}_T\right)}}} \danobisum_{\scanset^{(i,\alpha)}}\left(\floor{\bs}_T\right).
	\end{align*}

	We will now give precise formulas for the proposed scan statistic as  motivated in Section~\ref{danobi_motivation}:
	Since a fissure can change directions, we need to consider multiple angles when trying to detect the fissure. As $ \scanset^{(1,\alpha)}=\scanset^{(1,\alpha+k \pi)} $, $  \scanset^{(2,\alpha)}=\scanset^{(3,\alpha+(2k+1) \pi)} $ for all $  k\in \N $, we only need to consider angles in $ \rechtsoffen{0,\pi} $ in the following statistic. Therefore, for $ T>0 $, $ 0\le\alpha_1<\ldots<\alpha_P<\pi $, $ \bs\in[d/2,1-d/2]^2 $, define the one-sided  (\textbf{F}issures-) $ \F1 -$statistic by
	\begin{align}
		& \F1_T(\bs)=\maxx{i=1,\ldots,P} \F1_T(\bs,\alpha_i),  \label{f1}\\
		&\text{where }\F1_T(\bs,\alpha_i)=\  \frac{1}{\sigma}\min\geschweift{\danobimean^{(12,\alpha_i)}(\bs),\danobimean^{(13,\alpha_i)}(\bs)}, \label{f1alpha}\\
		&\phantom{\text{where }}	\danobimean^{(12,\alpha)}(\bs)=\ \danobimean_{\scanset^{(2,\alpha)}}\left(\floor{\bs}_T\right)-\danobimean_{\scanset^{(1,\alpha)}}\left(\floor{\bs}_T\right),\notag\\
		&\phantom{\text{where }}	\danobimean^{(13,\alpha)}(\bs)=\ \danobimean_{\scanset^{(3,\alpha)}}\left(\floor{\bs}_T\right)-\danobimean_{\scanset^{(1,\alpha)}}\left(\floor{\bs}_T\right),\notag
	\end{align}
	where $ \sigma^2>0 $ is the long-run variance as in \eqref{long_run_variance}
	
	In $	\danobimean^{(12,\alpha)}(\bs)$ and $	\danobimean^{(13,\alpha)}(\bs)$ we  take the one-sided difference subtracting the (rescaled) average in the inner "strip" $ \scanset^{(1,\alpha)} $  from
	the (rescaled) averages in the two sections $ \scanset^{(2,\alpha)} $ and $ \scanset^{(3,\alpha)}$.  Both of these differences can be expected to be larger at a fissure because a fissure is known to have lower gray values than the surrounding pixels, thus
	taking the minimum of those two differences will increase the contrast.  The maximization in $\F1_T(\bs) $ over all considered angles $ \alpha_1,\ldots,\alpha_P $ is necessary to enhance fissures irrespective of their (possibly varying) direction.  Indeed, as illustrated in the right panel of Figure \ref{figure_danobi_basesetting},  the statistic $\F1_T(\bs)$ tends to enhance fissures while simultaneously discarding noise.

	\begin{Remark}\label{rem_variance_est} 
		In practice, $ \sigma $ has to be replaced by an estimator $ \hat{\sigma}_T $ that is asymptotically consistent under the null hypothesis (ideally without overestimating $\sigma$ too much under alternatives to avoid power loss). 	As will be discussed in Section \ref{danobi_simstudy}, a proper choice of $ \hat{\sigma}_T$ is necessary  to detect fissures with large probability  without having too many false positives. Real data typically contains larger proportions of bubbles with average gray values differing from the rest of the image as displayed in Figure \ref{figure_motivation}, which will significantly influence non-robust global variance estimators. Therefore, in the simulations we use a global robust estimator for $ \sigma $ (as given in \eqref{var_est}), which works very well under Gaussianity. Alternatively, local estimators $ (\hat{\sigma}_{\alpha_i,T}(\bs))_{\bs} $ for $ \sigma $ can be used, that may also depend on the angle $\alpha_i$, as long as they fulfill under the null hypothesis that
		\begin{align*}
			\supp{\bs\in[d/2,1-d/2]^2}\abs{\hat{\sigma}_{\alpha_i,T}(\bs)-\sigma}=o_P(1).
		\end{align*}
		The use of such estimators, however, comes at both the cost of computational and statistical efficiency -- the latter due to the smaller sample size it is based on for each grid point, compare Section 12.2 of \cite{Diss}.
	\end{Remark}
	
	As motivated in Section~\ref{danobi_motivation} we need a threshold $ \thr=\thr_{d,h,T,P} $ in order to determine areas that can potentially contain fissures. Based on such a threshold we can  distinguish between \emph{significant} points $ \bs\in[d/2,1-d/2]^2 $, for which\[
	\F1_T(\bs)\ge\thr_{d,h,T,P}
	\]
	and \emph{non-significant} points. The significant points are likely to contain a fissure in their environment such that they can then be used by post-processing procedures to track  the exact paths of the fissures (see e.g.\ \cite{KL}, \cite{KL2}, \cite{KL3}). Typically, significant points are not isolated pixels because a fissure influences the statistic at multiple points in a local environment.
	
	A common choice for the threshold is obtained as the asymptotic $ (1-\alpha) $-quantile of the maximum of the statistic defined in \eqref{f1}.
	Such a choice has the additional benefit that -- by construction --  it controls the family-wise error rate at level $\alpha$ of a pixel being significant (in the usual testing sense) under the null hypothesis of a stationary random field without anomalies.
	
	Since all functions involved in $\F1_T(\bs)$ are uniformly continuous,  a limit process for  $ \left(	\F1_T(\bs)\right)_{\bs\in[d/2,1-d/2]^2} $ can be obtained as a Corollary to Theorem~\ref{danobi_convergence_map}  in the case of a stationary random field with no anomalies, compare Theorem \ref{theorem_danobi_size} (a) in the appendix.
	
	Using a simulated $95\%$-quantile  from this theorem nicely recovers the fissures from the example at the left-hand side of Figure~\ref{figure_motivation}; see the middle panel of Figure~\ref{figure_threshold_w5} below.
	In the following theorem, we will show that  we obtain asymptotic power $ 1 $ if the signal is large enough for a correctly specified angle.
	
	\begin{Theorem}\label{power_one}
		Let $ \anomaly\subset[0,1]^2 $ be a rectangle of length $ l\ge d $, width $ w<h $ turned by degree $ \alpha_0\in\rechtsoffen{0,\pi} $ against the $ x $-axis and with center point $ \bs_0 $.
		Let $ \left(\epsilon_{\bk}\right)_{\bk\in\Z^2} $ be a sequence of random variables fulfilling Assumption \ref{ass_errors}. Let $ Y_{\bk,T}=\mu_{\bk,T}+\epsilon_{\bk} $ be as in \eqref{eq_model} with
		\begin{align}
			\mu_{\bk,T}=\mu_0-\begin{cases}
				0, &\text{ for } \bk/T\not\in \anomaly,\\
				\delta_T, & \text{ for } \bk/T\in \anomaly
			\end{cases}\label{signal}
		\end{align}
		for some bounded sequence $ \delta_T>0$. Let $ \F1_T(\bs) $ be as in \eqref{f1} and let angles $ 0\le\alpha_1<\ldots<\alpha_P $ be such that $ \alpha_i=\alpha_0 $ for some $ i $. Then,
		\[
		\F1_T(\bs_0)\pto \infty, \qquad \text{hence also }\maxx{\bs\in[d/2,1-d/2]^2} \F1_T(\bs)\pto \infty,
		\]
		if $ \delta_TT\to\infty $ as $ T\to\infty $.
	\end{Theorem}
	\begin{figure}
		\includegraphics[width=0.8\textwidth]{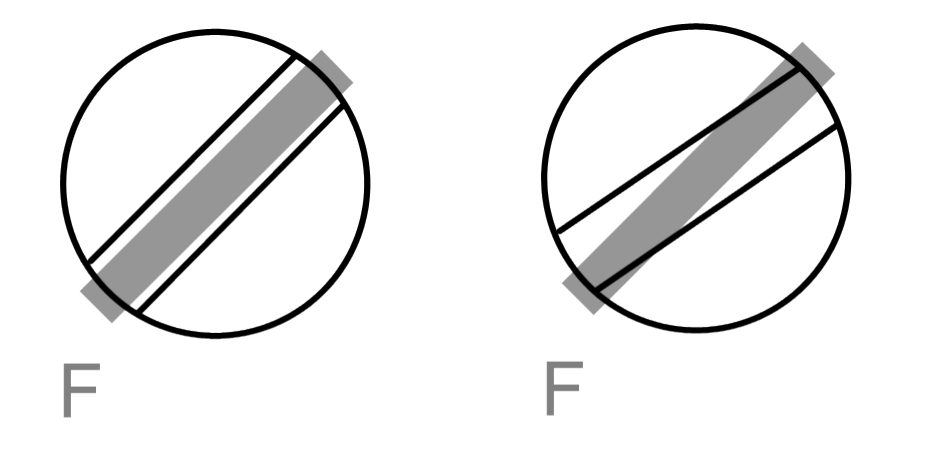}
		\caption{In this figure, the intersection between a rectangular-shaped fissure and the scan window from \eqref{f1} is illustrated. On the left-hand side of the figure, the angle of the fissure is correctly specified by the "inner strip", therefore displaying the situation of Theorem \ref{power_one}. On the right-hand side, the angle is slightly misspecified. Nevertheless, the mean of the inner circle segment is dominated by the gray area, while the mean of the outer segments is dominated by the white area.} \label{figure_power1}
	\end{figure}
	
	\begin{Remark} \label{rem_power_one}
		Due to the complicated shapes of the intersection between a rectangular fissure and the scan sets $ \scanset^{(1,\alpha_i)}  -  \scanset^{(3,\alpha_i)} $, it is non-trivial to give an explicit form for $ \lambda(\scanset^{(1,\alpha_i)}\cap \anomaly) - \lambda(\scanset^{(3,\alpha_i)}\cap \anomaly)  $ in the case of a misspecified angle. If such explicit forms are available, analogous arguments as in Theorem~\ref{power_one} give a corresponding  asymptotic power result. 
  See Figure \ref{figure_power1} for an illustration.
	\end{Remark}

	However, as can be observed in Figure \ref{figure_motivation}, fissures are not the only anomalies that can occur in concrete. By construction, concrete is a heterogeneous material that  can contain various different aggregates such as gravel, air pores, steel fibers etc.\ depending on the type of concrete.  As those natural anomalies are by no means dangerous but actually help increase the stability of the material, they should not result in (too many) significant points, i.e.\ the signal from these natural anomalies needs to be discarded. This can be done explicitly by making use of the geometric nature of these different types of natural anomalies.
	Furthermore, in other examples it might not be a priori clear that the anomaly of interest has lower gray values such that two-sided versions of the statistic are preferable.
	In the following, we present two statistics that aim to enhance fissures and discard these natural anomalies at the same time.
	
	\subsubsection{$ \FnB{2} $-statistic}\label{sec_fnb2}
	
	\begin{figure}
		\includegraphics*[width=0.32\textwidth]{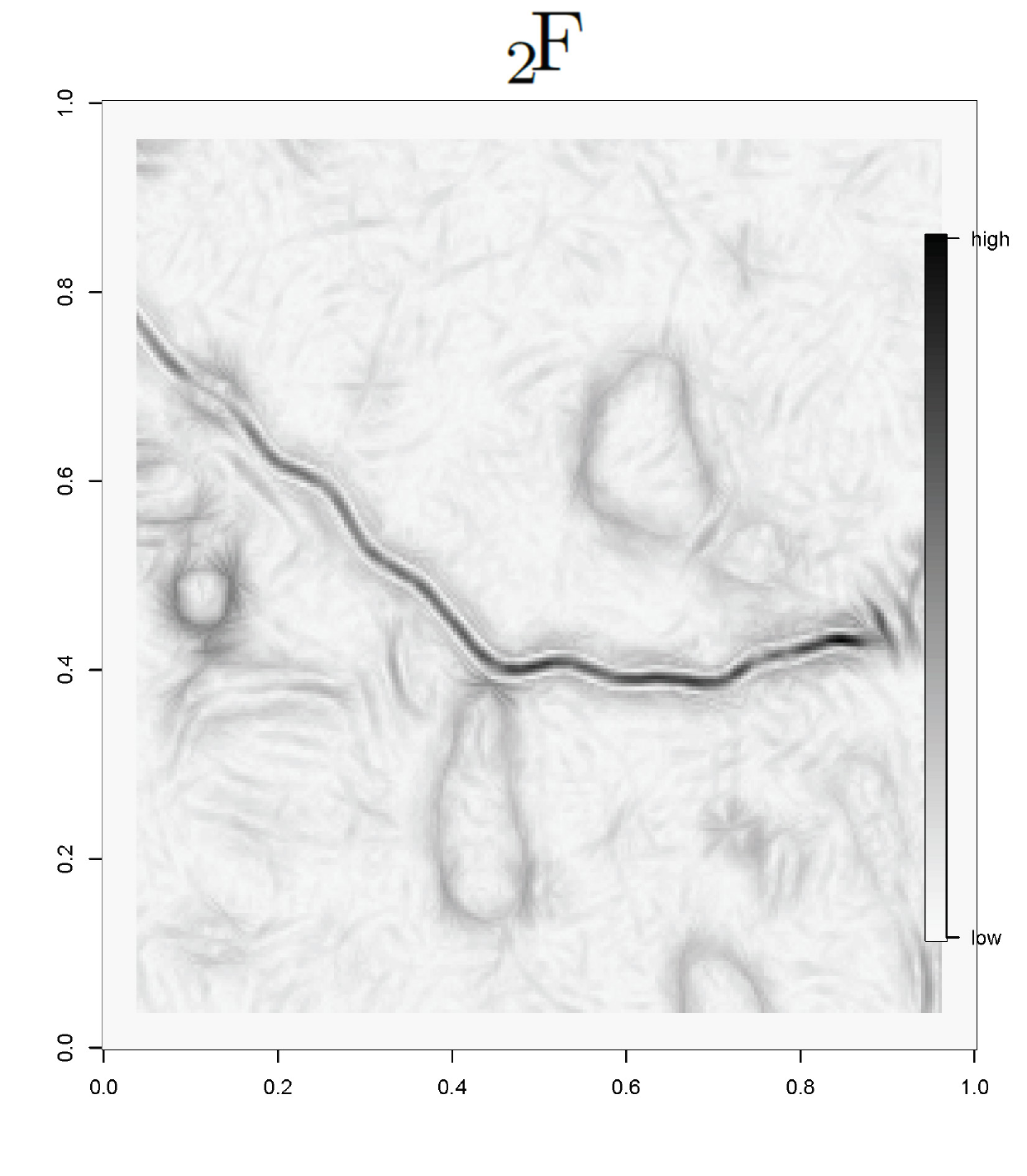}
		\includegraphics*[width=0.32\textwidth]{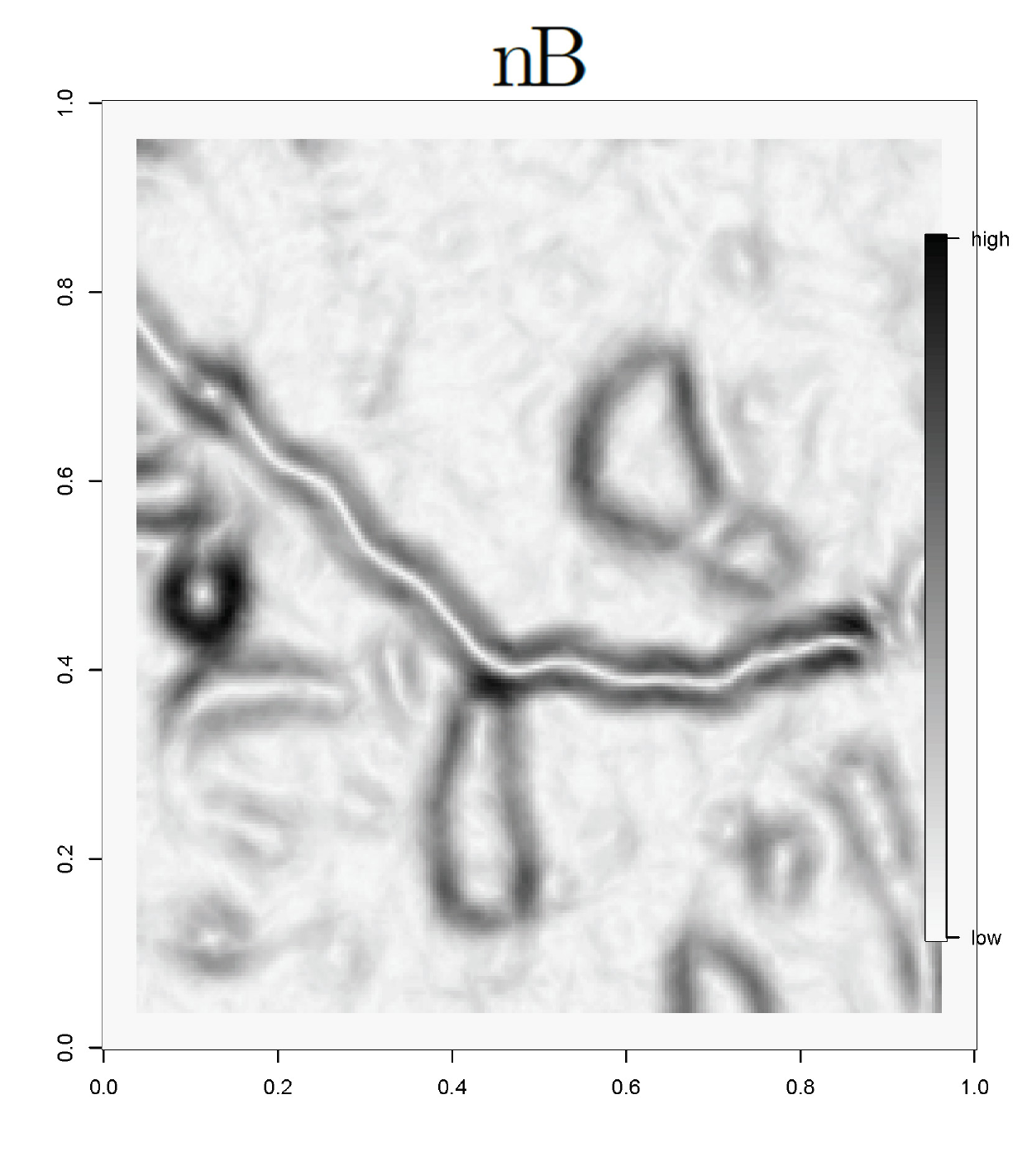}
		\includegraphics*[width=0.32\textwidth]{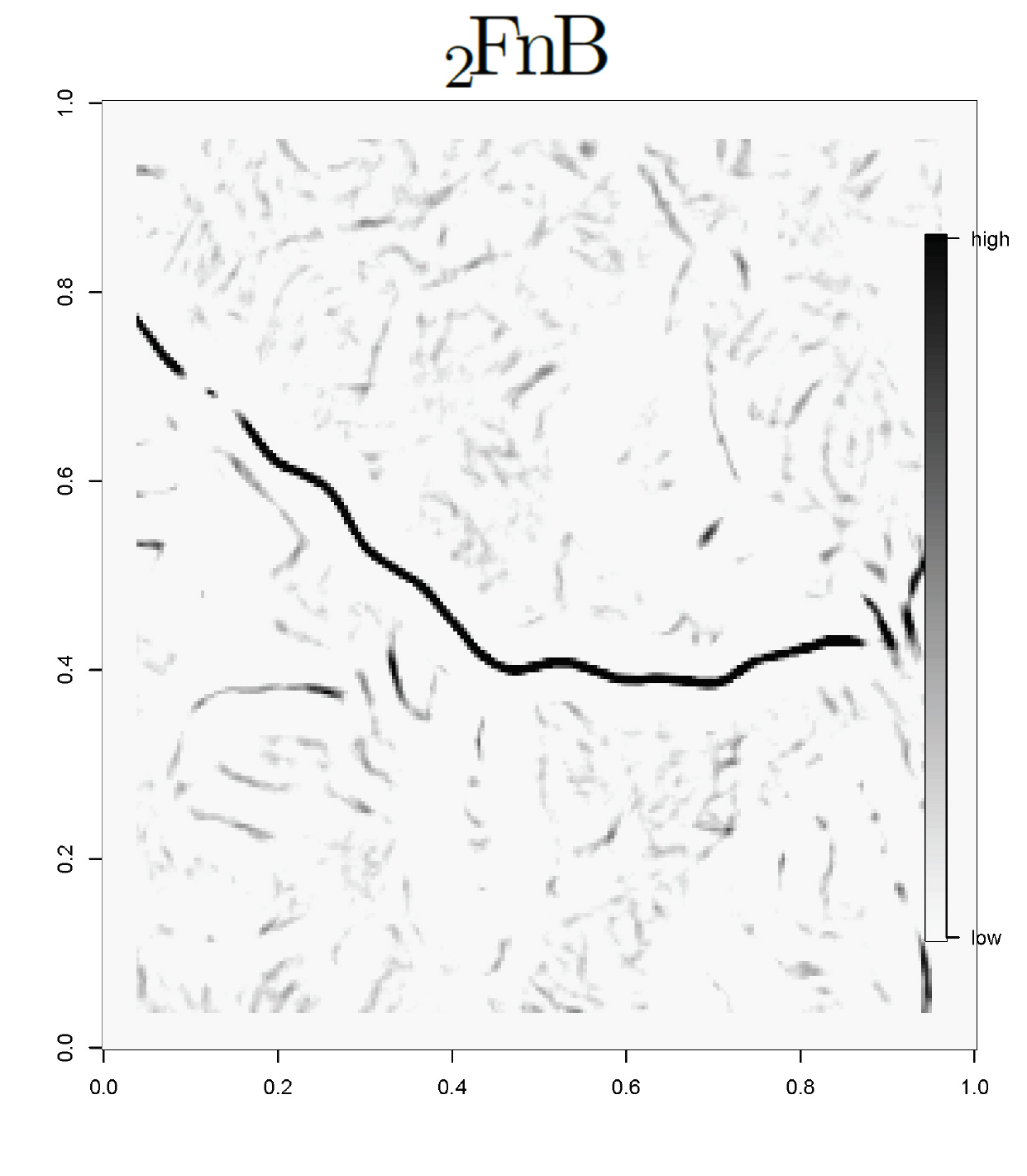}
		\caption[Geometric properties of fissures and bubbles together with scan statistics.]{The image shows heat maps for various statistics applied to the left panel of Figure \ref{figure_motivation}. The left panel shows a heat map of $ \F2 $ from \eqref{f2}, the middle panel shows a heat map of $ \nB $ from \eqref{nB} and the right panel shows a heat map for $ \FnB2 $ from \eqref{fnb2}, illustrating that $ \FnB{2} $ enhances fissures while discarding most of the bubbles and noise.}\label{figure_danobi_stat1}
	\end{figure}

	When the direction of the mean change for the fissure is not clear, we can use a two-sided version of the statistic by taking absolute differences in \eqref{f1}. Applying such a two-sided version $\F2$ as in \eqref{f2} to the left example in Figure~\ref{figure_motivation} does indeed enhance the fissure but also  (to a lesser degree) the borders of the naturally occurring anomalies in addition to some \emph{shadow} effects around the fissure (see left panel in Figure~\ref{figure_danobi_stat1}).
	Because these natural anomalies have a different geometric structure than the fissures,  their borders can be enhanced by contrasting the means of the two semi-circles (maximized over various angles) as in Figure~\ref{figure_danobi_method} (b). Indeed, the middle panel of Figure~\ref{figure_danobi_stat1} illustrates that the edges of the natural anomalies are enhanced by this statistic.
	
	Consequently, we define in the following the $ \FnB{2} $-statistic by subtracting these two scan statistics which enhances the fissures while at the same time discarding to a large degree the natural anomalies as can be seen by the right panel of Figure~\ref{figure_danobi_stat1}. Thresholding this last picture  results in a good detection of the fissure (see the left panel of Figure~\ref{figure_threshold_w5} below).

	To elaborate, for $ T>0 $, $ 0\le\alpha_1<\ldots<\alpha_P<\pi $, $ \bs\in[d/2,1-d/2]^2 $, define the two-sided \textbf{F}issures-\textbf{n}o-\textbf{B}ubbles $ \FnB{2} $- statistic by
	\begin{align}
		&\FnB2_T(\bs)=\ \max\geschweift{\F2_T(\bs)-\nB_T(\bs),0}, \label{fnb2}\\
		&\text{where }\F2_T(\bs)=\max_{i=1,\ldots,P}\frac{1}{\sigma}\min\geschweift{\abs{\danobimean^{(12,\alpha_i)}(\bs)},\abs{\danobimean^{(13,\alpha_i)}(\bs)}}\label{f2}\\	&\phantom{\text{where }} \nB_T(\bs)=\max_{i=1,\ldots,P}\nB(\bs,\alpha_i)=\  \max_{i=1,\ldots,P}\frac{1}{\sigma}\abs{\danobimean^{(45,\alpha_i)}(\bs)}, \label{nB}\\
		&\phantom{\text{where }\nB_T(\bs)=}\danobimean^{(45,\alpha)}(\bs)=\ \danobimean_{\scanset^{(4,\alpha)}}\left(\floor{\bs}_T\right)-\danobimean_{\scanset^{(5,\alpha)}}\left(\floor{\bs}_T\right)\notag
	\end{align}
	and $\F1_T(\bs,\alpha_i) $ is as in \eqref{f1alpha}.
    We obtain the asymptotic behavior of 
	$ \FnB{2} $ in the case of no anomalies analogously to the asymptotic behavior of $\F1$, as can be seen in Theorem \ref{theorem_danobi_size} (b) in the appendix. Hence we obtain thresholds for the $ \FnB{2}$ analogously to the thresholds for the $ \F1$-statistics.
	
	Furthermore, analogously to Theorem \ref{power_one}, we can show that for a correctly specified angle of the fissure, we obtain asymptotic power $ 1 $ if the signal is large enough, i.e.\ the subtraction of the $\nB$-statistic does not affect the detection of the fissures, as can be seen in Theorem \ref{power_one_fnb2} in the appendix.

	\subsubsection{$ \FnB{1} $-statistic} \label{sec_fnb1}
	\begin{figure}
		\includegraphics[width=0.32\textwidth]{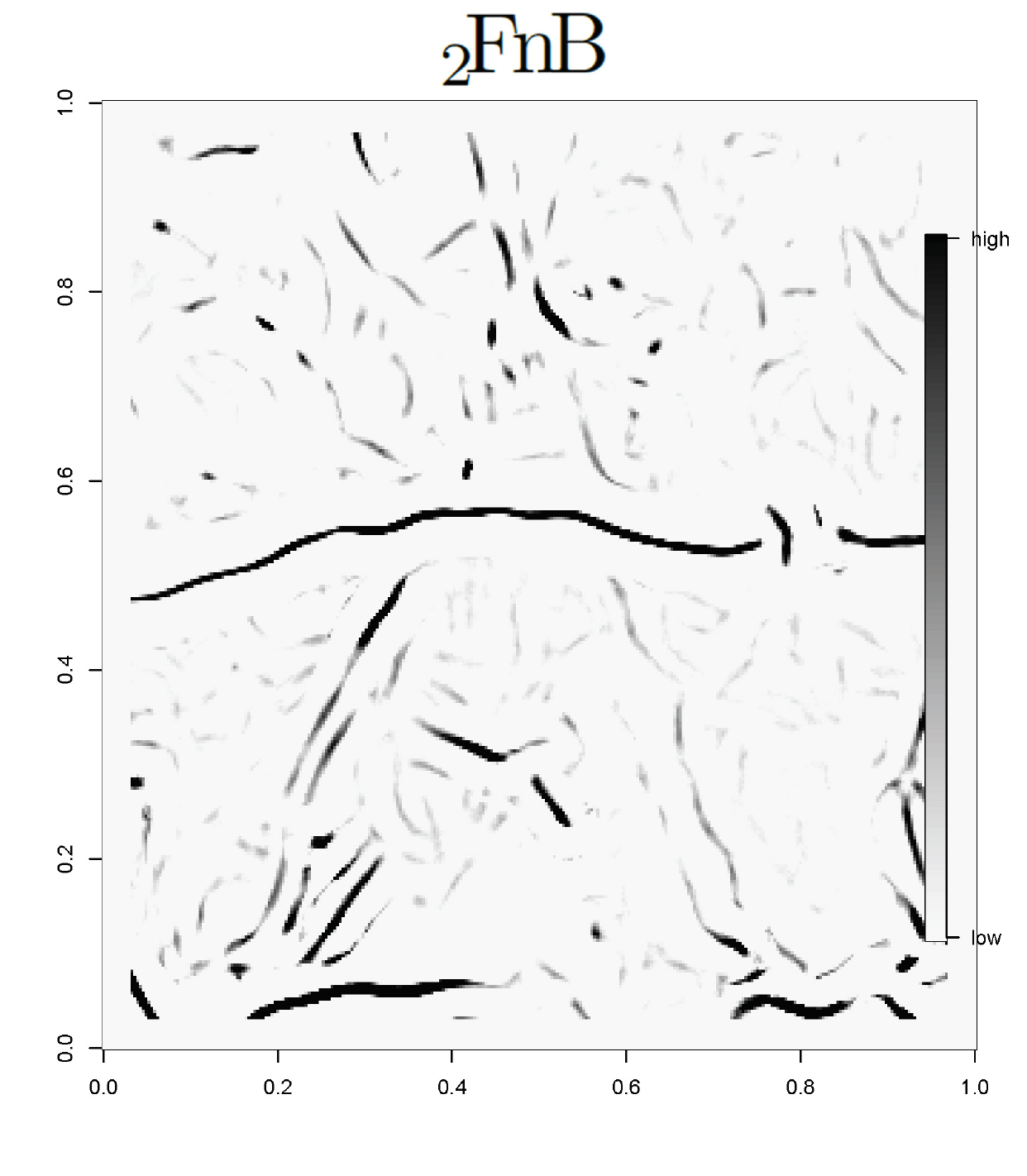}
		\includegraphics[width=0.32\textwidth]{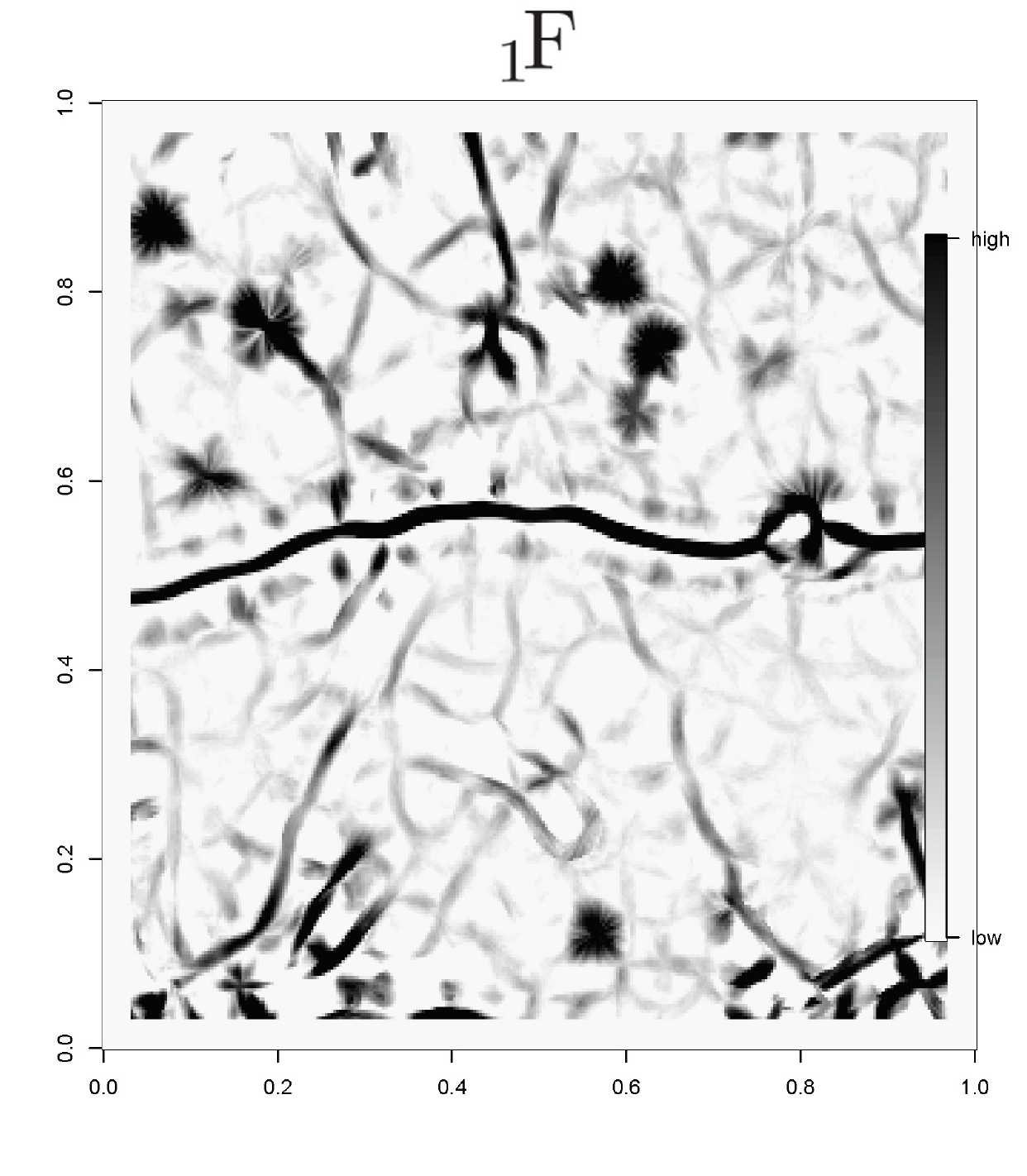}
		\includegraphics[width=0.32\textwidth]{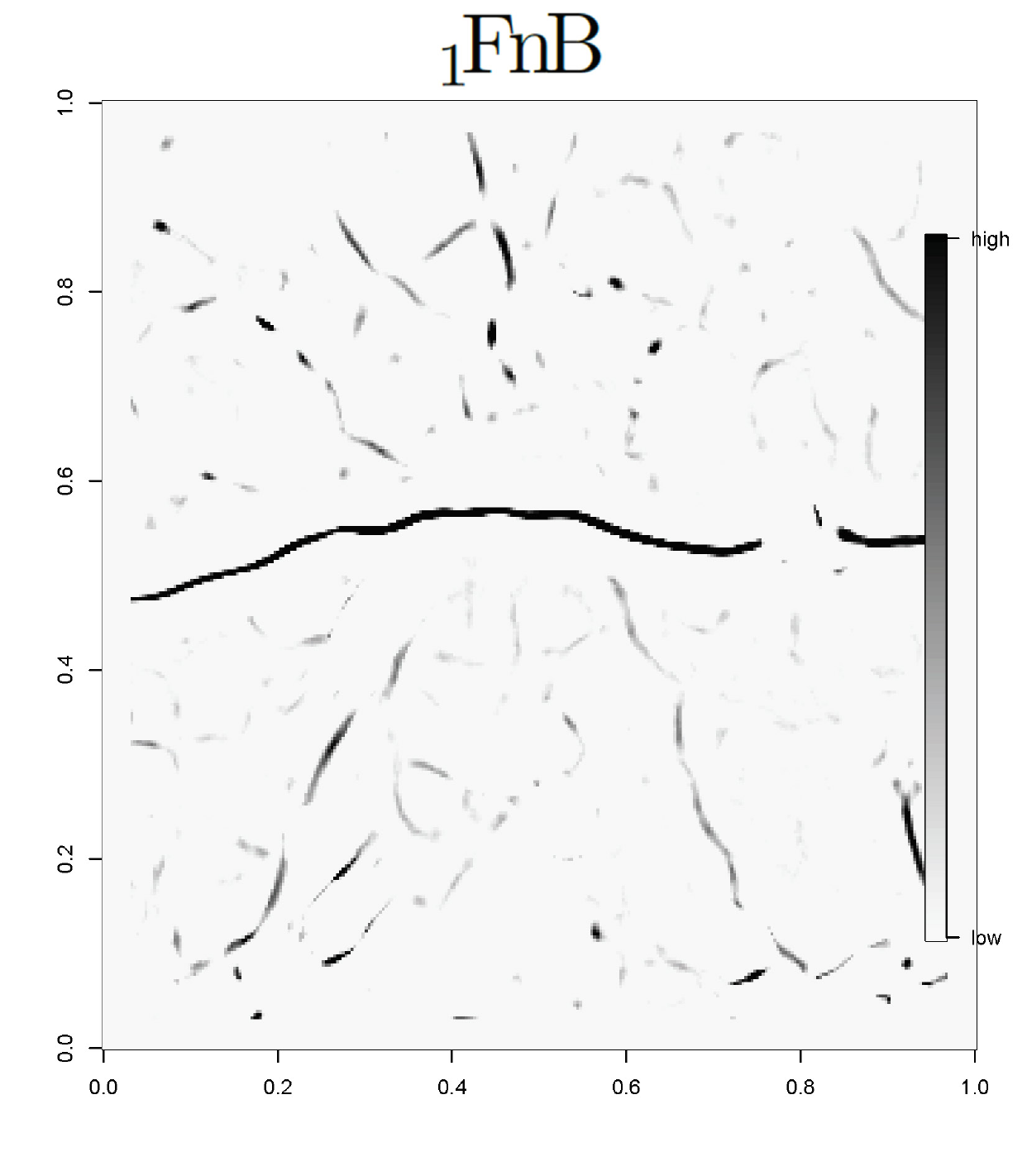}
		\caption{The heat maps from $ \FnB{2} $, $ \F1 $ and $ \FnB{1} $ applied to the steel fiber-reinforced concrete in the right panel of Figure \ref{figure_motivation} are displayed. 
			The statistics are maximized over $ 9 $ angles.}	 \label{figure_steelfiber_signif1}
	\end{figure}
	
	The right panel of Figure~\ref{figure_motivation} shows an example of steel fiber-reinforced concrete. In this case, the steel fibers have similar geometric properties as fissures.
	Not surprisingly, as a two-sided statistic, the $ \FnB{2} $-statistic detects both the fissures and steel fibers (see the left panel of Figure~\ref{figure_steelfiber_signif1}) as it does not make use of the fact that -- in contrast to the steel fibers --  the gray values of fissures are lower than the gray values of their neighboring environment.
	The $\F1$-statistic does take this into account but produces a series of artifacts due to other natural anomalies (see the middle panel of Figure~\ref{figure_steelfiber_signif1}).
	
	We can deal with this by subtracting the $ \nB $-statistic as in \eqref{nB} from the  $ \F1 $-statistic as in \eqref{f1} and define
	\begin{align*}
		\FnB{1}_T(\bs)=\max\left\{\F1_T(\bs)-\nB_T(\bs),0\right\}. 
	\end{align*}
	The corresponding heat map is given in the right panel of Figure~\ref{figure_steelfiber_signif1}.
	
We obtain the asymptotic behavior of 
	$ \FnB{1} $ in the case of no anomalies analogously to the asymptotic behavior of $\F1$, as can be seen in Theorem \ref{theorem_danobi_size} (c) in the appendix.

		Furthermore, analogously to Theorem \ref{power_one}, we can show that if the angle of the fissure is correctly specified, we obtain asymptotic power $ 1 $ for a positive and large enough signal $\delta_T$, i.e.\ the subtraction of the $\nB$-statistic does not affect the detection of the fissures, as can be seen in Theorem \ref{power_one_fnb1} in the appendix.

	\begin{figure}
		\includegraphics[width=0.32\textwidth]{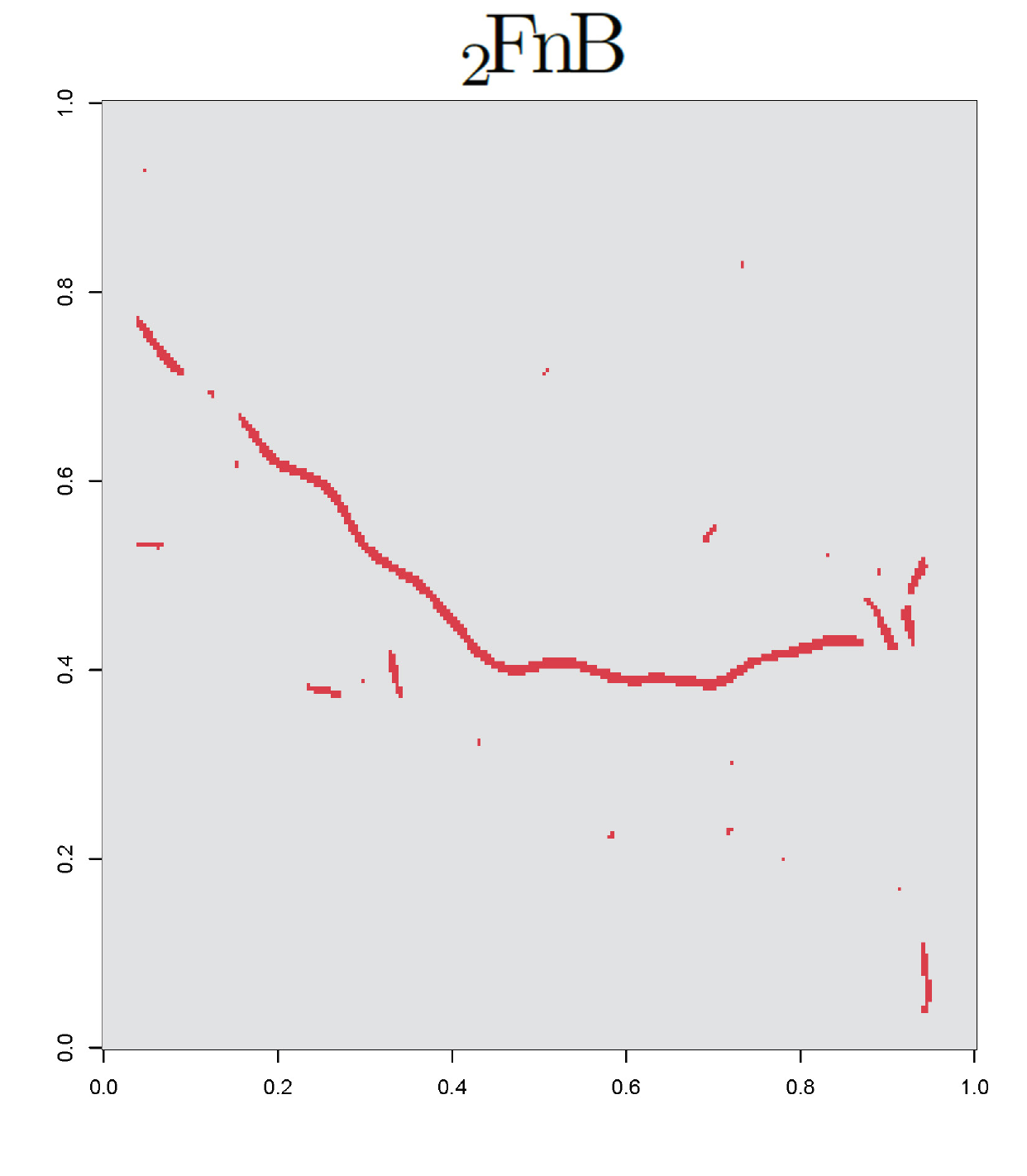}
		\includegraphics[width=0.32\textwidth]{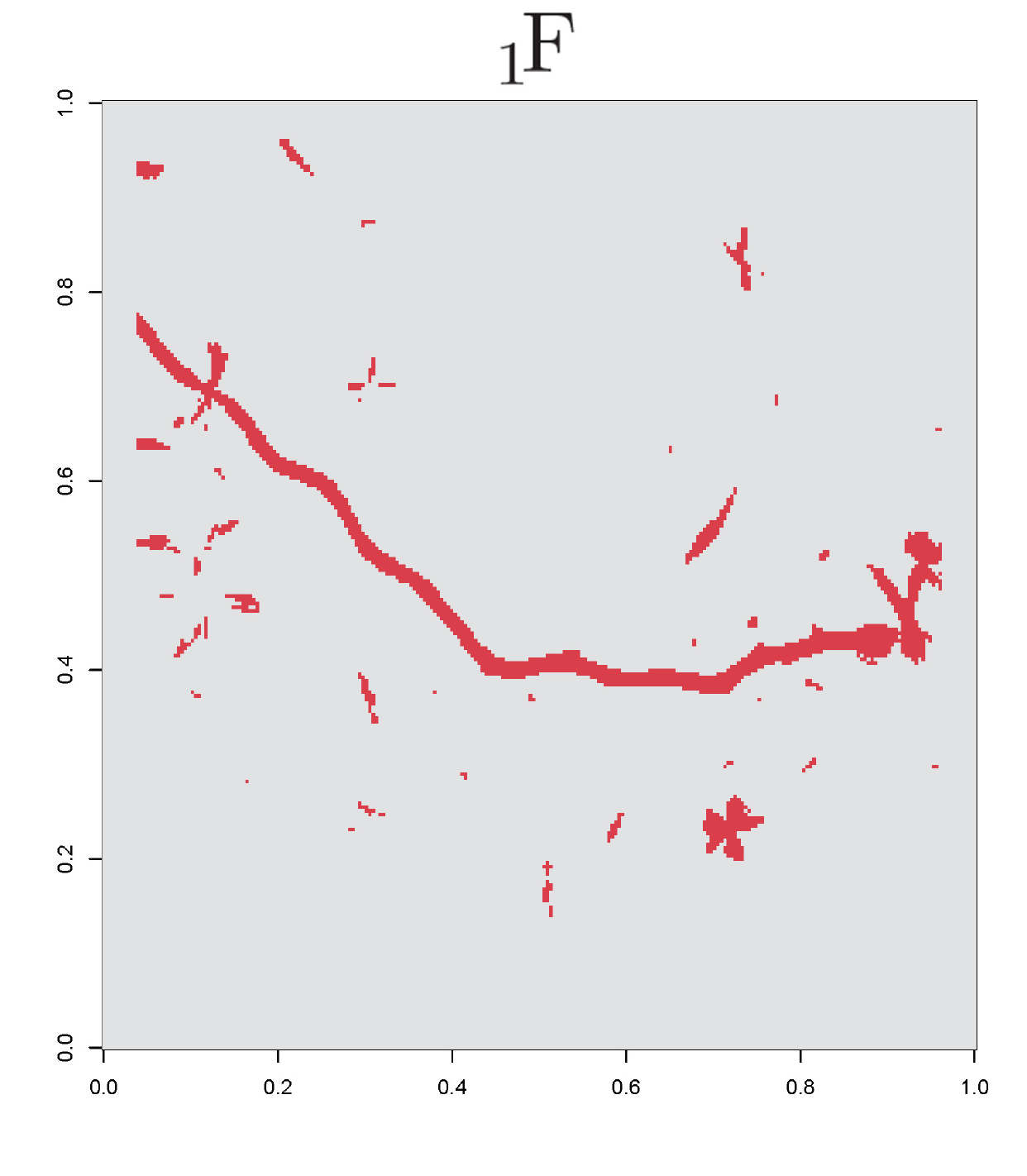}
		\includegraphics[width=0.32\textwidth]{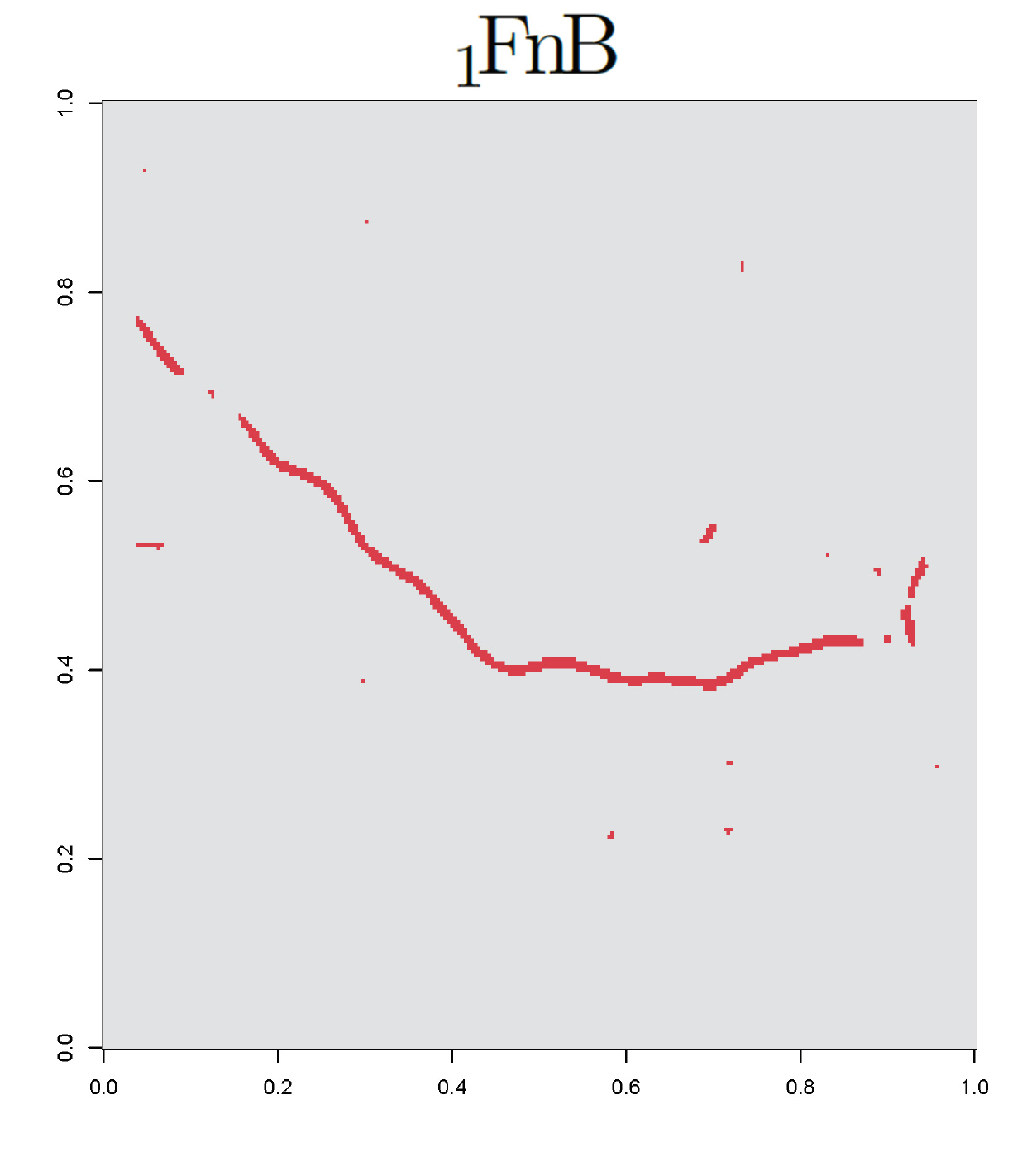}
		\caption{The thresholded heat maps from  $ \FnB{2} $, $ \F1 $ and $ 			\FnB{1} $ applied to the left panel of Figure \ref{figure_motivation} are displayed. 
			The statistics are maximized over $ 9 $ angles.}	 \label{figure_threshold_w5}
	\end{figure}

	\section{Further empirical results and data analysis} \label{danobi_simstudy}
	As the $ \FnB{1}$-statistic  has proven to do the task of detecting fissures and discarding bubbles best, we will perform the following empirical study with this statistic. In this simulation study we consider only i.i.d.\ standard normal noise. We aim to use the asymptotic $95\%$-quantiles of the maximum of the  $\FnB{1}$-statistic as thresholds $ \thr_{d,h,T,P} $.
	Since these  are not known analytically, we use empirical versions instead. To this end,  we generate $ 10\ 000 $ images of size $ T\times T $ with i.i.d.\ standard normal noise, scan each field with $ \FnB{1}_T(\bs,0) $ using relative width $ h $ for the width of the "inner strip" and calculate the empirical $ 95\% $-quantile of the maximal value of each image.
	
	All other empirical results (such as size or detection rates) in this section are based on $1000$ repetitions.
	
	In practice, the standard deviation $\sigma$ is not known and needs to be estimated (see also Remark~\ref{rem_variance_est}). Since anomalies lead to different expected values of the affected pixels, this will inevitably lead to distortions of corresponding estimators of the standard deviation unless robust estimators are employed.
	Therefore, we use the following global robust estimator (from here on referred to as "Silverman's estimator") which is based on an idea by \cite{Silverman} for the estimation of the optimal bandwidth in kernel density estimation:
	\begin{align}
		\hat{\sigma}_{T}(\bs)=\frac{\hat{q}_{T\times T}\left(0.75\right)-\hat{q}_{T\times T}\left(0.25\right) }{\Phi^{-1}\left(0.75\right)-\Phi^{-1}\left(0.25\right)}, \label{var_est}
	\end{align}
	where $ \Phi$  is the CDF of $ \mathcal{N}(0,1) $ and $\hat{q}_{T\times T}\left(\gamma\right)$ is the empirical $\gamma$-quantile of all pixels in the given $T\times T$-image.
	\begin{table}
		\centering

		\begin{tabular}{c|c|c|c|c|c}
			Emp. mean & Minimum & Lower quartile & Median & Upper quartile & Maximum  \\\hline
        $ 0.9998 $ & $0.9554$ & $0.9918$ &
			$ 0.9998 $ & $ 1.0076 $&  $ 1.0481 $
		\end{tabular}

		\caption{Empirical mean and various quantiles of the estimator \eqref{var_est} for $100\times 100$ images with i.i.d.\ standard normal pixels.}  \label{var_est_table_normal}
	\end{table}

	\subsection{Size control}\label{section_size_control}
	The estimator  \eqref{var_est} is asymptotically consistent if the noise is normally distributed.
	Therefore, it is not surprising  that it is a very precise estimator for the standard deviation for
	images of size $ 100\times100 $ with i.i.d.~$ \mathcal{N}(0,1) $-distributed pixels, see Table \ref{var_est_table_normal}.
	As a consequence,
	it also controls the relative number of false positives very precisely at $ 5\% $, as can be seen in Table \ref{var_est_fp}.
	\begin{table}
		\centering
		\begin{tabular}{c|cccc}
			$ P $ & $ 1 $ & $ 3 $ & $ 6 $ & $ 9 $\\\hline
			FP rate & $ 0.0508 $  & $ 0.0506 $  & $ 0.0498 $ & $ 0.0506 $
		\end{tabular}
		\caption{Relative number of false positives when using the estimator \eqref{var_est} for images of $ 100\times100 $ pixels without anomalies with i.i.d.~$ \mathcal{N}(0,1) $-distributed noise for the maximization over $ P $ angles.}  \label{var_est_fp}
	\end{table}
	It is important to note that this estimator is not consistent for the standard deviation if the underlying distribution is not normal.
	We illustrate this fact by the theoretic value (towards which the estimator converges) in Table \ref{var_est_table_notnormal}. For all considered distributions this theoretic value is smaller than the true standard deviation. As a consequence, the corresponding procedure will no longer control the number of images with spurious significant pixels at a level of $ 5\% $ but at (much) higher levels instead, and therefore lead to too many spurious significant pixels. Therefore, the estimator \eqref{var_est} should only be applied if one is sufficiently sure that the data is close enough to normal. Otherwise, more suitable robust variance estimators should be employed.
	\begin{table}
		\centering
		%	\begin{tabular}{c|c|c|c|c|c|c|c}
			\begin{tabular}{c|ccccccc}
				Distribution & 	$ \tdist(3) $ & $ \tdist(4) $ & $ \tdist(5) $ & $ \tdist(6) $ & $ \tdist(7) $ & $ \Exp(1) $ & $ \Gamma(4,2) $ \\\hline
				$ \sigma $  & $ 1.732 $ & $ 1.414 $ & $ 1.291 $ & $ 1.225 $ &  $ 1.183 $ & $ 1 $ &   $ 1 $ \\\hline
				$\frac{q(0.75)-q(0.25)}{\Phi^{-1}\left(0.75\right)-\Phi^{-1}\left(0.25\right)}$ & $ 1.134 $ & $ 1.098 $ & $ 1.077 $ & $ 1.064 $ & $ 1.054$ & $ 0.814 $ & $ 0.954 $
			\end{tabular}
			\caption{True standard deviation and theoretical values of "Silverman's estimator" \eqref{var_est} for various distributions. $ \tdist(n) $ denotes Student's $ t $-distribution with $ n $ degrees of freedom while $ \Gamma(s,\lambda) $ denotes the $ \Gamma $-distribution with shape $ s $ and rate $ \lambda $.} \label{var_est_table_notnormal}
		\end{table}

		\subsection{Angle misspecification}
		\begin{figure}
			\includegraphics[width=0.32\textwidth]{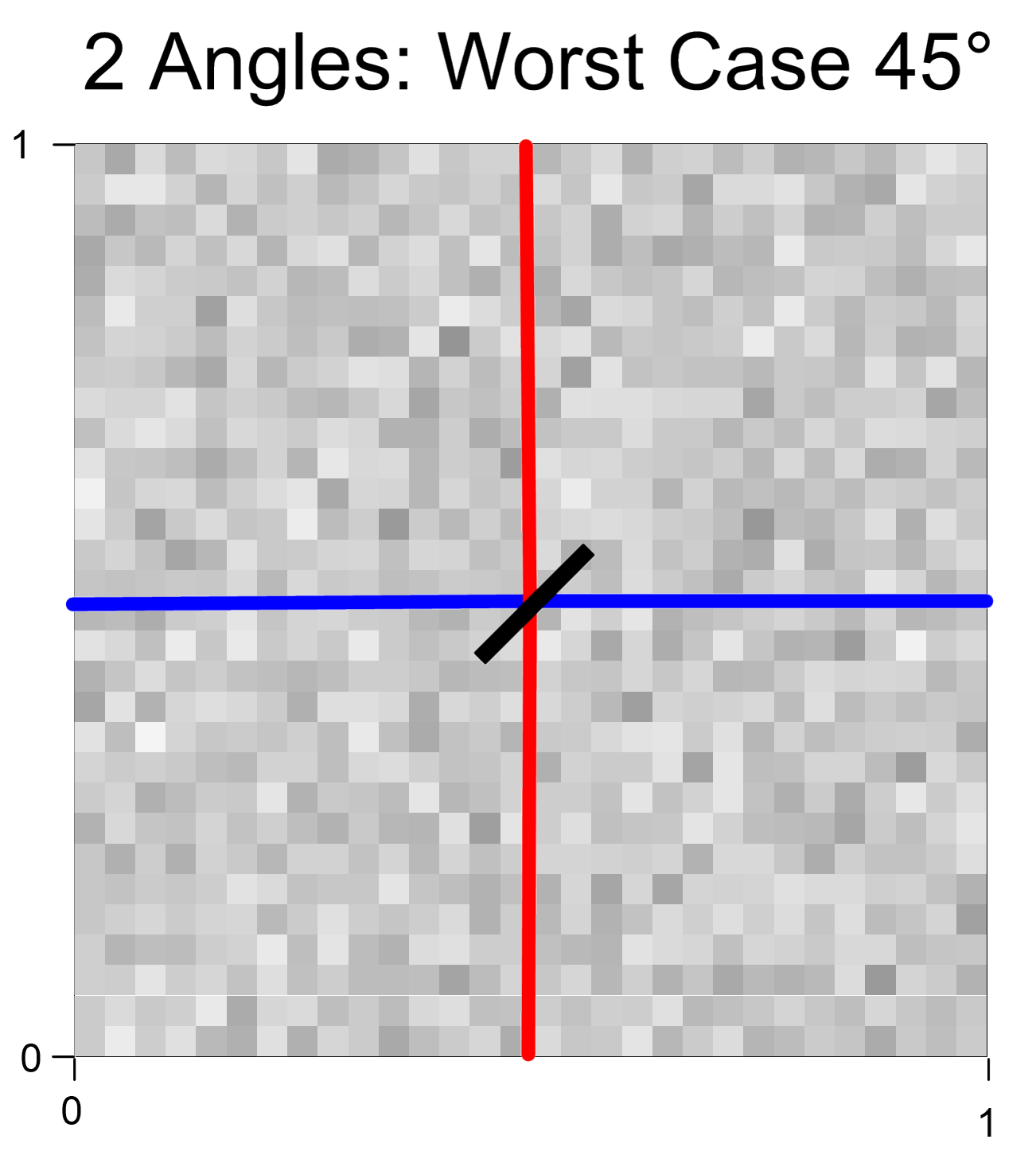}
			\includegraphics[width=0.32\textwidth]{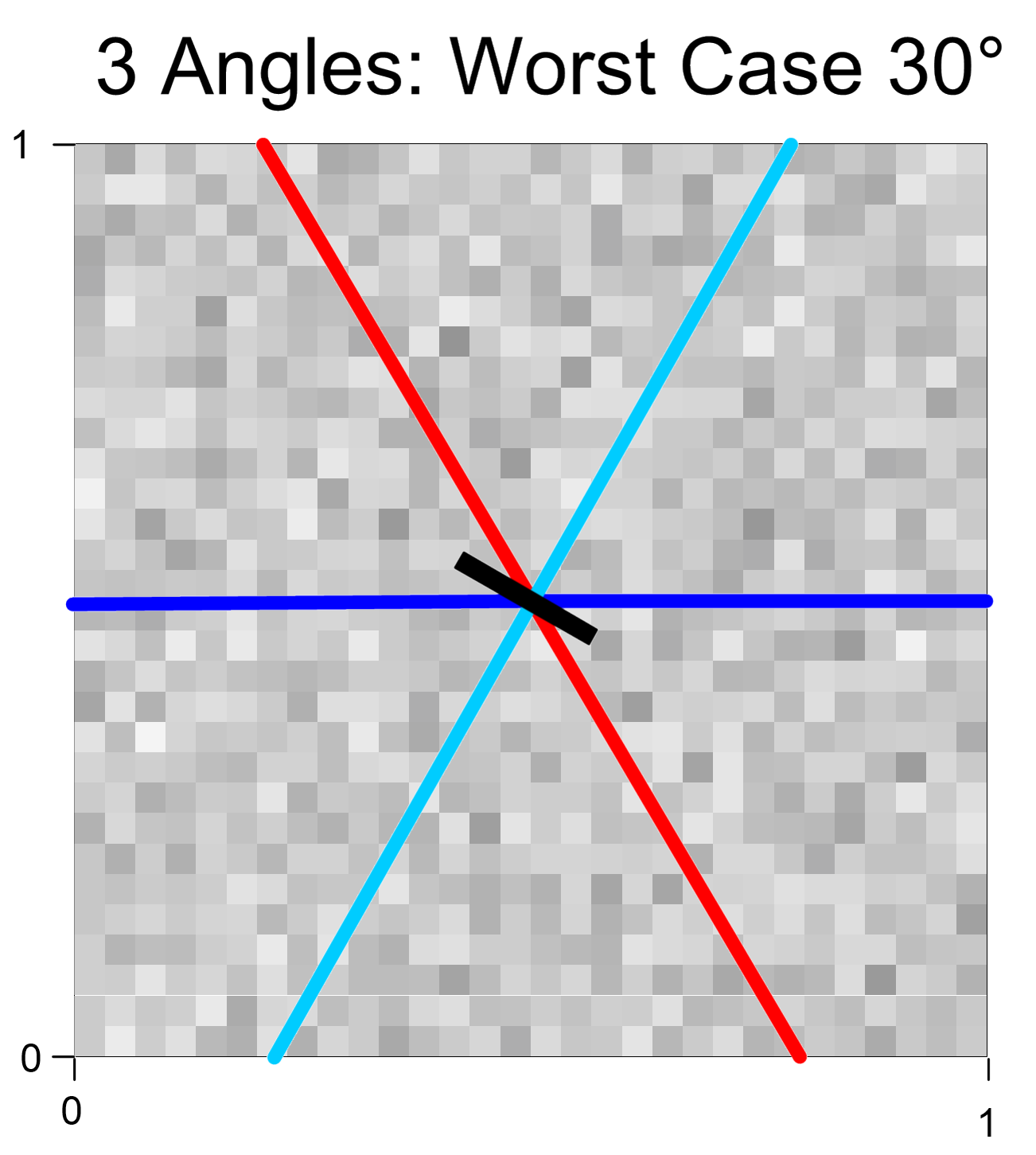}
			\includegraphics[width=0.32\textwidth]{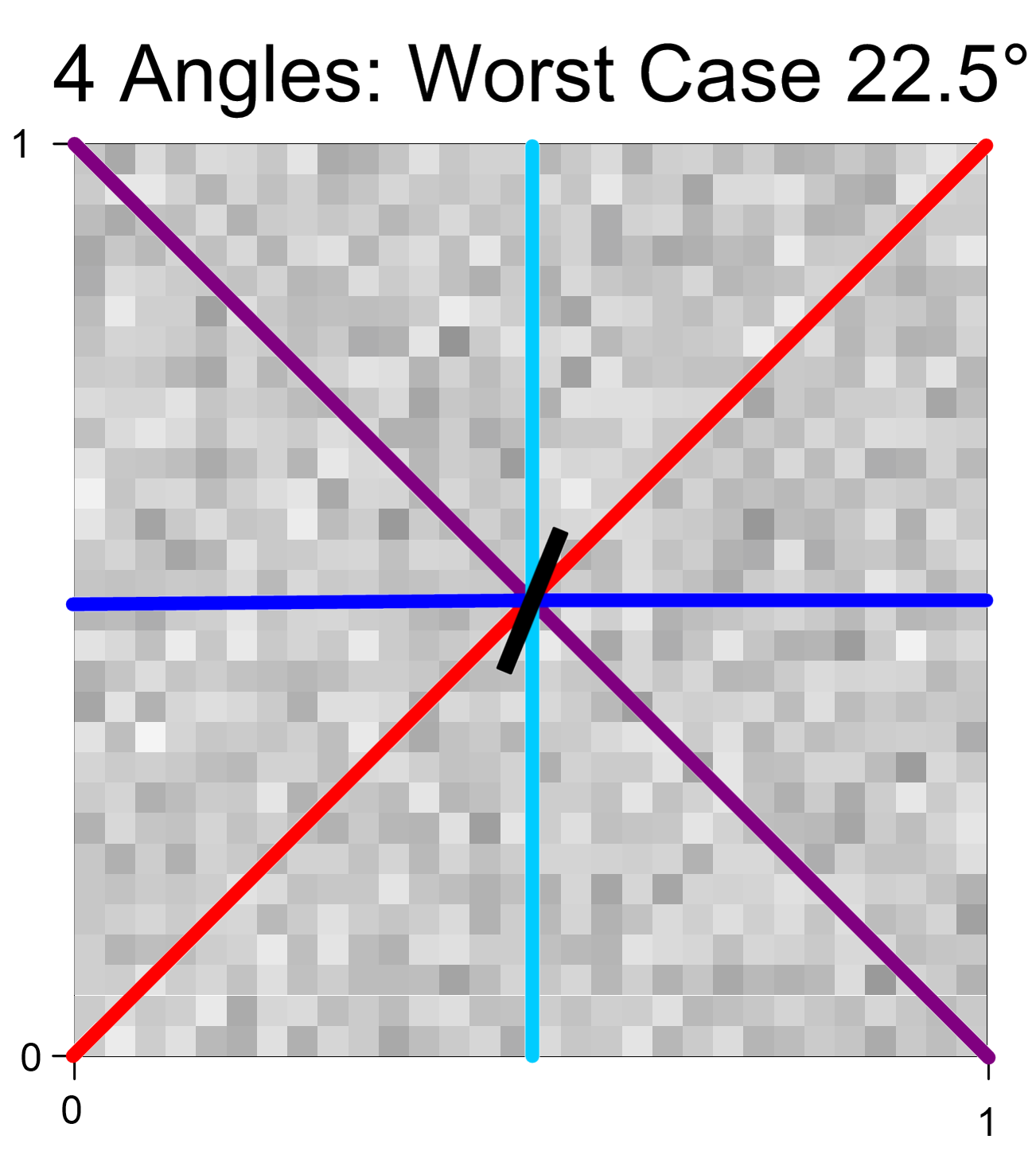}
			\caption{Schematic illustration of a worst-case scenario for the direction of a fissure for various numbers of angles $ P $ if  these angles are chosen equidistantly on $ \rechtsoffen{0^{\circ},180^{\circ}} $. The fissure is schematically illustrated as a thin black rectangle, while the colored lines represent the angles of the scan sets from Figure \ref{figure_danobi_method}. } \label{misspecification}
		\end{figure}
		
		The question, how many angles are necessary for a reliable detection of a fissure, is central for the setup of the methodology. If the number of angles is fixed, then the angles should be chosen as an equidistant partitioning of $ \rechtsoffen{0^{\circ},180^{\circ}} $, because this choice minimizes the worst-case misspecification (see  Figure \ref{misspecification} for an illustration).
		
		A very natural question, in particular in view of computation time, is the number of  angles $ P $  needed to detect a fissure at a "reasonably high" rate.
		If $ \alpha_1,\ldots,\alpha_P $ are chosen equidistantly on $ \rechtsoffen{0^{\circ},180^{\circ}} $, the direction of a fissure can be misspecified  by an angle of $ (90/P)^{\circ} $ in the worst case (see Figure \ref{misspecification}).
		
		In this subsection we investigate empirically the detection rate with respect to misspecification of the angle $\Delta$ as  given by the difference between the angle of the "inner strip" $ A^{(1,\alpha)} $ from Figure \ref{figure_danobi_method} and the fissure.
		Indeed, if a misspecification of at most $\Delta$ still leads to  sufficiently high detection rates, then
		at most $ \ceil{90/\Delta} $ angles are needed to obtain this detection rate because this choice guarantees that the worst-case misspecification is given by $\Delta$.

		\subsubsection*{Simulation setup}
		We simulate data according to the setting of \eqref{signal} with $ \mu_0=0 $ such that $Y_{\bk}= -\mathds{1}_{\geschweift{\bk/T\in \anomaly}}\delta+\epsilon_{\bk}$,
		where $ \left(\epsilon_{\bk}\right)_{\bk\in\Z^2} $ are i.i.d.~$ \mathcal{N}(0,1) $.
		\begin{figure}[b!]
			\includegraphics[width=\textwidth]{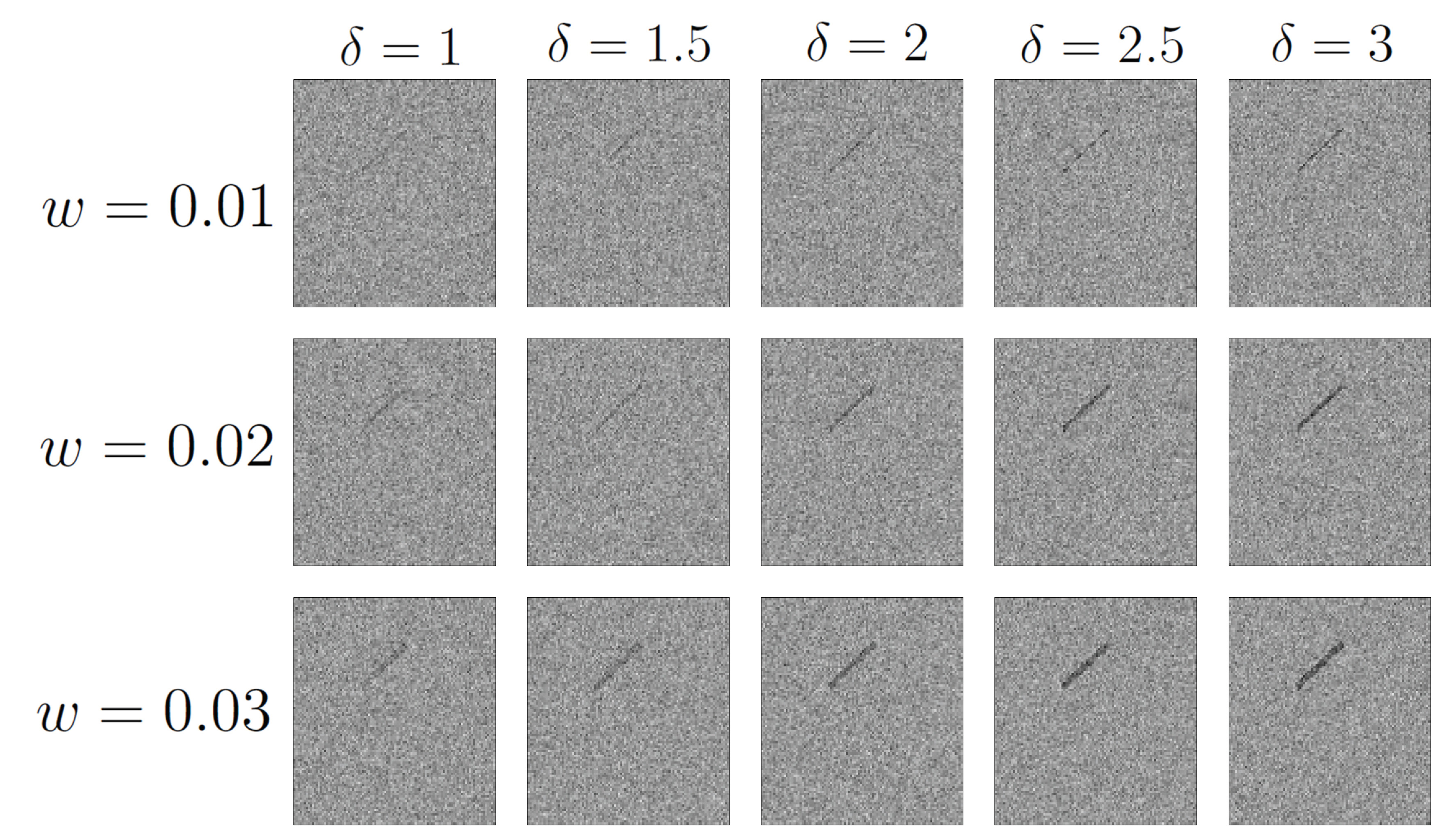}
			\caption[Examples of fissures with varying signal-width-combinations.]{One realization for each signal-width-combination $ (\mu,w) $ that is considered in this simulation study.}
			\label{danobi_bsp}
		\end{figure}
		
		In all simulations we use an inner strip of width $h=0.02$ and then discuss the misspecification with respect to the width by varying the actual width of the simulated fissure between $w=0.01, 0.02, 0.03$. For the misspecification of the angle, we consider   $\Delta=0^{\circ},5^{\circ},\ldots,25^{\circ}$.
  We use varying signal strengths as indicated by Figure~\ref{danobi_bsp}, which shows one realization of each scenario.

		\subsubsection*{Findings}\label{subsec_results}
		
		Figure \ref{figure_results} shows the corresponding empirical detection rates, where a fissure is considered \emph{detected}, if the corresponding statistic for at least one pixel $\bk/T$ is above the threshold and the corresponding scan window $ \scanset(\bk/T) $ intersects with the fissure.  Identifying such a  pixel  is sufficient for the purposes described in this paper and can be used e.g.\ as the starting point for machine learning methods (see e.g.\ \cite{KL}, \cite{KL2}, \cite{KL3}) to trace the fissure.
		
		\begin{figure}
			\includegraphics*[width=\textwidth]{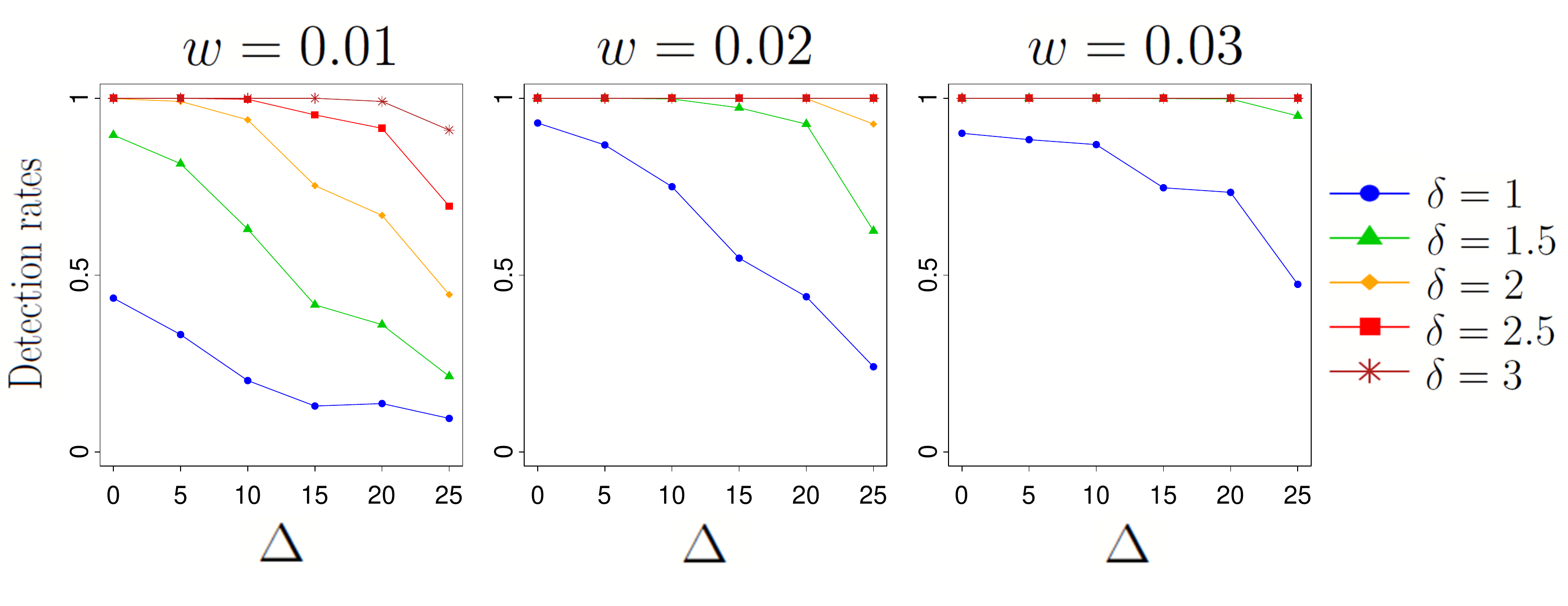}
			\caption[Detection rates with respect to angle misspecification.]{The image shows the detection rates for all parameter combinations $ (w,\mu,\Delta) $ when considering a single angle.} \label{figure_results}
		\end{figure}

		Even for moderate signal-to-noise ratios, where the fissure can  barely be observed by visual inspection (see Figure \ref{danobi_bsp}), fissures are detected at rates of at least $ 75\% $ with at most $ 9 $ angles. To elaborate, Table \ref{Table_75} lists all parameter combinations with an empirical detection rate of at least $75\%$ as well as the corresponding number $P$ of required angles to guarantee this detection rate.  With two exceptions (the first two scenarios in the first row in Figure~\ref{danobi_bsp}) choosing 9 angles is sufficient, where in most cases even $ 6 $ or less angles are enough.
		In the application the fissure can be expected to change direction so that a relatively small number of angles should be sufficient for detection in most cases. This is because -- due to the variation in direction -- at least part of the fissure should be only slightly misspecified with respect to one of the angles, which leads to a high likelihood of the fissure being detected.
		
		Our procedure also proves to be robust against slight misspecifications of the width of the fissure except in situations where the signal-to-noise ratio becomes too small to be reliably detected.

		\begin{table}[t]
			\begin{tabular}{c||c|c|c|c||c|c|c||c|c}
				$ w $ & $ 0.01 $ & $ 0.01 $ &$ 0.01 $ &$ 0.01 $ & $ 0.02 $ & $ 0.02 $ & $ 0.02 $ & $ 0.03 $ & $ 0.03 $\\
				$ \delta $ & $ 1.5 $ & $ 2 $ & $ 2.5 $ & $ 3 $ & $ 1 $ & $ 1.5 $ & $\ge 2 $ & $ 1 $ & $ \ge1.5 $ \\
				max. $ \Delta $ & $ 5 $ &$ 15 $ & $ 20 $ &$ 25 $  & $ 10 $& $ 20 $ & $ 25 $& $ 10 $ & $ 25 $ \\\hline
				min. $ P $ & $ 18 $ & $ 6 $ & $ 5 $ & $ 4 $  & $ 9 $ & $ 5 $ & $ 4 $ & $ 9 $ & $ 4 $
			\end{tabular}
			\vspace*{2mm}
			\caption[Parameter combinations for which detection rates in Section \ref{danobi_simstudy} are higher than $ 75\%$.]{The table lists all combinations  $ (w,\delta)$ and the maximum $\Delta$  for which the detection rates are higher than $ 75\% $.}
			\label{Table_75}
		\end{table}

		\subsection{Real data analysis} \label{real_data}
		In this subsection, we take another look at the data examples showing that the observations concerning the required number of angles from the previous subsection are also consistent with what happens in real data.
		
		In order to illustrate the performance of $ \FnB{1} $ on real data with respect to the number of angles needed for reliable fissure detection, we analyze the three images of Figure \ref{figure_motivation} with $3$, $6$ and $9$ equidistant angles starting at $0^{\circ}$. For the left image, which contains an obvious fissure, the $ \FnB{1}$-statistic is able to detect most of the pixels belonging to a fissure even when maximizing over just three angles, as can be seen in Figure \ref{figure_w5_signif_angles}.
		We can observe the same for the steel fiber-reinforced concrete from the right panel of Figure \ref{figure_motivation}, as can be seen in Figure \ref{figure_steelfiber_signif}. Additionally, as has been discussed in Subsection \ref{sec_fnb1}, the $ \FnB{1} $-statistic is able to eliminate the steel fibers. For the middle panel in Figure \ref{figure_motivation}, which contains a thin fissure that can barely be observed by visual inspection,  the  $ \FnB{1} $-statistic can detect some pixels belonging to the fissure even for $ 3 $ angles. Furthermore, with an increasing number of angles parts of the fissure are enhanced better, as can be observed in the upper row of Figure \ref{figure_w1_stats_angles}, and that more pixels belonging to the fissure are significant, as can be seen in the lower panels of Figure \ref{figure_w1_stats_angles}.
		
		These findings are consistent with the simulation results in Subsection \ref{subsec_results}, where even for thin fissures that can barely be observed by visual inspection, $ 9 $ angles are sufficient to have some significant pixels adjacent to the fissure.
		
		\begin{figure}
			\includegraphics[width=0.32\textwidth]{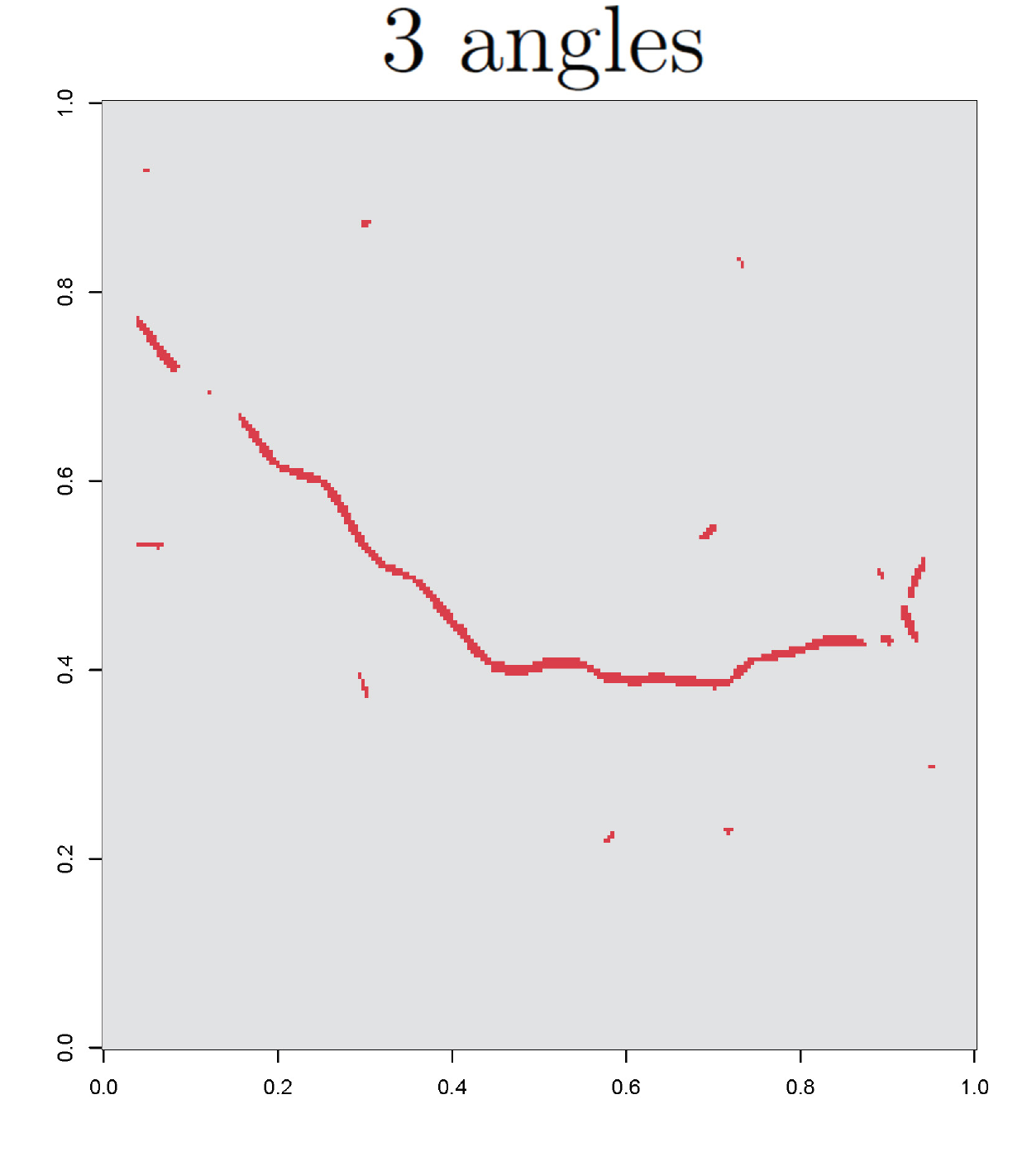}
			\includegraphics[width=0.32\textwidth]{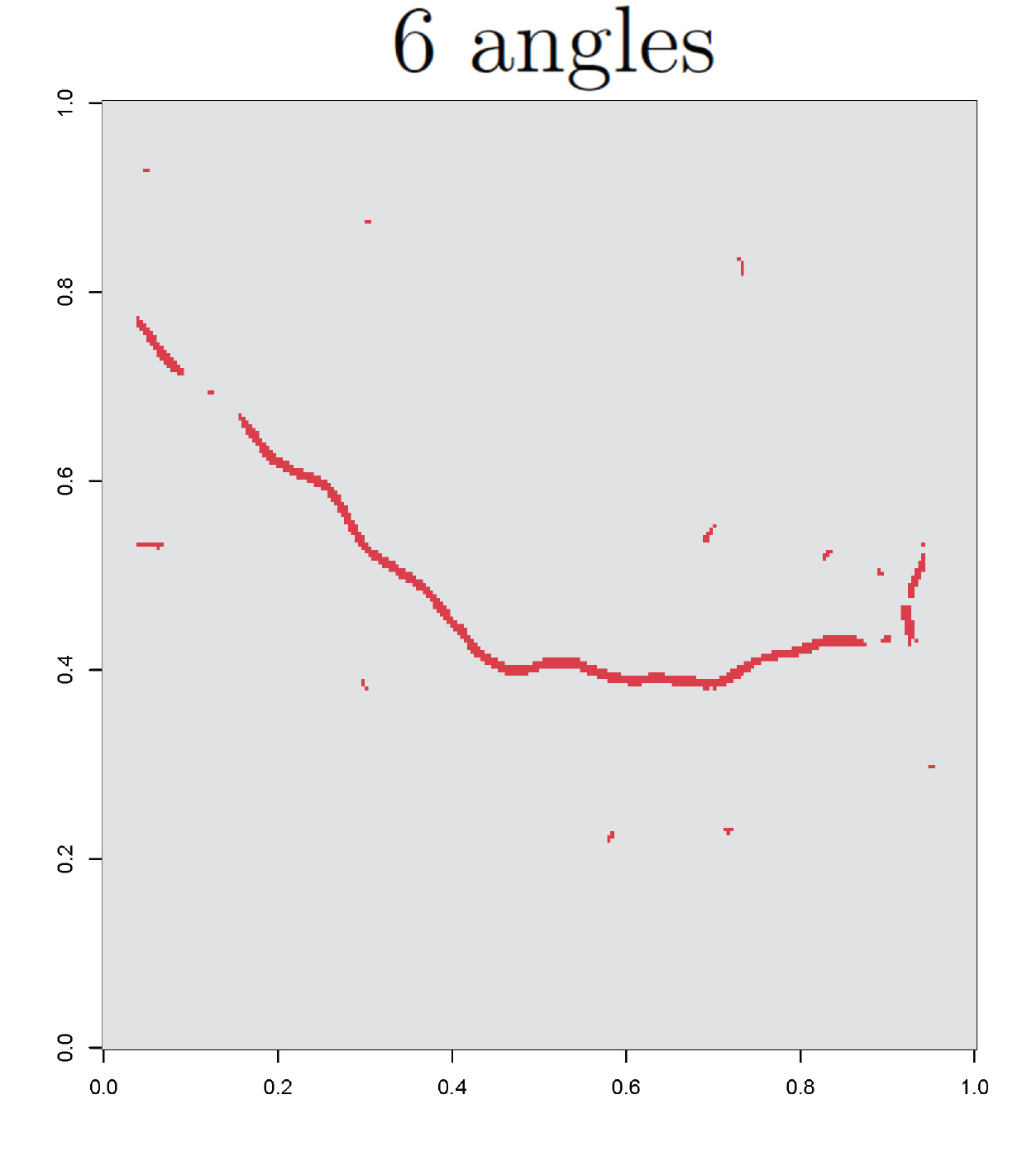}
			\includegraphics[width=0.32\textwidth]{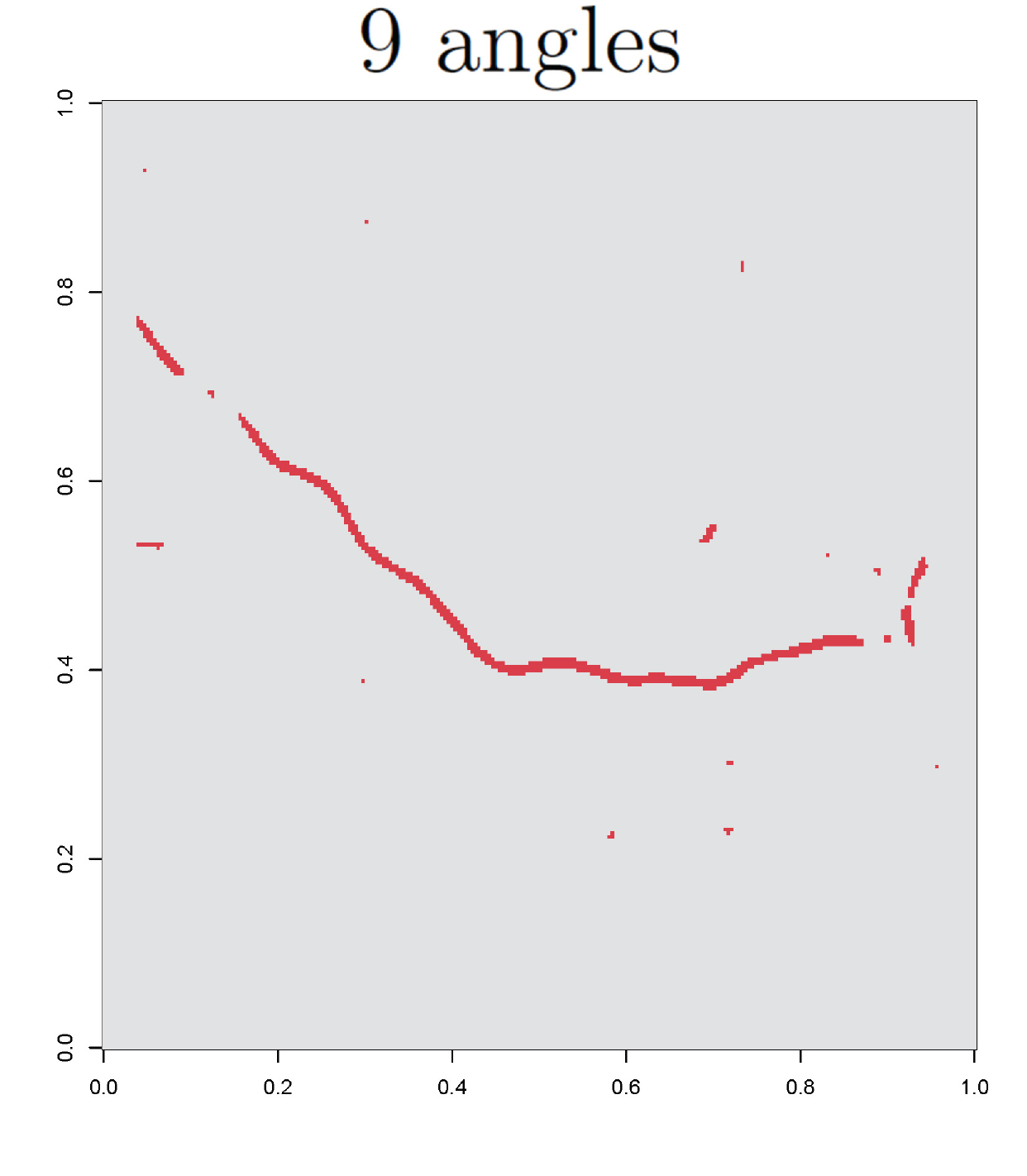}
			\caption{The thresholded $ \FnB{1} $ statistics  applied to the left panel of Figure \ref{figure_motivation} are displayed, where the statistic was maximized over $ 3 $, $ 6 $ and $ 9 $ angles, respectively.} \label{figure_w5_signif_angles}
		\end{figure}
		\begin{figure}
			\includegraphics[width=0.32\textwidth]{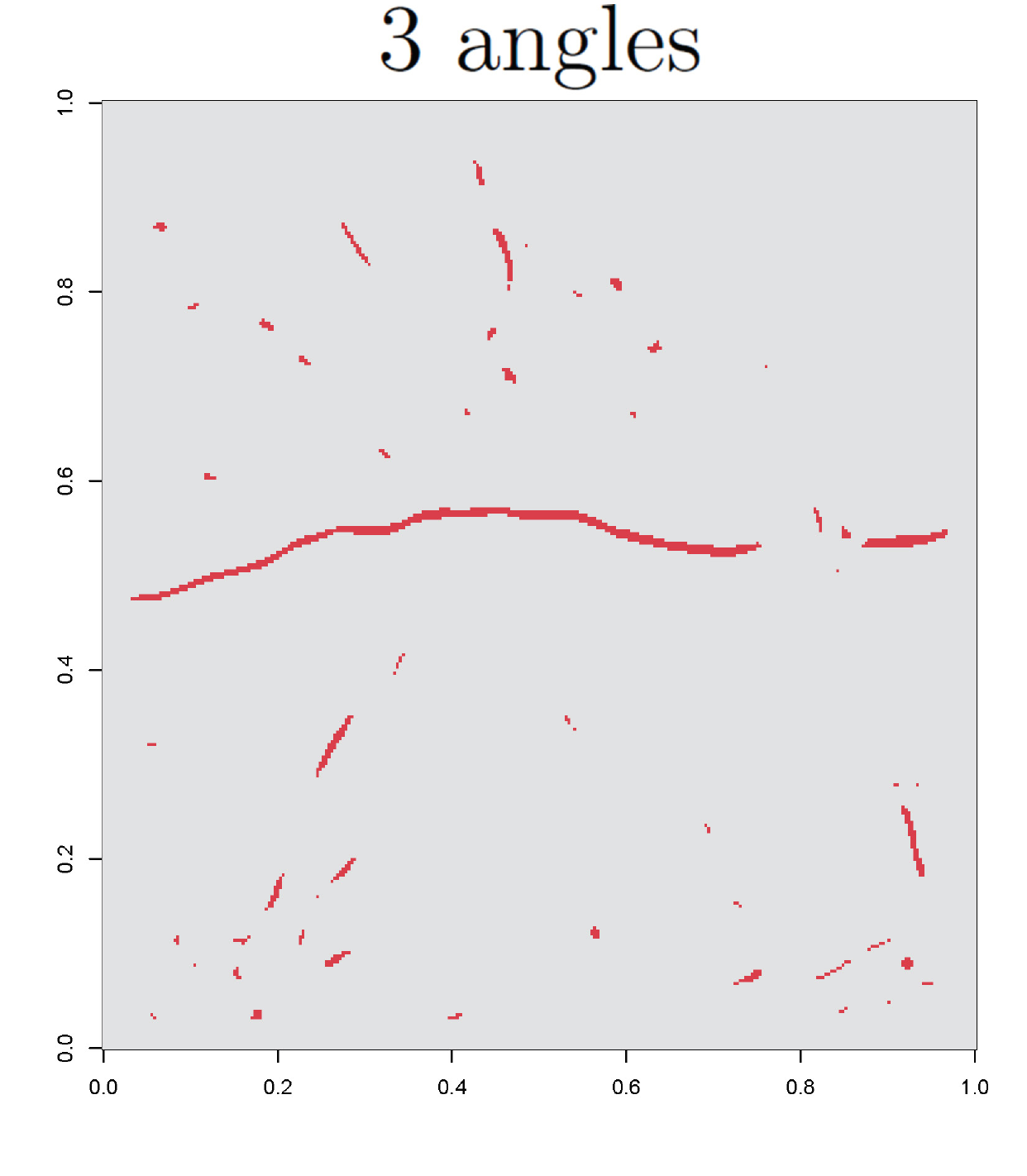}
			\includegraphics[width=0.32\textwidth]{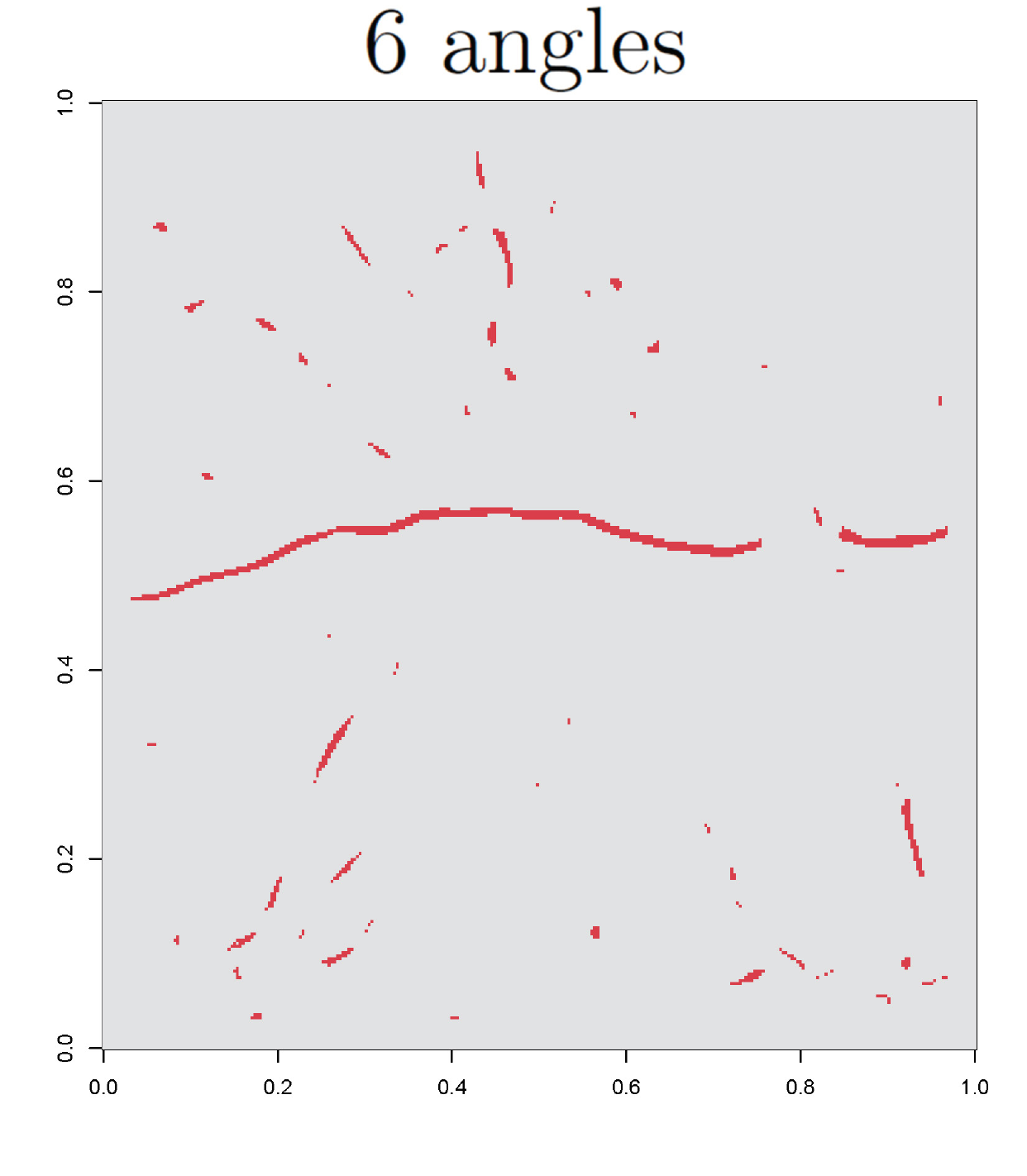}
			\includegraphics[width=0.32\textwidth]{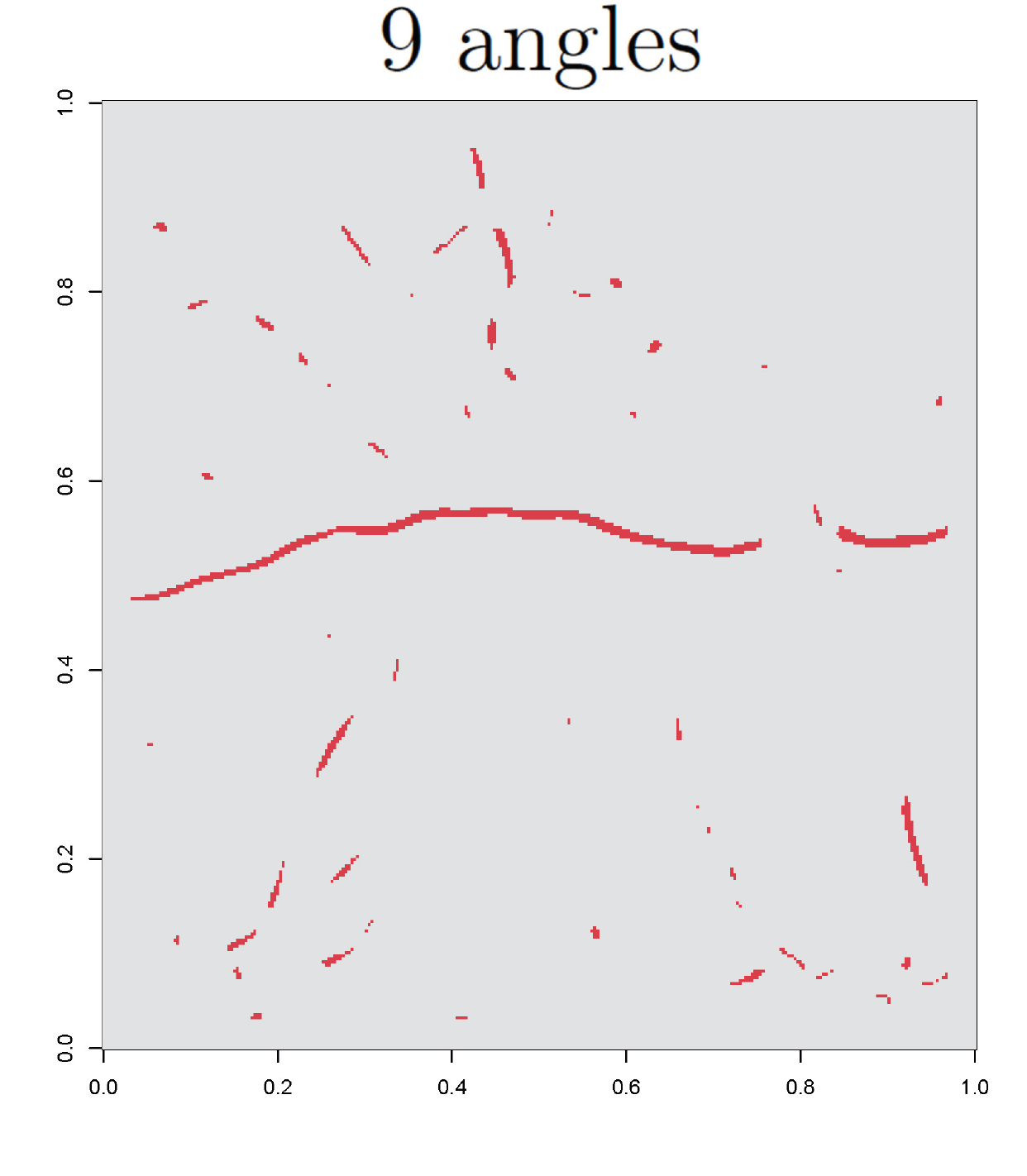}
			\caption{In this figure, the thresholded statistics from $ \FnB{1} $ applied to steel fiber-reinforced concrete in the right panel of Figure \ref{figure_motivation} are displayed, where the statistic was maximized over $ 3 $, $ 6 $ and $ 9 $ angles, respectively.} \label{figure_steelfiber_signif}
		\end{figure}

		\begin{figure}
			\includegraphics[width=0.32\textwidth]{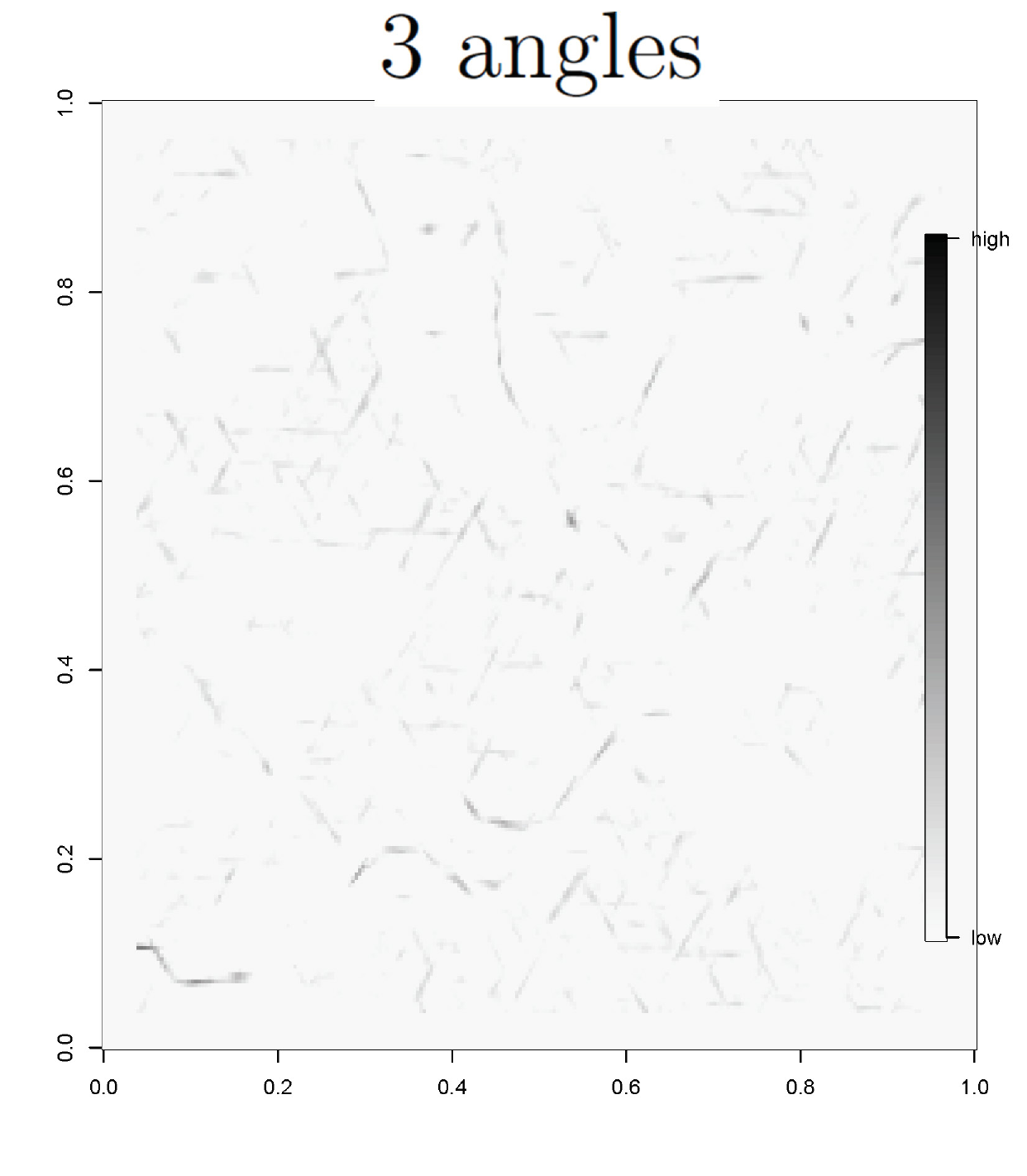}
			\includegraphics[width=0.32\textwidth]{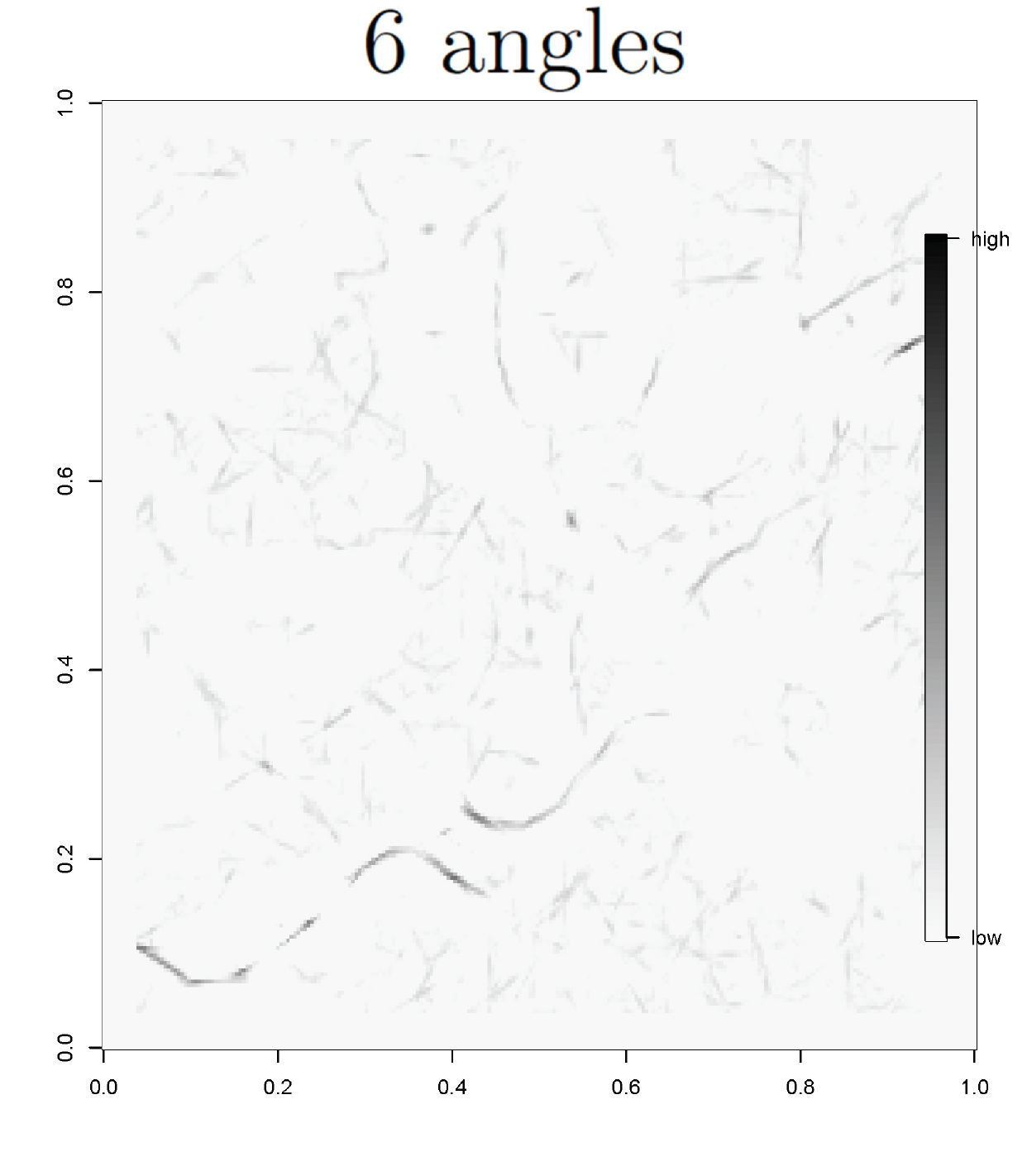}
			\includegraphics[width=0.32\textwidth]{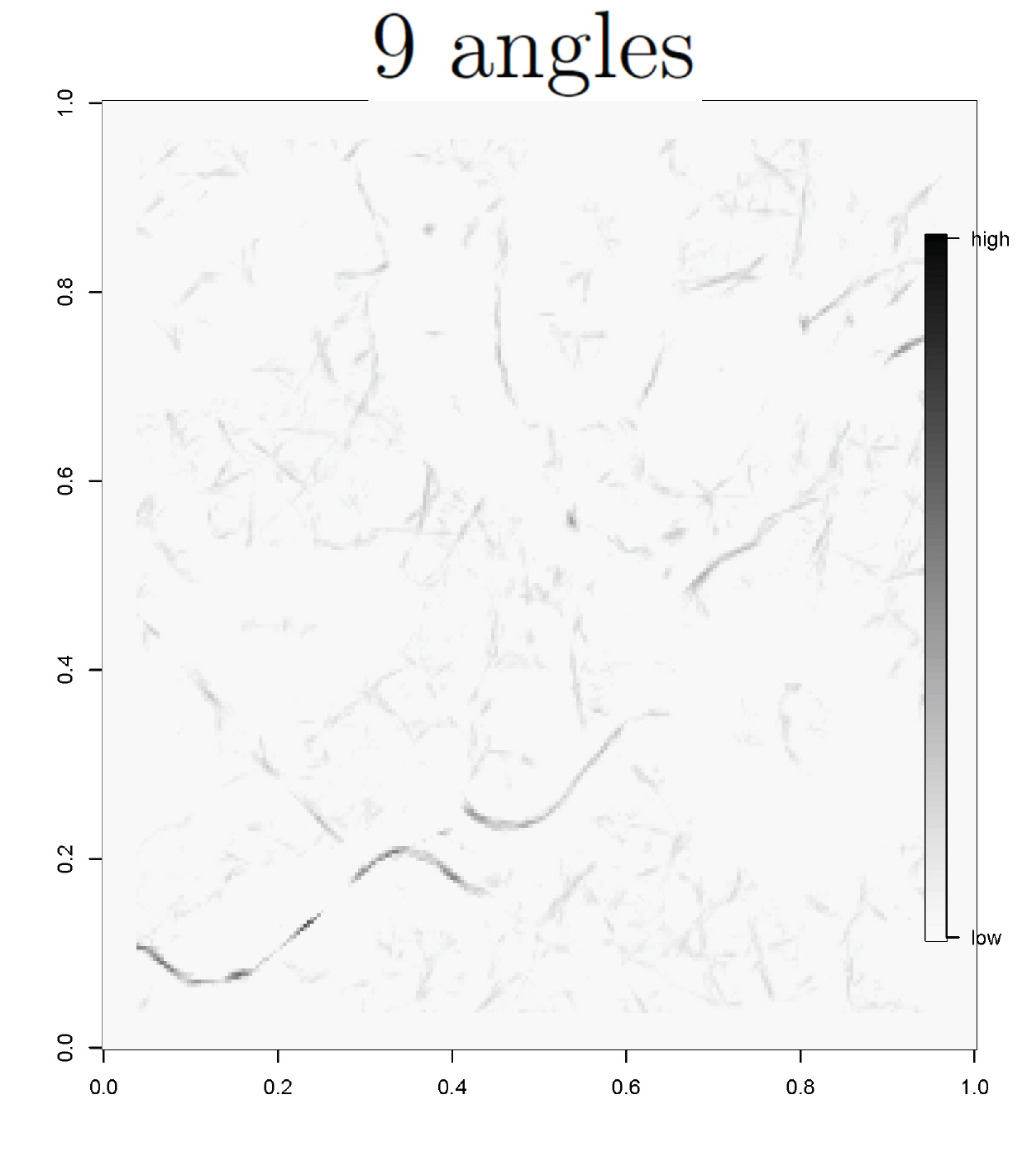}\\
			\includegraphics[width=0.32\textwidth]{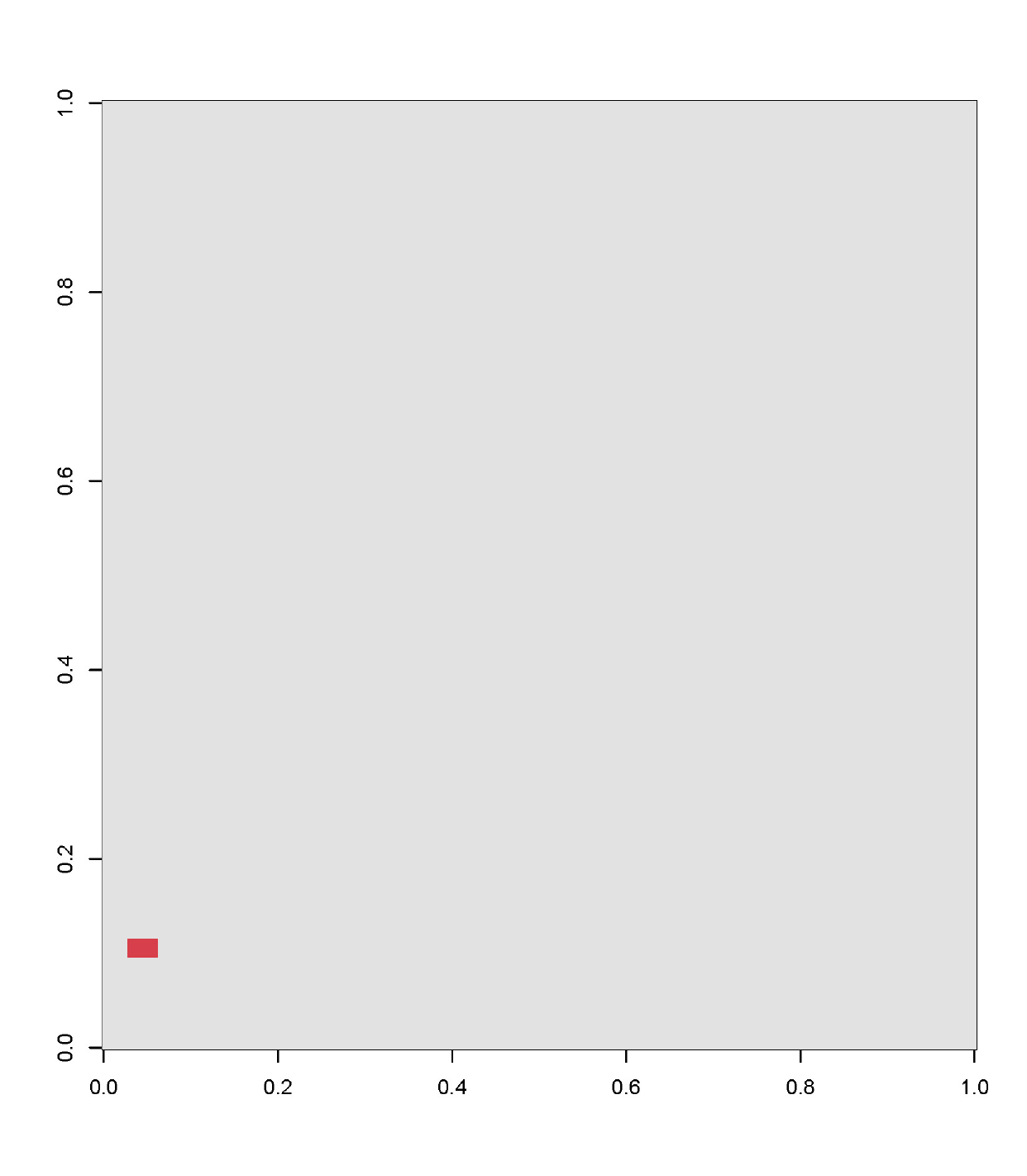}
			\includegraphics[width=0.32\textwidth]{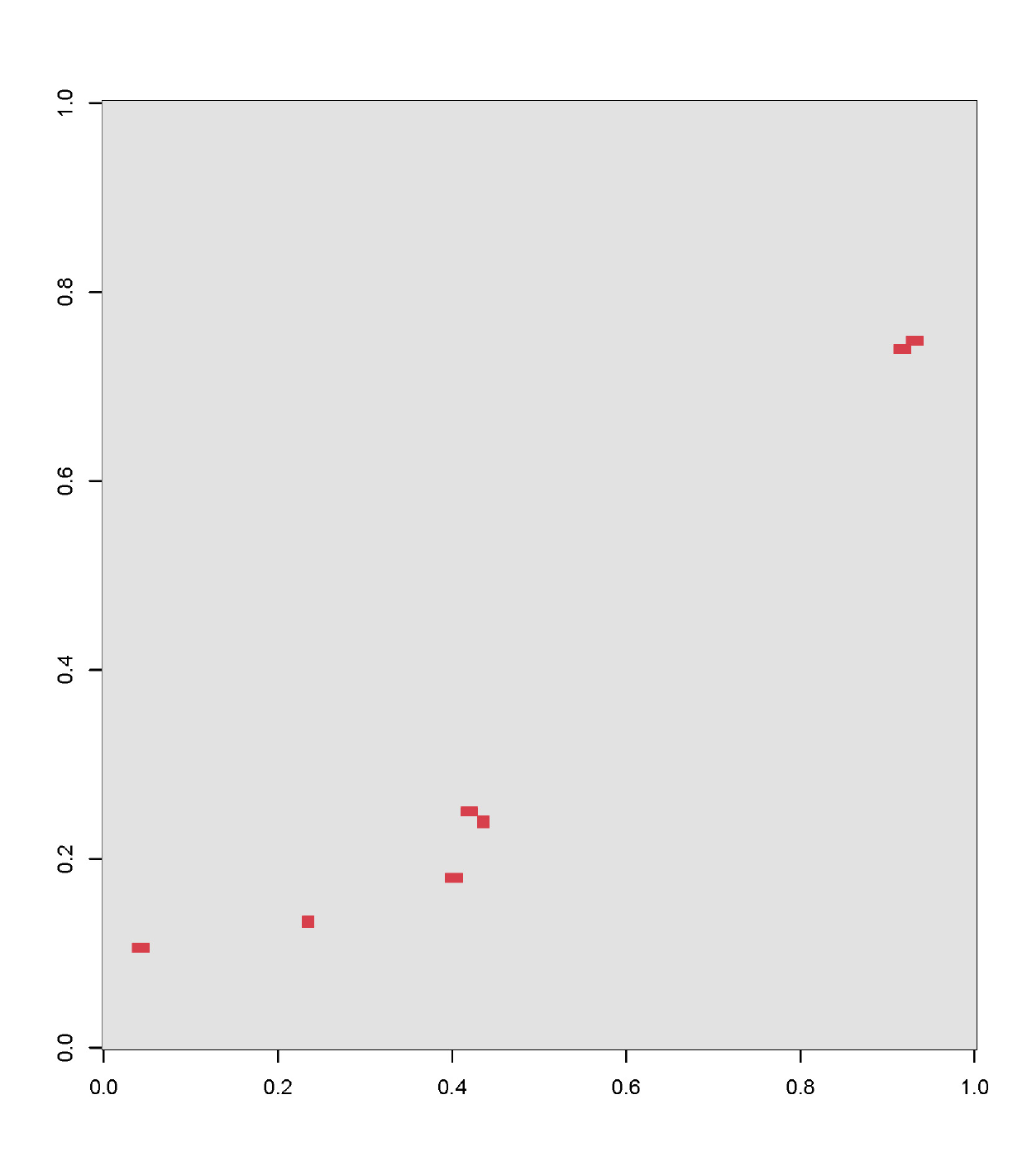}
			\includegraphics[width=0.32\textwidth]{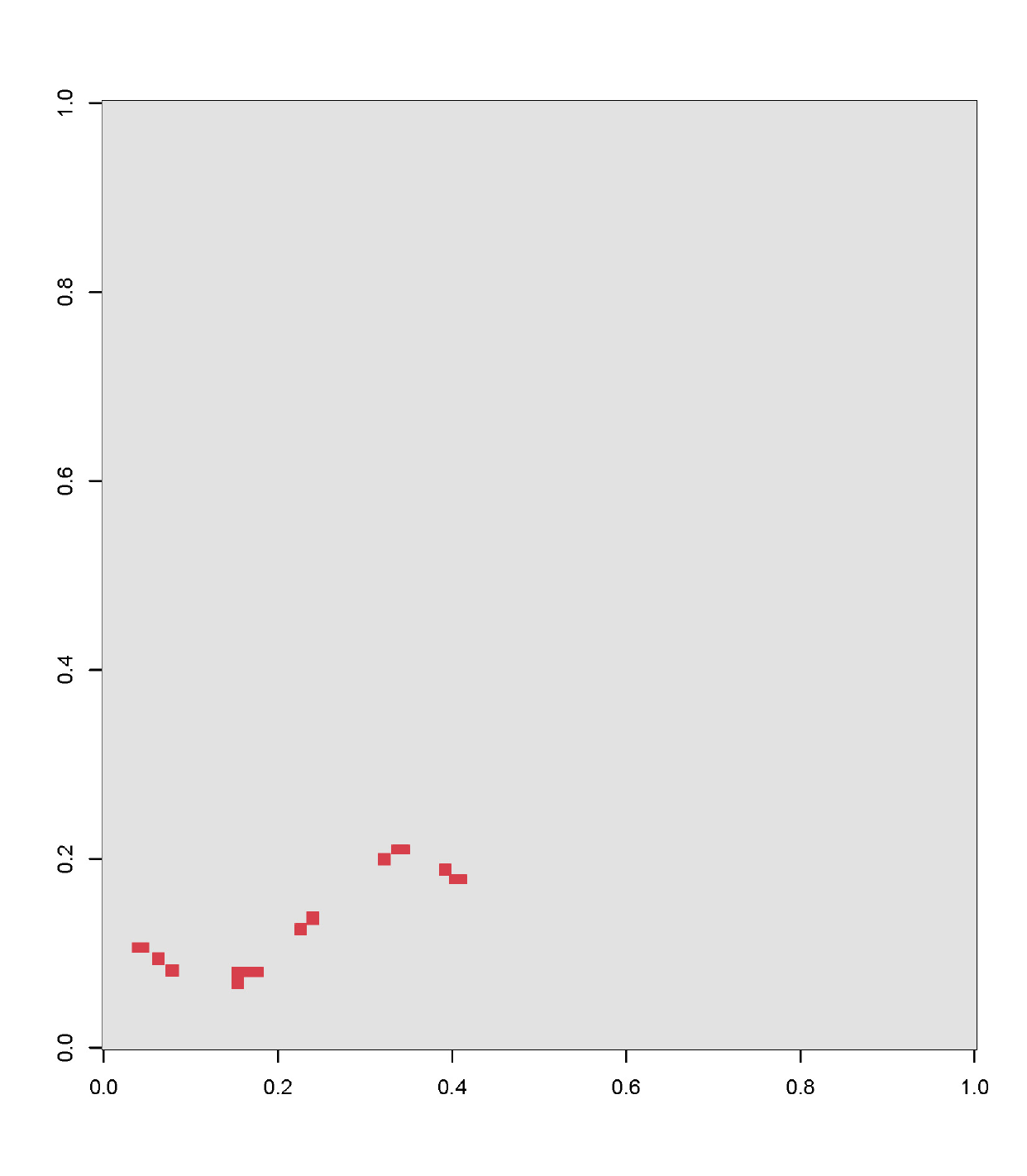}
			\caption{In the upper row, the heat maps from $ \FnB{1} $ applied to the middle panel of Figure \ref{figure_motivation} are displayed, where the statistic was maximized over $ 3 $, $ 6 $ and $ 9 $ angles, respectively. In the lower row, the corresponding thresholded statistics are displayed.} \label{figure_w1_stats_angles}
		\end{figure}

		\subsection{Algorithmic aspects}	
		A naive implementation with linear scan windows as discussed in this work has a computational complexity of $O(T^p\cdot T^p)$. In some situations an improvement can be obtained if means are calculated based on the difference with previously calculated means. At best as e.g.\ for correctly oriented rectangular scanning windows it can be calculated linearly in the number of voxels, i.e.\ $O(T^p)$ but this requires some additional memory of order $O(T^{p-1})$.
		
		The above naive computational complexity is actually given by $O(T^p\cdot N)$, where $N$ denotes the number of pixels in the scan window.
		To derive the theoretic results in this paper, we assumed that $N= c T^p (1+o(1))$ for some constant $c$, i.e.\ that the size of the scan window is linear in the image size. In practice, however, in particular for large data sets, this is not a realistic assumption and $N$ will be significantly smaller than this. Mathematically, this can be dealt with by considering scan windows which extend sublinearly relative to the image size. However, this requires a different mathematical theory and results in a different type of limit results.
		
		For the application discussed in this paper a significant amount of computational time can be saved by only calculating the statistics for pixels that are sufficiently dark.  For such pixels we can first calculate $\F{1}$  with a few angles. If the statistic is above a liberal threshold (e.g.\ the corresponding $80\%$-quantile) for at least one pixel, one can increase the number of angles in a second step. If then, the statistic is above a more conservative threshold (e.g.\ the corresponding $95\%$-quantile) for at least one pixel, one can then subtract the $\nB$-statistics (possibly again iteratively).
		For example, if we only consider the $ 10\% $ darkest pixels in the steel-fiber example, we can  reduce the runtime and simultaneously eliminate a significant amount of spurious pixels.

		\section{Conclusions}
In this paper, we have introduced and mathematically investigated a wide class of scan statistics under rather general assumptions. For practical purposes, these scan statistics have to be tailored to the particular application at hand by appropriately choosing and combining scan windows. We illustrate how to do this by the example of the detection of anomalies in large image data using 2D-slices of (semi-artificial) scans of different kinds of concrete. Indeed, these examples show that different types of anomalies are detected with different local contrasts tailored to the geometric properties of the anomaly of interest. In particular, this indicates how to adapt and develop scan statistics suitable for anomaly detection tasks that are different from the ones considered in this paper.  The theoretical framework investigated in this paper is general enough to cover many of these statistics.   

Indeed, we obtain statistical guarantees for tests based on our wide class of scan statistics from using thresholds associated to the corresponding limit distributions under rather general assumptions and in arbitrary dimension for stationary $M$-dependent random fields. The corresponding proofs are mathematically involved and the obtained central limit result as well as functional central limit theorem may also be of independent interest. On the downside, the theory is only valid for fixed size windows (with respect to rescaling the data to the unit hypercube), while in large data examples with relatively small windows asymptotic results obtained from shrinking windows (with respect to rescaling to the unit hypercube) may yield more realistic thresholds.  Such an analysis to obtain the corresponding extreme-value asymptotics for shrinking windows is left for future work.

		\section*{Acknowledgment}
		This work was supported by the German Federal Ministry of Education	and Research (BMBF) [grant number 05M2020 (DAnoBi)] and  by the Deutsche Forschungsgemeinschaft (DFG, German Research Foundation) - 314838170, GRK 2297 MathCoRe. The examples in the left and middle panel of Figure \ref{figure_motivation} are 2D slices of semi-synthetic 3D images published in \cite{KL}, where artificial cracks were added to 3D images of uncracked concrete by Franziska M\"usebeck (TU Kaiserslautern). This is open data under the CC BY license http://creativecommons.org/licenses/by/4.0/. 
  The example of steel fiber-reinforced concrete in the right panel of Figure \ref{figure_motivation} is one of the concrete samples investigated in \cite{Schuler}, imaged by Franz Schreiber at ITWM.
		\newpage
  \part*{Appendix}
		\appendix
\section{Proofs}\label{sec_proofs}

	For $ \gamma>0 $, $ \bs=(s_1,\ldots,s_p)^{\prime},$ $\bt=(t_1,\ldots,t_p)^{\prime}\in\R^p $ denote by
\begin{align*}
	\left[\bs,\bs+\bt\right)=&\ \left[s_1,s_1+t_1\right)\times\ldots\times\left[s_p,s_p+t_p\right)\\
	\left[\bs,\bs+\gamma\right)=&\ \left[s_1,s_1+\gamma\right)\times\ldots\times\left[s_p,s_p+\gamma\right),
\end{align*}
$ p- $dimensional hyperrectangles and -cubes, where $ \left[s_i,s_i+t_i\right):=\left[s_i+t_i,s_i\right) $ if $ t_i<0 $. Furthermore, for $c\in\R$ let $c\bs=(cs_1,\ldots,cs_p)^{\prime}$.

	 \begin{proof}[Proof of Theorem \ref{clt_m_dependence}]

	For the proof, we will replace $A$ by a  set $A_T$ which has an asymptotically negligible difference to $A$ and is a finite union of hypercubes. Let		
		\begin{align*}
		\inside_{A,T}=&\ \geschweift{\balpha\in\Z^p\Bigg|\rechtsoffen{\frac{\Tstar}{T}\balpha,\frac{\Tstar}{T}(\balpha+1)}\cap A\ne\emptyset}\\
		%		\label{indexset}\\
		%	 		\largeblock_{\balpha,T}=&\ \rechtsoffen{\frac{\Tstar\balpha}{T},\frac{\Tstar(\balpha+1)}{T}}\\
		A_{T}=&\ \bigcup_{\balpha\in\inside_{A,T}}\rechtsoffen{\frac{\Tstar}{T}\balpha,\frac{\Tstar}{T}(\balpha+1)}.\notag
	\end{align*}
		By construction, $A\subset A_T$ and it holds that for any $\bx\in A_T$, there exists $\by\in A$ such that
		$\left\|\bx-\by\right\|_{\infty}\le \ceil{T^{1/2}}/T\le 2/\sqrt{T}
		$.
		Therefore, it holds that $A_T\setminus A\subset A^{\left(2/\sqrt{T}\right)}\setminus A$, where $A^{(\gamma)}$ is as in Definition \ref{def_fattening}. By Lemma \ref{lem_jordan_assumptions} (c) 
		\[ 
		\abs{\geschweift{\bk\in\Z^p:\frac{\bk}{T}\in A_T\setminus A}}=o(T^p).
		\]
	With $S_A=S_A(\epsilon;0)$, $S_{A_T}=S_{A_T}(\epsilon;0)$ as in \eqref{eq_def_sums} it follows by \eqref{moments} and Markov's inequality that
\[ 
\frac{1}{T^{p/2}}\abs{S_A-S_{A_T}}=o_P(1).
\]
	 We will divide $ A_T $ into "big/large" independent blocks that are separated by "small/narrow" blocks which are asymptotically negligible. % and  use Lyapunov's condition from Theorem \ref{Lyapunov}. 
	To elaborate, let for $\balpha\in\inside_{A,T}$ 
	\begin{align*}
		&\bigblock_{\balpha,T}=\geschweift{\frac{\Tstar\balpha+\bl}{T}\Bigg| \bl\in\geschweift{0,\ldots, \Tstar-1-M }^p}
	\end{align*}
	be the \textbf{L}arge (big) blocks (which contain the information about the dependency within the random field) with $M$ as in Assumption \ref{ass_errors}. For $\bl=(l_1,\ldots,l_p)$  let $\max\bl=\max\geschweift{l_1,\ldots,l_p}$ and for $\balpha\in\inside_{A,T}$ let
	\begin{align*}
		&\smallblock_{\balpha,T}=\geschweift{\frac{\Tstar\balpha+\bl}{T}\Bigg| \bl\in\geschweift{0,\ldots, \Tstar-1 }^p,\ \max \bl\ge\Tstar-M }
	\end{align*}
	be the \textbf{N}arrow (small) blocks that separate the large blocks from each other. Furthermore, let \[ \largeblock_{\balpha,T}=\bigblock_{\balpha,T}\cup\smallblock_{\balpha,T}=\geschweift{\frac{\Tstar\balpha+\bl}{T}\Bigg| \bl\in\geschweift{0,\ldots, \Tstar-1 }^p}. \]
	The sums over the narrow blocks are asymptotically negligible: 
	Indeed, 
for $ \balpha\in\inside_{A,T} $, it holds by the mean value theorem that
\[  \abs{\smallblock_{\balpha,T}}=\abs{\largeblock_{\balpha,T}}-\abs{\bigblock_{\balpha,T}}=\Tstar^p-\left(\Tstar-M\right)^p=O\left(T^{(p-1)/2}\right).
\]
	Furthermore it holds by Lemma \ref{lem_measure_fattening}   that  $\lambda(A_T)=\lambda(A)(1+o(1))$. 
Since $\lambda(A_T)=\abs{\inside_{A,T}}\left(\Tstar/T\right)^p$	it follows that
$\abs{\inside_{A,T}}=\lambda(A)T^{p/2}(1+o(1))$.		
	Therefore it holds by \eqref{moments} and Markov's inequality that
	\begin{align*}
		&\ \frac{1}{T^{p/2}}\abs{\sum\limits_{\balpha\in \inside_{A,T}}\sum\limits_{\frac{\bk}{T}\in \smallblock_{\balpha,T}}\epsilon_{\bk}}=o_P(1).%\label{clt_smallblock_result}
	\end{align*}
	
	Since the large blocks  $\bigblock_{\balpha,T} $ are separated by the narrow blocks $\smallblock_{\balpha,T}$, it holds for $ \bk,\bl\in\Z^p $ such that $ \bk/T\in\bigblock_{\balpha,T} $, $ \bl/T\in\bigblock_{\balpha^*,T} $ with $ \balpha\ne\balpha^* $ that $ \left\|\bk-\bl\right\|_{\infty}>M $. Therefore, the sums $S_{L_{\balpha,T}}$ over the large blocks are independent.
	Thus, we can use the Lindeberg central limit theorem for triangular arrays in combination with Lyapunov's condition to prove asymptotic normality of the sum over the large blocks:

First note that the set of points $\bk/T\in L_{\balpha,T}$ with $\bl/T\in L_{\balpha,T}$ for $\left\|\bk-\bl\right\|_{\infty}<M$ has the form 
	\[ 
	\geschweift{\frac{\Tstar\balpha+\bl}{T}\Bigg| \bl\in\geschweift{M,\ldots, \Tstar-1-2M }^p}.
	\]
	Hence it holds by the mean value theorem for the set of all points $\bk/T\in L_{\balpha,T}$ that have at least one point $\bl/T\notin L_{\balpha,T}$ with $\left\|\bk-\bl\right\|_{\infty}<M$ that
	\begin{align*}
		&\ \abs{L_{\balpha,T}\setminus\geschweift{\frac{\Tstar\balpha+\bl}{T}\Bigg| \bl\in\geschweift{M,\ldots, \Tstar-1-2M }^p}}\\
		=&\left(\Tstar-M\right)^p-\left(\Tstar-3M\right)^p= O\left(T^{(p-1)/2}\right).
	\end{align*} 
Hence
	\begin{align*}
		&\ \Var{\sum\limits_{\frac{\bk}{T}\in \bigblock_{\balpha,T}}\epsilon_{\bk}}= \sum\limits_{\frac{\bk}{T}\in \bigblock_{\balpha,T}}\Var{\epsilon_{\bk}}+\sum\limits_{\stackrel{\frac{\bk}{T},\frac{\bl}{T}\in\bigblock_{\balpha,T}:}{ 1\le\left\|\bk-\bl\right\|_{\infty}\le M}}\Cov{\epsilon_{\bk}}{\epsilon_{\bl}}\\
		=&\ \sum\limits_{\frac{\bk}{T}\in \bigblock_{\balpha,T}}\left(\Var{\epsilon_{\bk}}+\sum\limits_{\bl: 1\le\left\|\bk-\bl\right\|_{\infty}\le M}\Cov{\epsilon_{\bk}}{\epsilon_{\bl}}\right)+O\left(T^{(p-1)/2}\right)
%		&+\sum\limits_{\frac{\bk}{T}\in \bigblock_{\balpha,T}\setminus\bigblock_{\balpha,T}}\left(\Var{\epsilon_{\bk}}+\sum\limits_{\stackrel{\frac{\bl}{T}\in\bigblock_{\balpha,T}:}{ 1\le\left\|\bk-\bl\right\|_{\infty}\le M}}\Cov{\epsilon_{\bk}}{\epsilon_{\bl}}\right)\\
		= \sigma^2T^{p/2}(1+o(1)).
	\end{align*}
	Therefore, 
	\begin{align*}
		s_{A,T}^2=\sum\limits_{\balpha\in \inside_{A,T}}\Var{\sum\limits_{\frac{\bk}{T}\in \bigblock_{\balpha,T}}\epsilon_{\bk}}=\sigma^2\lambda(A)T^p(1+o(1)).
	\end{align*}
	By \eqref{moments} and due to the fact that all $ \bigblock_{\balpha,T} $ have the same size, it holds for $ \delta>0 $ with suitable constants $ C,C^*>0 $
		\begin{align*}
			&\ \frac{1}{s_{A,T}^{2+\delta}}\sum\limits_{\balpha\in \inside_{A,T}}\EW{\abs{\sum\limits_{\frac{\bk}{T}\in \bigblock_{\balpha,T}}\epsilon_{\bk}}^{2+\delta}}\le\frac{1}{s_{A,T}^{2+\delta}}\sum\limits_{\balpha\in \inside_{A,T}} C\abs{\bigblock_{\balpha,T}}^{\frac{2+\delta}{2}}=\frac{1}{s_{A,T}^{2+\delta}}C \abs{\inside_{A,T}}\abs{\bigblock_{\balpha,T}}^{\frac{2+\delta}{2}}\\
			\le& C^*\frac{T^{p/2}\left(T^{p/2}\right)^{\frac{2+\delta}{2}}}{\left(T^p\right)^{\frac{2+\delta}{2}}}=C^*T^{(2+\delta/2)p/2-(2+\delta)p/2}=C^*T^{-p\delta/4}\to 0
		\end{align*}
	as $ T\to\infty $.
	By the Lyapunov condition, 
 %(from Theorem \ref{Lyapunov})\footnote{\textcolor{blue}{TODO: Remove when Thm. \ref{Lyapunov} is removed}}, 
 it follows that
		\begin{align*}
			&\ \frac{1}{T^{p/2}} \sum_{\balpha\in \inside_{A,T}}\sum\limits_{\frac{\bk}{T}\in \bigblock_{\balpha,T}}\epsilon_{\bk}
			=\frac{s_{A,T}}{T^{p/2}}\left(\frac{1}{s_{A,T}}
			\sum\limits_{\balpha\in \inside_{A,T}} \sum\limits_{\frac{\bk}{T}\in \bigblock_{\balpha,T}} \epsilon_{\bk}\right)\notag\\
			=&\ \sigma\sqrt{\lambda(A)}(1+o(1))\left(\frac{1}{s_{A,T}}
			\sum\limits_{\balpha\in \inside_{A,T}} \sum\limits_{\frac{\bk}{T}\in \bigblock_{\balpha,T}} \epsilon_{\bk}\right)\to\mathcal{N}\left(0,\sigma^2\lambda(A)\right), %\label{clt_bigblock_result}
		\end{align*}
		thus showing the assertion.
	\end{proof}

		\begin{proof}[Proof of Theorem \ref{danobi_convergence_map}]
		By Theorem \ref{convergence_D} in the appendix, we need to show (i) that for all $ n\in\N $, $ \bt_1,\ldots,\bt_n\in[0,1]^p $
		\[ 
		\left(\Hoelder\left(\Sprocess_T(\bt_1)\right),\ldots,\Hoelder\left(\Sprocess_T(\bt_n)\right)\right)^{\prime}\dto\left(\Hoelder(Z(\bt_1)),\ldots,\Hoelder(Z(\bt_n))\right)^{\prime}
		\]
		and (ii) for all $ x>0 $, it holds that
		\[ 
		\lim_{\delta\to 0}\limsup_{T\to\infty}\Prob{\supp{\left\|\bs-\bt\right\|_{\infty}<\delta}\abs{\Hoelder\left(\Sprocess_T(\bs)\right)-\Hoelder\left(\Sprocess_T(\bt)\right)}\ge x}=0.
		\]
		Proof of (i): 
		By Lemma \ref{ass21_implies_jordan}, $A_1,\ldots,A_P$ are Jordan-measurable. 
		Hence by Lemma \ref{convergence_fidis} in the appendix it holds that
		\[
		\left(\Sprocess_T(\bt_1),\ldots\Sprocess_T(\bt_n)\right)\dto\left(Z(\bt_1),\ldots,Z(\bt_n)\right). 
		\]
		Denote for $n\in\N $, $ \bt_1,\ldots,\bt_n\in[0,1]^p $ by $ \pi_{\bt_1,\ldots,\bt_n}: \left(\mathcal{D}\left([0,1]^p\right)\right)^P\to \left(\R^{P}\right)^n $, $ \pi_{\bt_1,\ldots,\bt_n}(f)=(f(\bt_1),\ldots,f(\bt_n))^{\prime} $
		the projection mapping from $ \left(\mathcal{D}\left([0,1]^p\right)\right)^P $ to $ \left(\R^{P}\right)^n $. 
		Since the U-topology by \cite{Wichura} is induced by the maximum norm on $ \mathcal{D}([0,1]^p) $, it follows by Lemma \ref{projection_continuous}  in the appendix  that $ \pi_{\bt_1,\ldots,\bt_n}(\cdot) $ is continuous if $ \left(\mathcal{D}\left([0,1]^p\right)\right)^P $ is equipped with a proper metric (which we will assume here).   Furthermore, as $ \Hoelder $ is continuous, $ \Hoelder^{(n)}:\left(\R^P\right)^n\to\R^n $, $ \Hoelder^{(n)}(\bt_1,\ldots,\bt_n)=(\Hoelder(\bx_1),\ldots,\Hoelder(\bx_n)) $ is continuous as well. Since $ \Hoelder^{(n)}\circ\pi_{\bt_1,\ldots,\bt_n} $ as a composition of continuous functions is continuous, it follows by the continuous mapping theorem that
		\begin{align*}
			&\ \left(\Hoelder\left(\Sprocess_T(\bt_1)\right),\ldots,\Hoelder\left(\Sprocess_T(\bt_n)\right)\right)^{\prime}=
			\left(\Hoelder^{(n)}\circ \pi_{\bt_1,\ldots,\bt_n}\right)\left(\Sprocess_T(\bt)\right)\\ &\ \dto\left(\Hoelder^{(n)}\circ \pi_{\bt_1,\ldots,\bt_n}\right)\left(Z(\bt)\right)= \left(Z(\bt_1),\ldots,Z(\bt_n)\right)^{\prime}
		\end{align*}
		as $ T\to\infty. $\\\\
		Proof of (ii): 
  %Recall that uniform continuity means that for all $\varepsilon>0$ there exists $\delta=\delta(\varepsilon)>0$ such that for every $\bx,\by\in\R^P$ with $\left\|\bx-\by\right\|_{\infty}<\delta $ it holds that $\abs{G(\bx)-G(\by)}<\varepsilon $.
  As $ \Hoelder $ is uniformly continuous, it holds by the definition of uniform continuity that for $ x>0 $, there exists $y=y(x)$ such that for $\ba_1,\ba_2\in\R^P$ with $\abs{G(\ba_1)-G(\ba_2)}\ge x$ it holds that $\left\|\ba_1-\ba_2\right\|_{\infty}\ge y$. Hence
		\begin{align*}
			\Prob{\supp{\left\|\bs-\bt\right\|_{\infty}<\delta}\abs{\Hoelder\left(\Sprocess_T(\bs)\right)-\Hoelder\left(\Sprocess_T(\bt)\right)}\ge x}\le&\  	\Prob{\supp{\left\|\bs-\bt\right\|_{\infty}<\delta}\left\|\Sprocess_T(\bs)-\Sprocess_T(\bt)\right\|_{\infty}\ge y}
		\end{align*}
		and the assertion follows immediately from Lemma \ref{modulus_of_continuity}  in the appendix 
	\end{proof}
	
			\begin{proof}[Proof of Remark \ref{rem_rectangles}]
			Let $ \bs=(s_1,\ldots,s_p)^{\prime}\in[0,1]^p $.  Recall from \eqref{eq_def_sums} that
			\[ 
			\danobisum_{\scanset}\left(\epsilon;\floor{\bs}_T\right)=\sum\limits_{\frac{\bk}{T}\in\scanset(\floor{\bs}_T)}\epsilon_{\bk}.
			 \]
			Since by definition
			\[ 
				\scanset(\floor{\bs}_T)=\linksoffen{\floor{\bs}_T-\ba,\floor{\bs}_T+\ba},
				\]
				it holds that 
				\begin{align}
				\danobisum_{\scanset}\left(\epsilon;\floor{\bs}_T\right)=&\ \sum_{k_1=\floor{\left(\floor{s_1}_T-a_1\right)T}+1}^{\floor{\left(\floor{s_1}_T+a_1\right)T}}\cdots \sum_{k_p=\floor{\left(\floor{s_p}_T-a_p\right)T}+1}^{\floor{\left(\floor{s_p}_T+a_p\right)T}}\epsilon_{k_1,\ldots,k_p}\notag \\
				=&\ \sum_{(d_1,\ldots,d_p)^{\prime}\in\geschweift{0,1}^p} (-1)^{\sum_i d_i} \sum_{k_1=1}^{\floor{\left(\floor{s_1}_T+(-1)^{d_1}a_1\right)T}}\cdots\sum_{k_p=1}^{\floor{\left(\floor{s_p}_T+(-1)^{d_p}a_p\right)T}} \epsilon_{k_1,\ldots,k_p}, \label{sum_square}
				\end{align}
				where we can see the last equality as follows: We need to distinguish between two cases: If $ (l_1,\ldots,l_p)^{\prime}\in\N^p $ with $ \floor{\left(\floor{s_i}_T-a_i\right)T}<l_i\le \floor{\left(\floor{s_i}_T+a_i\right)T} $ for all $ i $, then $ \epsilon_{l_1,\ldots,l_p} $ is only counted for $ (d_1,\ldots,d_p)^{\prime}=(0,\ldots,0)^{\prime} $ in \eqref{sum_square}.\\
				If there exists some $j=1,\ldots,p  $ such that $ l_j\le \floor{\left(\floor{s_j}_T-a_j\right)T} $, then $ \epsilon_{l_1,\ldots,l_p} $ is counted in both $ (d_1,\ldots,d_{j-1},0,d_{j+1},\ldots,d_p)^{\prime} $ and  $ (d_1,\ldots,d_{j-1},1,d_{j+1},\ldots,d_p)^{\prime} $ with $ d_i\in\geschweift{0,1} $ appropriate for $ i\ne j $. Since $ (-1)^{\sum_{i\ne j}d_i}=-(-1)^{1+\sum_{i\ne j}d_i} $, $ \epsilon_{l_1,\ldots,l_p} $ gets canceled out in all of these summands. This telescoping sum is also in the spirits of \cite{Nelsen}, Def.~2.10.1. for the volume of $ p $-dimensional hyperrectangles.\\ Furthermore, since $ s_i-\floor{s_i}_T=s_i-\floor{s_iT}/T\le 1/T $, it holds for $ d\in\geschweift{0,1} $ that $ \abs{\floor{(s_i+(-1)^d a_i)T}-\floor{(\floor{s_i}_T+(-1)^d a_i)T}}\le 2 $. Thus, it holds by \eqref{FCLT2} for all $ (d_1,\ldots,d_p)^{\prime}\in\geschweift{0,1}^p $ uniformly in $ \bs $ that 
					\begin{align*}
					&\ \left|\frac{1}{T^{p/2}}\sum_{k_1=1}^{\floor{\left(\floor{s_1}_T+(-1)^{d_1}a_1\right)T}}\cdots\sum_{k_p=1}^{\floor{\left(\floor{s_p}_T+(-1)^{d_p}a_p\right)T}} \epsilon_{k_1,\ldots,k_p}\right.\\*
					&\ \left.-\frac{1}{T^{p/2}}\sum_{k_1=1}^{\floor{\left(s_1+(-1)^{d_1}a_1\right)T}}\cdots\sum_{k_p=1}^{\floor{\left(s_p+(-1)^{d_p}a_p\right)T}} \epsilon_{k_1,\ldots,k_p}\right|=o_P(1).
				\end{align*}		 
			The assertion follows by \eqref{FCLT} and the continuous mapping theorem.
			\end{proof}

	\begin{proof}[Proof of Theorem \ref{power_one}]
		It is sufficient to show that $ \F1_T(\bs_0)\pto\infty $ as $ T\to\infty $. Furthermore, because only differences of means are considered, we can assume w.l.o.g.\ that $\mu_0=0$. It holds that
\begin{align*}
	\hspace*{-0.5cm}		 \F1_T(\bs_0)=&\ \max_{i=1\ldots,P} \frac{1}{\sigma}\min\geschweift{\danobimean^{(12,\alpha_i)}(\bs_0),\danobimean^{(13,\alpha_i)}(\bs_0)} \ge \frac{1}{\sigma}\min\geschweift{\danobimean^{(12,\alpha_0)}(\bs_0),\danobimean^{(13,\alpha_0)}(\bs_0)}.
\end{align*}
Since for $ \alpha_0 $, the "inner strip" $ \scanset^{(1,\alpha_0)} $ has the same angle as the rectangle $ \anomaly $ and since $ w<h $, it holds that $ \anomaly\cap \scanset^{(2,\alpha_0)}(\bs_0)=\anomaly\cap \scanset^{(3,\alpha_0)}(\bs_0)=\emptyset  $. Furthermore, as $ w<h $ and $ l\ge d $, $ \scanset^{(1,\alpha_0)}\cap \anomaly $ is an "inner strip" of a circle with diameter $ d $ and width $ w $. 
Therefore, it holds by Theorem \ref{danobi_convergence_map}, Lemma \ref{lem_order} and Lemma \ref{lem_jordan_assumptions} (a) for $j=1,2,3$ that
\begin{align*}
	\danobimean_{\scanset_j}(Y;\floor{\bs}_T) 	=&\  \frac{T}{\abs{\geschweift{\bk\in\Z^p:\frac{\bk}{T}\in \scanset_j\left(\floor{\bs}_T\right)}}}
	\sum_{ \frac{\bk}{T}\in \scanset_j(\floor{\bs}_T)} \left(-\mathds{1}_{\geschweift{\bk/T\in \anomaly}}\delta_T+\epsilon_{\bk}\right)\\
	=&\ -\delta_TT\frac{\abs{\geschweift{\bk\in\Z^p:\frac{\bk}{T}\in \scanset_j\left(\floor{\bs}_T\right)\cap \anomaly}}}{\abs{\geschweift{\bk\in\Z^p:\frac{\bk}{T}\in \scanset_j\left(\floor{\bs}_T\right)}}}
	+O_P(1)\\
	=&\ -\delta_TT\frac{T^2\lambda(\scanset_j(\bs)\cap \anomaly)+o\left(T^2\right)}{T^2\lambda(\scanset_j(\bs))+o\left(T^2\right)}+O_P(1)\\
	=&\ -\delta_TT\frac{\lambda(\scanset_j(\bs)\cap \anomaly)}{\lambda(\scanset_j(\bs))}+O_P\left(1\right).
\end{align*}
Therefore,
\begin{align*}
	\danobimean^{(12,\alpha_0)}(\bs_0)=&\ \delta_TT\left(\frac{\lambda\left(\scanset^{(1,\alpha_0)}\left(\bs_0\right)\cap \anomaly\right)}{\lambda(\scanset^{(1,\alpha_0)})}-
	\frac{\lambda\left(\scanset^{(2,\alpha_0)}\left(\bs_0\right)\cap \anomaly\right)}{\lambda(\scanset^{(2,\alpha_0)})}
	\right)+O_P(1)\\
	=&\ \delta_TT\frac{\lambda\left(\scanset^{(1,\alpha_0)}\left(\bs_0\right)\cap \anomaly\right)}{\lambda(\scanset^{(1,\alpha_0)})}+O_P(1)\\
	\intertext{and similarly}
	\danobimean^{(13,\alpha_0)}(\bs_0)
	=&\ \delta_TT\frac{\lambda\left(\scanset^{(1,\alpha_0)}\left(\bs_0\right)\cap \anomaly\right)}{\lambda(\scanset^{(1,\alpha_0)})}+O_P(1).
\end{align*}
Therefore, it holds that
\begin{align*}
	\F1_T(\bs_0)\ge\delta_TT\frac{\lambda\left(\scanset^{(1,\alpha_0)}\left(\bs_0\right)\cap \anomaly\right)}{\lambda(\scanset^{(1,\alpha_0)})}+O_P(1),
\end{align*}
As $ \frac{\lambda\left(\scanset^{(1,\alpha_0)}\left(\bs_0\right)\cap \anomaly\right)}{\lambda\left(\scanset^{(1,\alpha_0)}\right)}>0 $ and does not depend on $ T $, it follows that
\[ 
\F1_T(\bs_0)\pto\infty
 \]
 if $ \delta_TT\to\infty $ for $ T\to\infty $.
	\end{proof}

%	\part{Appendix}

\section{Size and power results for Section \ref{sec_danobi_statistics}}\label{sec_size_power}

 \begin{Lemma}\label{lem_order}
		Let model \eqref{eq_model} be given with errors fulfilling Assumption~\ref{ass_errors}.
		Then it holds uniformly in $ \bs\in[d/2,1-d/2]^2 $ that for all $ i=1,\ldots,5 $,
		\begin{align*}
			\danobimean_{\scanset^{(i,\alpha)}}\left(\epsilon;\floor{\bs}_T\right)=&\
			\frac{1}{T\lambda\left(\scanset^{(i,\alpha)}\right)} \danobisum_{\scanset^{(i,\alpha)}}\left(\epsilon;\floor{\bs}_T\right)+o_P(1).
		\end{align*}
  				\begin{proof}
		As by Remark \ref{domain_fclt}, the FCLT from
		Theorem \ref{danobi_convergence_map} can also be applied on $[d/2,1-d/2]^2$, it holds that 
		\[ 
		\supp{\bs\in[d/2,1-d/2]^2} \abs{\frac{1}{T}\danobisum_{\scanset}\left(\epsilon;\floor{\bs}_T\right)}=O_P(1).
		\]
		By Lemma \ref{lem_jordan_assumptions} (a), it holds that
		\[ 
		\abs{T^2\lambda(\scanset^{(i,\alpha)})-\abs{\geschweift{\bk\in\Z^p:\frac{\bk}{T}\in \scanset^{(i,\alpha)}\left(\floor{\bs}_T\right)}}}=o\left(T^2\right).
		\]
		Therefore, it holds that
		\begin{align*}
			&\ \supp{\bs\in[d/2,1-d/2]^2}\abs{\danobimean_{\scanset}\left(\epsilon;\floor{\bs}_T\right)-\frac{1}{T\lambda\left(\scanset^{(i,\alpha)}\right)} \danobisum_{\scanset}\left(\epsilon;\floor{\bs}_T\right)}\\ \le&\  \supp{\bs\in[d/2,1-d/2]^2} \abs{\frac{T^2}{\abs{\geschweift{\bk\in\Z^p:\frac{\bk}{T}\in \scanset^{(i,\alpha)}\left(\floor{\bs}_T\right)}}}-\frac{1}{\lambda\left(\scanset^{(i,\alpha)}\right)} } \supp{\bs\in[d/2,1-d/2]^2} \abs{\frac{1}{T}\danobisum_{\scanset}\left(\epsilon;\floor{\bs}_T\right)}\\
			=&\ \supp{\bs\in[d/2,1-d/2]^2} \frac{\abs{T^2\lambda\left(\scanset^{(i,\alpha)}\right)-\abs{\geschweift{\bk\in\Z^p:\frac{\bk}{T}\in \scanset^{(i,\alpha)}\left(\floor{\bs}_T\right)}}}}{\lambda\left(\scanset^{(i,\alpha)}\right)\abs{\geschweift{\bk\in\Z^p:\frac{\bk}{T}\in \scanset^{(i,\alpha)}\left(\floor{\bs}_T\right)}}} \supp{\bs\in[d/2,1-d/2]^2} \abs{\frac{1}{T}\danobisum_{\scanset}\left(\epsilon;\floor{\bs}_T\right)}\\
			=&\ o\left(\frac{T^2}{T^2}\right)\cdot O_P(1)=o_P\left(1\right).
		\end{align*}
	\end{proof}

	\end{Lemma}
	
	\begin{Theorem} \label{theorem_danobi_size}
		Let model \eqref{eq_model} be given under the null hypothesis, i.e.\ with $\mu_{\bk,T}=\mu_0$ for all $ \bk,T $ and some (unknown) $\mu_0\in \R$, and  errors fulfilling Assumption~\ref{ass_errors}.
		Then there exists a $ 5P $-dimensional centered Gaussian process
		\[
		\left(\left(W_{1,1}(\bs),\ldots,W_{1,P}(\bs),\ldots,W_{5,1}(\bs),\ldots,W_{5,P}(\bs)\right)^{\prime}\right)_{\bs\in[d/2,1-d/2]^2}
		\]
		with

		\begin{align*}
			\Cov{W_{i,j}(\bs)}{W_{k,l}(\bt)}=\frac{\lambda\left(\scanset^{(i,\alpha_j)}(\bs)\cap\scanset^{(k,\alpha_l)}(\bt)\right)}{\lambda \left(\scanset^{(i,\alpha_j)}\right)\lambda\left(\scanset^{(k,\alpha_l)}\right)}%\label{covariance}
		\end{align*}
  such that, with
		\begin{align*}
			\Gamma_{\F1}(\bs)=&\ \max_{i=1,\ldots,P}\min\geschweift{W_{2,i}(\bs)-W_{1,i}(\bs),W_{3,i}(\bs)-W_{1,i}(\bs)},\\
   			\Gamma_{\FnB2}(\bs)=&\ \max\left\{0,\max_{i=1,\ldots,P}\min\geschweift{\abs{W_{2,i}(\bs)-W_{1,i}(\bs)},\abs{W_{3,i}(\bs)-W_{1,i}(\bs)}}\right.\\
			&\ \left. - \max_{i=1,\ldots,P} \abs{W_{4,i}(\bs)-W_{5,i}(\bs)}\right\}\\
   		\Gamma_{\FnB1}(\bs)=&\ \max\left\{0,\max_{i=1,\ldots,P}\min\geschweift{W_{2,i}(\bs)-W_{1,i}(\bs),W_{3,i}(\bs)-W_{1,i}(\bs)}\right.\\
		&\ \left. - \max_{i=1,\ldots,P} \abs{W_{4,i}(\bs)-W_{5,i}(\bs)}\right\}.
		\end{align*}
		it holds that
		\begin{align*}
		\text{(a)}\qquad&\	\left(\F1_T(\bs)\right)_{\bs\in[d/2,1-d/2]^2}\wto\left(\Gamma_{\F1}(\bs)\right)_{\bs\in[d/2,1-d/2]^2}\\
  			\text{(b)}\qquad&\	\left(\FnB2_T(\bs)\right)_{\bs\in[d/2,1-d/2]^2}\wto\left(\Gamma_{\FnB2}(\bs)\right)_{\bs\in[d/2,1-d/2]^2}\\
       			\text{(c)}\qquad&\	\left(\FnB1_T(\bs)\right)_{\bs\in[d/2,1-d/2]^2}\wto\left(\Gamma_{\FnB1}(\bs)\right)_{\bs\in[d/2,1-d/2]^2}
		\end{align*}
		in $ \cD([d/2,1-d/2]^2) $ as $ T\to\infty $.
  		\begin{proof}
		Since the $ \scanset^{(i,\alpha)} $ are convex bounded sets with positive Lebesgue measure, they fulfill Assumption \ref{ass_sets} by Lemma \ref{convex_sets}. Furthermore, since maximum, minimum, and subtraction are Lipschitz continuous, as can be seen by simple calculations, the assertions follow by Theorem \ref{danobi_convergence_map}, Remark \ref{domain_fclt}, Lemma \ref{lem_order} and Slutzky's theorem.
	\end{proof}
	\end{Theorem}

	\begin{Theorem}\label{power_one_fnb2}
		In the situation of Theorem~\ref{power_one}, where we can drop the assumption of positivity of $\delta_T$,
		it holds that
		\begin{align*}
			\text{(a)}\qquad
			&\ \maxx{\bs\in[d/2,1-d/2]^2} \F2_T(\bs)\ge\F2_T(\bs_0)\pto\infty,\\
			\text{(b)}\qquad
			&\ \nB_T(\bs_0)=O_P(1),\\
			\text{(c)}\qquad
			&\ \maxx{\bs\in[d/2,1-d/2]^2}\FnB{2}_T(\bs)\ge\FnB{2}_T(\bs_0)\pto\infty,
		\end{align*}	
		if $ \abs{\delta_T}T\to\infty $, as $ T\to\infty $.
  \begin{proof}
(a) Analogous to the proof of Theorem \ref{power_one}.\\\\
(b) 	Because of symmetry, it holds for every angle $ \alpha_1,\ldots,\alpha_P $ that $ \lambda\left(\scanset^{(4,\alpha_0)}\left(\bs_0\right)\cap \anomaly\right)=\lambda\left(\scanset^{(5,\alpha_0)}\left(\bs_0\right)\cap \anomaly\right) $ and $ \lambda\left(\scanset^{(4,\alpha_0)}\right)=\lambda\left(\scanset^{(5,\alpha_0)}\right) $. 
Since it holds analogously to the proof of Theorem \ref{power_one} for $ j=4,5 $ that 
\[ 	\danobimean_{\scanset_j}(Y;\floor{\bs}_T) = \delta_TT\frac{\lambda(\scanset_j(\bs)\cap \anomaly)}{\lambda(\scanset_j(\bs))}+O_P\left(1\right),
 \]
 it holds uniformly in $ i $ that
\begin{align*}
	\danobimean^{(45,\alpha_i)}(\bs_0)=&\ O_P(1).
\end{align*}
and thus
\[ 
\nB_T(\bs_0)=O_P(1).
 \]
 (c) Since
\begin{align*}
	\FnB{2}_T(\bs_0)=\max\geschweift{\F2_T(\bs_0)-\nB_T(\bs_0),0},
\end{align*}
the assertion follows immediately from (a) and (b).
\end{proof}
	\end{Theorem}
	
	\begin{Theorem}\label{power_one_fnb1}
		In the situation of Theorem~\ref{power_one}
		it holds that
		\begin{align*}
			&\ \maxx{\bs\in[d/2,1-d/2]^2}\FnB{1}_T(\bs)\ge\FnB{1}_T(\bs_0)\pto\infty,
		\end{align*}	
		if $ \delta_TT\to\infty $, as $ T\to\infty $.
  \begin{proof}
 Since
\begin{align*}
	\FnB{1}_T(\bs_0)=\max\geschweift{\F1_T(\bs_0)-\nB_T(\bs_0),0},
\end{align*}
the assertion follows immediately from Theorem \ref{power_one} and \ref{power_one_fnb2} (b).
\end{proof}

	\end{Theorem}

	\section{Limit theorems for $ M $-dependent r.v.'s} \label{sec_limit_thms}

The following Lemma shows that the concept of $ M $-dependence does not depend on the chosen norm as long as equivalent norms are being considered. For simplicity, $ M $ will always refer to the constant in the supremum norm in the following proofs.
		\begin{Lemma}\label{lem_equivalence_M_dependence}
			Let $ \left\|\cdot\right\|_1 $ and $ \left\|\cdot\right\|_2 $ be arbitrary norms for which there exists some $ C>0 $ such that $ C\left\|\cdot\right\|_2\ge\left\|\cdot\right\|_1 $.
			Let $ X=\left(X_{\bt}\right)_{\bt\in S} $ be a stochastic process that is $ M $-dependent w.r.t.~$ \left\|\cdot\right\|_1 $.  Then, $ X $ is $M/C $-dependent w.r.t.~$ \left\|\cdot\right\|_2 $.
			\begin{proof}
				Let $ \bs,\bt $ be such that $ \left\|\bt-\bs\right\|_1>M $. Then it holds that
				\[ 
				M<\left\|\bt-\bs\right\|_1\le C\left\|\bt-\bs\right\|_2.
				 \]
				 Therefore, $ \left\|\bt-\bs\right\|_2>M/C $.
			\end{proof}
		\end{Lemma}

	\begin{Lemma}\label{lem_partitions}
		Let $ \left(X_{\bt}\right)_{\bt\in\Z^p} $ be a sequence of $ M $-dependent random variables and let $ \bl\in\geschweift{1,\ldots,M+1}^p $. Then $ \left(X_{\bl+\boldsymbol{\alpha}(M+1)}\right)_{\boldsymbol{\alpha}\in\Z^p} $ is a sequence of independent random variables.
		\begin{proof}
			Let $ \bs=\bl+\boldsymbol{\alpha}_{\bs}(M+1) $, $ \bt=\bl+\boldsymbol{\alpha}_{\bt}(M+1) $  and let $ \boldsymbol{\alpha}_{\bs}=(\alpha_{1,\bs},\ldots,\alpha_{p,\bs}) $, $ \boldsymbol{\alpha}_{\bt}=(\alpha_{1,\bt},\ldots,\alpha_{p,\bt}) $. If $ \bs\ne\bt $, there exists $ i=1,\ldots,p $ such that $ \alpha_{i,\bs}\ne\alpha_{i,\bt} $. Therefore,
			\[ 
			\left\|\bs-\bt\right\|_{\infty}=(M+1)\left\|\boldsymbol{\alpha}_{\bs}-\boldsymbol{\alpha}_{\bt}\right\|_{\infty}\ge M+1>M.
			 \]
		\end{proof} 
	\end{Lemma}

\begin{Remark} \label{rem_partitions}
Lemma \ref{lem_partitions}   allows a partitioning of $ \Z^p $ into finitely many sets $ D_{\bl}=\geschweift{\bl+\boldsymbol{\alpha}(M+1)|\boldsymbol{\alpha}\in\Z^p} $ on which $ \left(X_{\bt}\right)_{\bt\in D_{\bl}} $ are sequences of independent random variables.
This  shows that moment bounds for sums of independent random variables immediately carry over to the corresponding results for $M$-dependent random variables due to $|\sum_{i=1}^M a_i|^r\le (M \max_{i}|a_i|)^r\le M^r \sum_{i=1}^M| a_i|^r$.\\
 In particular, for an $M$-dependent sequence $ \left(\epsilon_{\bk}\right)_{\bk\in\Z^p} $ with $ \EW{\epsilon_{\bk}}=0 $ and any $ \scanset_n\subset\R^p $ such that $\abs{\geschweift{\bk\in\Z^p:\bk\in \scanset_n}}=n$, it holds for any $r\ge 2$ that there exists $ C_r>0 $ such that 
\begin{align}
	\EW{\abs{\sum_{\bk\in \scanset_n}\epsilon_{\bk}}^r}\le C_rn^{r/2}.\label{moments}
\end{align}
The corresponding result for i.i.d.\ random variables can e.g.\ be found in \cite{LinBai}, 9.4, equation (45) (see also \cite{Yokoyama} for general result for mixing sequences).

\end{Remark}

\begin{Lemma}\label{convergence_fidis}
	Let $ \scanset_1,\ldots,\scanset_P\subset[0,1]^p $ be compact Jordan-measurable sets with $\lambda(A_i)>0$ for all $i$.
Let $ \left(\epsilon_{\bk}\right)_{\bk\in\Z^p} $ fulfill Assumption \ref{ass_errors} and let $ \sigma>0 $ be as in \eqref{long_run_variance}. 
Let $ \danobisum_{\scanset}(\floor{\bs}_T)= \danobisum_{\scanset}(\epsilon;\floor{\bs}_T) $ be as in \eqref{eq_def_sums} and let
\[ 
\Sprocess_T(\bs)=\left(\frac{1}{T^{p/2}}\danobisum_{\scanset_1}(\floor{\bs}_T),\ldots,\frac{1}{T^{p/2}}\danobisum_{\scanset_P}(\floor{\bs}_T)\right)^{\prime}.
\]
Then there exists a $ P $-dimensional centered Gaussian process \\ $ (\mathbb{W}(\bs))_{\bs\in[0,1]^p}= ((W_1(\bs),\ldots,W_{P}(\bs))^{\prime})_{\bs\in[0,1]^p} $  with
	\[ 
	\Cov{W_i(\bs)}{W_j(\bt)}=\sigma^2\lambda\left(\scanset_i(\bs)\cap\scanset_j(\bt)\right)
	\]
	such that for all $ n\in\N $, $ \bt_1,\ldots,\bt_n\in[0,1]^p $ it holds that
	\[ 
	\left(\Sprocess_T(\bt_1),\ldots\Sprocess_T(\bt_n)\right)\dto\left(\mathbb{W}(\bt_1),\ldots,\mathbb{W}(\bt_n)\right).
	\]
\begin{proof}
	By the theorem of Cram\'{e}r-Wold (compare e.g.\ \cite{Durrett}, Theorem 3.10.6), it is  sufficient to show that for arbitrary $ a_{1,1},\ldots,a_{n,1}, \ldots,$ $a_{1,P},\ldots,a_{n,P}\in\R $, it holds that
\begin{align} \label{eq_CramWold}
\sum_{i=1}^{n}\sum_{j=1}^{P}a_{i,j}\frac{1}{T^{p/2}}\danobisum_{\scanset_j}(\floor{\bt_i}_T)
\dto N\left(\mathbf{0},\Var{\sum_{i=1}^{n}\sum_{j=1}^{P}a_{i,j} W_j(\bt_i) }\right).
\end{align}
Denote  by $ \scanset^1=\scanset $, $ \scanset^0=\scanset^{\prime}$, where $\scanset^{\prime} $ is the complement of $\scanset$ in $[0,1]^p$. 
	For ease of notation we are also renumbering the $\scanset_{j}(\floor{\bt_i}_T)  $ successively with a slight abuse of notation in the following way:
	For $ i=1,\ldots,n $, $ j=1,\ldots,P $, let
	$ \scanset_{(j-1)n+i-1}=\scanset_{j}(\floor{\bt_i}_T) $ and let $ m=nP $.
	\\
	For $ l=1,\ldots,2^{m}-1 $ let $ l=\sum_{i=0}^{m-1}l_i2^i $ be the unique binary representation of $l $ with $ l_i\in\geschweift{0,1} $ and corresponding mutually exclusive sets
	\begin{align*}
		%		l_i=&\ \mathds{1}_{\geschweift{l\modulo 2^{i+1}\ge2^i}}\\
		D_l=D_{\sum\limits_{i=0}^{m-1}l_i2^i}=&\ \bigcap_{i=0}^{m-1}\scanset_i^{l_i}.
	\end{align*}
	Since we are using the unique binary representation for $ l $, the sets $ D_l $ are pairwise disjoint: For $ k\ne l $, there exists $ i=0,\ldots,m-1 $ such that $ k_i\ne l_i $. It follows that
	$ \scanset_i^{k_i}\cap \scanset_i^{l_i}=\emptyset $ and therefore $ D_k\cap D_l=\emptyset $. 
Furthermore, it holds that
\begin{align}
		\sum\limits_{l: l_i=1} D_l=\scanset_i.
		\label{union}
\end{align}
We first show that $ \scanset_i\subset \sum_{l: l_i=1} D_l $. Let  $ \bx\in\scanset_i $. For $ k=0,\ldots,m-1 $, $ k\ne i $, let $ l_k=\mathds{1}_{\geschweift{\bx\in \scanset_k}} $. Then,
	\begin{align*}
		&\ \bx\in\left(\bigcap_{k\ne i} \scanset_k^{l_k}\right)\cap \scanset_i=D_{\sum\limits_{k=0}^{m-1}l_k2^k}
	\end{align*}
	with $ l_i=1 $.
	Next, we  show that $ \scanset_i\supset \sum_{l: l_i=1} D_l $. Indeed, by definition, it holds for $ l $ with $ l_i=1 $ that
	\[ 
	D_l= \scanset_i\cap\left(\bigcap_{k\ne i}\scanset_k^{l_k}\right)\subset \scanset_i.
	\]
	Due to the $M$-dependence of  $\geschweift{\epsilon_i}$, the sums $S_{D_l}$ over the $D_l$ are not independent. To ensure independence, we define slight modifications of   $D_l$ that have minimal distance of at least $M/T$ in the supremum norm between them. 
Define the sets $ \tscanset_i=\tscanset_{i,T}=\scanset_i^{(M/T)} \cap [0,1]^p$ and define
	\begin{align*}
		&D_l^{(-M)}=\bigcap_{i:l_i=0}\tscanset_i^0\cap \bigcap_{i:l_i=1}\scanset_i.
	\end{align*} Recall that $D_l=\bigcap_{i:l_i=0}\scanset_i^0\cap \bigcap_{i:l_i=1}\scanset_i$.
	Thus it holds that
	\begin{align*}
D_l\setminus D_l^{(-M)}=\bigcap\limits_{l:l_i=1}A_i\cap\left(\left(\bigcap\limits_{l:l_i=0} A_i^0\right)\setminus\left(\bigcap\limits_{l:l_i=0} \tilde{A}_i^0\right)\right)\subset\bigcup\limits_{i=0}^{m-1}A_i^0\setminus\tilde{A}_i^0=\bigcup\limits_{i=0}^{m-1}\tilde{A}_i\setminus A_i.
\end{align*}
By Lemma \ref{lem_jordan_assumptions} (c) it holds for $ \tscanset=\scanset^{(M/T)} \cap [0,1]^p$ that
\begin{align*}
	& \abs{\geschweift{\bk\in\Z^p:\frac{\bk}{T}\in \tscanset\setminus\scanset}}=o(T^p).
%	\intertext{and thus by \eqref{moments} and Markov's inequality,}
%	&\frac{1}{T^{p/2}}\danobisum_{\scanset}(\floor{\bt_i}_T)=\frac{1}{T^{p/2}}\danobisum_{\tscanset}(\floor{\bt_i}_T)+o_P(1). \label{replacement_set_fidis}
\end{align*}	
Thus it holds for all $l=1,\ldots,2^m-1$ by \eqref{moments} and Markov's inequality that
	\begin{align}\label{eq_DmM}
		&\frac{1}{T^{p/2}}S_{D_l}=\frac{1}{T^{p/2}}S_{D^{(-M)}_l}+o_P(1),
	\end{align}
	  where the $o_P$-term is uniform in $l$.
	
Indeed, the $ S_{D^{(-M)}_l} $ are independent: For $ k\ne l $  there exists $ r\in\geschweift{0,\ldots,m-1} $ with $ k_r=1 $, $ l_r=0 $ (or vice-versa) such that $ D^{(-M)}_k \subset \scanset_r $ and $ D^{(-M)}_l \subset \tscanset_r^{\prime} $ (or vice-versa). By definition of $\tscanset_r$ the minimal distance with respect to the supremum norm between $\scanset_r$ and $\tscanset_r^{\prime}$ (and by extension also the minimal distance between $ D^{(-M)}_k$ and $ D^{(-M)}_l$) is more than $M/T$ and thus, the minimal distance between $\bk\in\Z^p$ with $\bk/T\in D^{(-M)}_k$ and $\bl\in\Z^p$ with $\bl/T\in D^{(-M)}_k$ is at least $M+1$.
Consequently, the independence of $ S_{D^{(-M)}_l} $, $l=1,\ldots,2^m-1$, follows from the $M$-dependence of $\{\epsilon_i\}$.
By Theorem \ref{clt_m_dependence}, it holds that  $S_{D_l}/T^{p/2}\dto\mathcal{N}(0,\sigma^2\lambda(D_l))$. By \eqref{eq_DmM} it follows that $S_{D_l^{(-M)}}/T^{p/2}\dto\mathcal{N}(0,\sigma^2\lambda(D_l))$.
Due to the independence of the $S_{D_l^{(-M)}}$, there exist independent $\mathcal{N}(0,\sigma^2\lambda(D_l))$-distributed $\widetilde{W}_l$ such that 
\begin{align*}
	&\ \left(\frac{1}{T^{p/2}}S_{D_1^{(-M)}},\ldots,
	\frac{1}{T^{p/2}}S_{ D_{2^m-1^{(-M)}}} \right)^{\prime}
	\dto\left(\widetilde{W}_1,\ldots,\widetilde{W}_{2^m-1}\right)^{\prime}\\
	\intertext{and by \eqref{eq_DmM}, it holds that}
	&\ \left(\frac{1}{T^{p/2}}S_{D_1},\ldots,
\frac{1}{T^{p/2}}S_{ D_{2^m-1}} \right)^{\prime}\dto\left(\widetilde{W}_1,\ldots,\widetilde{W}_{2^m-1}\right)^{\prime},
\end{align*}
as well. For $i=0,\ldots,m-1$ let 
\[ 
W_i=\sum\limits_{l:l_i=1}\widetilde{W}_l.
 \]
 Due to \eqref{union}, $W_i$ is $\mathcal{N}(0,\sigma^2\lambda(A_i))$-distributed.

	Analogously to above, we are renumbering the $ a_{i,j} $ for the ease of notation with a slight abuse of notation as $ a_{(j-1)n+i-1}=a_{i,j}$ for $ i=1,\ldots,n $, $ j=1,\ldots,P $. Then, \eqref{eq_CramWold} holds as
	\begin{align*}
		&\ \sum_{i=1}^{n}\sum_{j=1}^{P}a_{i,j}\frac{1}{T^{p/2}}\danobisum_{\scanset_j}(\bt_i)= \sum_{i=0}^{m-1} a_i\frac{1}{T^{p/2}} \sum_{ \frac{\bk}{T}\in \scanset_i}\epsilon_{\bk}
		= \sum_{i=0}^{m-1}a_i\sum\limits_{l: l_i=1}\frac{1}{T^{p/2}}  \sum_{ \frac{\bk}{T}\in D_l}\epsilon_{\bk}\\
		=&\ \sum_{i=0}^{m-1}a_i\sum\limits_{l: l_i=1}\frac{1}{T^{p/2}} S_{D_l}
		\dto \sum_{i=0}^{m-1}a_i\sum\limits_{l: l_i=1} \widetilde{W}_l =\sum_{i=0}^{m-1}a_i W_i.
	\end{align*} 
\end{proof}	
\end{Lemma}

\begin{Lemma}\label{modulus_of_continuity}
	Let $ \scanset_1,\ldots,\scanset_P\subset[0,1]^p $ fulfill Assumption \ref{ass_sets}.
	Let $ \left(\epsilon_{\bk}\right)_{\bk\in\Z^p} $ be a sequence of  $ M $-dependent random variables fulfilling Assumption \ref{ass_errors}, i.e.\ $ \EW{\epsilon_{\bk}}=0 $, $\EW{\epsilon_{\bk}^2}\in(0,\infty)  $ and let $ \EW{\abs{\epsilon_{\bk}}^r}<\infty $ for some $ r>2p $. 
	Let $ \danobipixel_{\bk,T}=\mu_{\bk,T}+\epsilon_{\bk} $ with $ \mu_{\bk,T}\in\R $,
	let $ \danobisum_{\scanset}(\floor{\bs}_T)= \danobisum_{\scanset}(\epsilon;\floor{\bs}_T) $ be as in \eqref{eq_def_sums} and let
	\[ 
	\Sprocess_T(\bs)=\left(\frac{1}{T^{p/2}}\danobisum_{\scanset_1}(\floor{\bs}_T),\ldots,\frac{1}{T^{p/2}}\danobisum_{\scanset_P}(\floor{\bs}_T)\right)^{\prime}.
	\]
	For any $ x>0 $ it holds that
	\[ 
	\lim\limits_{\delta\to 0}\limsup_{T\to\infty}\Prob{\supp{\left\|\bs-\bt\right\|_{\infty}<\delta}\left\|\Sprocess_T(\bs)-\Sprocess_T(\bt)\right\|_{\infty}\ge x}=0.
	\]
	\begin{proof}
	For arbitrary $ \bs,\bt\in[0,1]^p $, it holds that
	\[ 
	\left\|\Sprocess_T(\bs)-\Sprocess_T(\bt)\right\|_{\infty}=\max_{i=1,\ldots,P}\abs{\frac{1}{T^{p/2}}\danobisum_{\scanset_i}\left(\floor{\bs}_T\right)-\frac{1}{T^{p/2}}\danobisum_{\scanset_i}\left(\floor{\bt}_T\right)}.
	\]
	Therefore, showing
	\[ 
	\lim\limits_{\delta\to 0}\limsup_{T\to\infty}\Prob{\supp{\left\|\bs-\bt\right\|_{\infty}<\delta}\left\|\Sprocess_T(\bs)-\Sprocess_T(\bt)\right\|_{\infty}\ge x}=0
	\]
	for all $ x>0 $ is equivalent to showing
	\[ 
	\lim\limits_{\delta\to 0}\limsup_{T\to\infty}\Prob{\supp{\left\|\bs-\bt\right\|_{\infty}<\delta}\abs{\frac{1}{T^{p/2}}\danobisum_{\scanset_i}\left(\floor{\bs}_T\right)-\frac{1}{T^{p/2}}\danobisum_{\scanset_i}\left(\floor{\bt}_T\right)}\ge x}=0.
	\]
	for all $ x>0 $, $ i=1,\ldots,P $.
	
	For any $ \delta>0 $, define $ k_0=k_0(\delta) $  such that $ 2^{-(k_0+1)}<\delta\le2^{-k_0} $. Then letting $ \delta\to0 $ is equivalent to letting $ k_0\to\infty $. Let $ m=m(T)=\ceil{\left(2p/r+\varepsilon\right)\log_2 T} $ for some sufficiently small $ \varepsilon>0 $ such that $ 2^m=o(T) $ as $ 2p<r $. Since in the following, we are only going to use sets of the form $ \scanset^{\left(2^{-k}\right)} $, we will write $ \scanset^{k}=\scanset^{\left(2^{-k}\right)} $ for the ease of notation. For $ \bs,\bt\in[0,1]^p $ with $ \left\|\bs-\bt\right\|_{\infty}<\delta $, we can write
	\begin{align}
		&\ \danobisum_{\scanset_i}\left(\floor{\bs}_T\right)-\danobisum_{\scanset_i}\left(\floor{\bt}_T\right)\notag\\*
		=&\ \danobisum_{\scanset_i}\left(\floor{\bs}_T\right)-\danobisum_{\scanset_i}\left(\floor{\bs}_{2^m}\right)\label{step1}\\*
		&\ +\danobisum_{\scanset_i}\left(\floor{\bs}_{2^m}\right)-\danobisum_{\scanset_i^{m}}\left(\floor{\bs}_{2^m}\right)\label{step2}\\*
		&\ +\sum_{k=k_0}^{m-1} \left(\danobisum_{\scanset_i^{k+1}}\left(\floor{\bs}_{2^{k+1}}\right)-\danobisum_{\scanset_i^{k}}\left(\floor{\bs}_{2^{k}}\right)\right)\label{step3}\\* &\ +\danobisum_{\scanset_i^{k_0}}\left(\floor{\bs}_{2^{k_0}}\right) -\danobisum_{\scanset_i^{k_0}}\left(\floor{\bt}_{2^{k_0}}\right)\label{step4}\\* 
		&\ +\sum_{k=k_0}^{m-1} \left(\danobisum_{\scanset_i^{k}}\left(\floor{\bt}_{2^{k}}\right)-\danobisum_{\scanset_i^{k+1}}\left(\floor{\bt}_{2^{k+1}}\right)\right)\notag\\* &\ +\danobisum_{\scanset_i^{m}}\left(\floor{\bt}_{2^m}\right)-\danobisum_{\scanset_i}\left(\floor{\bt}_{2^m}\right)\notag\\*
		&\ +\danobisum_{\scanset_i}\left(\floor{\bt}_{2^m}\right)-\danobisum_{\scanset_i}\left(\floor{\bt}_T\right)\notag.
	\end{align}
	Note that this telescoping series consists of four separate elements:
	In the first step \eqref{step1}, the step function $ \danobisum_{\scanset_i}\left(\floor{\bs}_T\right) $ on a grid of spacing $ 1/T $ is replaced by the corresponding step function $ \danobisum_{\scanset_i}\left(\floor{\bs}_{2^m}\right) $ on a grid of spacing $ 2^{-m}>1/T $.
	In the second step \eqref{step2}, $ \scanset_i $ is replaced by a larger set of Lebesgue measure $ \lambda(\scanset_i)+O(2^{-m}) $, compare Lemma \ref{lem_covering_fattening} (a). 
	In the third step \eqref{step3} we continue to coarsen the grid and expand the sets around $ \bs $ until the distance between adjacent grid points is more than $ \left\|\bs-\bt\right\|_{\infty} $.
	In the fourth step \eqref{step4}, we move from $ \bs $ to $ \bt $ on sets of the form $ \scanset_i^{k_0} $. Then steps 3 to 1 are reversed for $ \bt. $
	In order to show the assertion, we will analyze \eqref{step1}-\eqref{step4} separately.\\\\
	For \eqref{step1}: For $ T $ large enough (by assumption, $ 2^m=o(T) $), it holds that 
	\begin{align}
		\left\|\floor{\bs}_T-\floor{\bs}_{2^m}\right\|_{\infty}\le\left\|\floor{\bs}_T-\bs\right\|_{\infty}+\left\|\bs-\floor{\bs}_{2^m}\right\|_{\infty}\le \frac{1}{T}+2^{-m}\le 2\cdot 2^{-m}=2^{-(m-1)}. \label{grids}
	\end{align}	
	Therefore, it holds  that
	\begin{align*}
		&\ \scanset_i\left(\floor{\bs}_T\right)\setminus \scanset_i\left(\floor{\bs}_{2^{m}}\right)\subset \scanset_i^{m-1}\left(\floor{\bs}_{2^{m}}\right)\setminus \scanset_i\left(\floor{\bs}_{2^{m}}\right),\\
		&\ \scanset_i\left(\floor{\bs}_{2^{m}}\right)\setminus \scanset_i\left(\floor{\bs}_T\right) \subset \scanset_i^{m-1}\left(\floor{\bs}_{T}\right)\setminus \scanset_i\left(\floor{\bs}_{T}\right),\\
		\intertext{and thus, it holds by Lemma \ref{lem_covering_fattening} (b) that}
		&\ \abs{\geschweift{\bk\in\Z^p:\frac{\bk}{T}\in \scanset_i\left(\floor{\bs}_T\right)\setminus \scanset_i\left(\floor{\bs}_{2^{m}}\right)}}\\
		\le&\ \abs{\geschweift{\bk\in\Z^p:\frac{\bk}{T}\in \scanset_i^{m-1}\left(\floor{\bs}_{2^{m}}\right)\setminus \scanset_i\left(\floor{\bs}_{2^{m}}\right)}}
		\le CT^p2^{-m}\\
		\intertext{and}
		&\ \abs{\geschweift{\bk\in\Z^p:\frac{\bk}{T}\in \scanset_i\left(\floor{\bs}_{2^{m}}\right)\setminus \scanset_i\left(\floor{\bs}_T\right)}}\\
		\le&\ \abs{\geschweift{\bk\in\Z^p:\frac{\bk}{T}\in \scanset_i^{m-1}\left(\floor{\bs}_{T}\right)\setminus \scanset_i\left(\floor{\bs}_{T}\right)}}
		\le CT^p2^{-m}
	\end{align*}
	with a suitable constant $ C>0 $.\\
	
	Therefore, we obtain by Markov's inequality and \eqref{moments}
	with suitable constants $ C_1,C_2>0 $ that
	\begin{align*}
		&\ \Prob{\abs{\frac{1}{T^{p/2}}\danobisum_{\scanset_i}\left(\floor{\bs}_T\right)-\frac{1}{T^{p/2}}\danobisum_{\scanset_i}\left(\floor{\bs}_{2^m}\right)}\ge x} \notag\\
		=&\ \Prob{\frac{1}{T^{p/2}}
			\abs{
				\sum\limits_{\frac{\bk}{T}\in\scanset_i\left(\floor{\bs}_T\right)}\epsilon_{\bk}- \sum\limits_{\frac{\bk}{T}\in\scanset_i\left(\floor{\bs}_{2^m}\right)}\epsilon_{\bk}	
			}\ge x}\notag\\
		=&\ \Prob{\frac{1}{T^{p/2}}
			\abs{
				\sum\limits_{\frac{\bk}{T}\in\scanset_i\left(\floor{\bs}_T\right)\setminus\scanset_i\left(\floor{\bs}_{2^m}\right)}\epsilon_{\bk}- \sum\limits_{\frac{\bk}{T}\in\scanset_i\left(\floor{\bs}_{2^m}\right)\setminus\scanset_i\left(\floor{\bs}_{T}\right)}\epsilon_{\bk}	
			}\ge x}\notag\\
		\le&\  
		\Prob{\frac{1}{T^{p/2}}\abs{\sum\limits_{\frac{\bk}{T}\in\scanset_i\left(\floor{\bs}_T\right)\setminus\scanset_i\left(\floor{\bs}_{2^m}\right)}\epsilon_{\bk}}\ge \frac{x}{2}}+\Prob{\frac{1}{T^{p/2}}\abs{\sum\limits_{\frac{\bk}{T}\in\scanset_i\left(\floor{\bs}_{2^m}\right)\setminus\scanset_i\left(\floor{\bs}_{T}\right)}\epsilon_{\bk}}\ge \frac{x}{2}}\notag\\
		\le&\ C_1\left(\frac{\abs{\geschweift{\bk:\frac{\bk}{T}\in\scanset_i\left(\floor{\bs}_T\right)\setminus\scanset_i\left(\floor{\bs}_{2^m}\right)}}^{r/2}}{T^{pr/2}x^r}	
		+\frac{\abs{\geschweift{\bk:\frac{\bk}{T}\in\scanset_i\left(\floor{\bs}_{2^m}\right)\setminus\scanset_i\left(\floor{\bs}_{T}\right)}}^{r/2}}{T^{pr/2}x^r}\right)
		\\
		\le&\ C_22^{-mr/2}\le C_2T^{-\frac{r}{2}\left(\frac{2p}{r}+\varepsilon\right)}=C_2T^{-(p+r\varepsilon/2)},
	\end{align*}

		since by definition, $ m=\ceil{\left(2p/r+\varepsilon\right)\log_2 T} $ and thus, $ 2^{-m}\le T^{-(2p/r+\varepsilon)} $.		
		We will now consider the supremum over $ \bs\in[0,1]^p $ of the above expression that only depends on $ \bs $ via $ \floor{\bs}_T $ and $ \floor{\bs}_{2^m} $. Thus, it is by \eqref{grids} sufficient to take the supremum over $ \floor{\bs}_T $ with $ \bs\in[0,1]^p $ and over all $ \floor{\bs}_{2^m} $ with $ \left\|\floor{\bs}_{2^m}-\floor{\bs}_T\right\|\le 2\cdot 2^{-m} $ for fixed $ \floor{\bs}_T $.\\
		The first supremum is taken over the set
		$ \geschweift{\floor{\bs}_T|\bs\in[0,1]^p} $, which has cardinality $ (T+1)^p\le2^pT^p $. The second supremum is taken over a set that  has finite cardinality as it can be easily verified that the number of points of the form $ \bk2^{-m} $, $ \bk\in\Z^p $ in the hypercube
		$ \left[\floor{\bs}_T-2\cdot2^{-m},\floor{\bs}_T+2\cdot2^{-m}\right] $ is finite.		
		Therefore, we obtain for any $ x>0 $ with a suitable constant $ C_3>0 $ potentially depending on $ x $ by subadditivity that
		\begin{align}
			&\ \Prob{\supp{\bs\in[0,1]^p}\abs{\frac{1}{T^{p/2}}\danobisum_{\scanset_i}\left(\floor{\bs}_T\right)-\frac{1}{T^{p/2}}\danobisum_{\scanset_i}\left(\floor{\bs}_{2^m}\right)}\ge x} \notag\\
			\le&\ \sum_{\floor{\bs}_T:\ \bs\in[0,1]^p}\ \sum_{\stackrel{\floor{\bs}_{2^m}:\ \bs\in[0,1]^p,}{  
					\left\|\floor{\bs}_T-\floor{\bs}_{2^m}\right\|_{\infty}\le 2\cdot2^{-m}}}  \Prob{\abs{\frac{1}{T^{p/2}}\danobisum_{\scanset_i}\left(\floor{\bs}_T\right)-\frac{1}{T^{p/2}}\danobisum_{\scanset_i}\left(\floor{\bs}_{2^m}\right)}\ge x}\notag\\
			\le&\ C_3T^pT^{-(p+r\varepsilon/2)}=C_3T^{-r\varepsilon/2}\to0  \label{danobi_term1}.
		\end{align}
		as $ T\to\infty $.\\\\
		For \eqref{step2}, it holds by Lemma \ref{lem_covering_fattening} (b) that
		\[ 
		\abs{\geschweift{\bk\in\Z^p:\frac{\bk}{T}\in\scanset_i^{m}\left(\floor{\bs}_{2^m}\right)\setminus\scanset_i\left(\floor{\bs}_{2^m}\right)}}\le CT^p2^{-m}.
		\]
		Therefore, we obtain analogously to \eqref{danobi_term1} with  suitable constants $ C_1,C_2>0 $ that
		\begin{align}
			&\ \Prob{\supp{\bs\in[0,1]^p} \abs{\frac{1}{T^{p/2}}\danobisum_{\scanset_i}\left(\floor{\bs}_{2^m}\right)-\frac{1}{T^{p/2}}\danobisum_{\scanset_i^{m}}\left(\floor{\bs}_{2^m}\right)}\ge x}\notag\\
			\le&\ \sum_{\floor{\bs}_{2^m}:\ \bs\in[0,1]^p} \Prob{\abs{\frac{1}{T^{p/2}}\danobisum_{\scanset_i}\left(\floor{\bs}_{2^m}\right)-\frac{1}{T^{p/2}}\danobisum_{\scanset_i^{m}}\left(\floor{\bs}_{2^m}\right)}\ge x}\notag\\
			\le &\ 2^p\cdot2^{pm} C_1 \frac{\abs{\geschweift{\bk\in\Z^p:\frac{\bk}{T}\in\scanset_i^{m}\left(\floor{\bs}_{2^m}\right)\setminus\scanset_i\left(\floor{\bs}_{2^m}\right)}}^{r/2}}{T^{pr/2}x^r}\notag\\
			\le&\ C_2 2^{pm}2^{-mr/2}=C_22^{m(p-r/2)}\to 0 \label{danobi_term2}
		\end{align}
		as $ T\to\infty $ since $ r>2p $ and $ m\to\infty $ as $ T\to\infty $.\\\\
		For \eqref{step3},it holds that for any $ k=k_0,\ldots,m-1 $ and any $ \floor{\bs}_{2^{k+1}} $, $ \floor{\bs}_{2^k} $ is uniquely determined: If $ \floor{\bs}_{2^{k+1}}=\bl2^{-(k+1)}=(\bl/2)2^{-k} $ for $ \bl\in\Z^p $, then $ \floor{\bs}_{2^{k}}=\floor{\bl/2}2^{-k} $, where $ \floor{\bl/2} $ denotes the componentwise integer part of $ \bl/2 $. 
		By construction it holds that
		$ 	\left\|\floor{\bs}_{2^k}-\floor{\bs}_{2^{k+1}}\right\|_{\infty}\le 2^{-(k+1)}$ 
		and therefore $ \scanset_i\left(\floor{\bs}_{2^k}\right)\subset \scanset_i^{k+1}\left(\floor{\bs}_{2^{k+1}}\right)$. Let $ \bx\in\scanset_i^{k+1}\left(\floor{\bs}_{2^{k+1}}\right) $. By definition, there exists $ \by_0\in\scanset_i\left(\floor{\bs}_{2^{k+1}}\right)  $ such that $ \left\|\bx-\by_0\right\|_{\infty}\le2^{-(k+1)} $. For $ \by_1=\by_0-\floor{\bs}_{2^{k+1}}+\floor{\bs}_{2^{k}} $ it holds that $ \by_1\in\scanset_i\left(\floor{\bs}_{2^{k}}\right) $ and
		\begin{align}
			&\ 
			\inff{\by\in\scanset_i\left(\floor{\bs}_{2^{k}}\right)}\left\|\bx-\by\right\|_{\infty}\le\left\|\bx-\by_1\right\|_{\infty}
			\le \left\|\bx-\by_0\right\|_{\infty}+\left\|\floor{\bs}_{2^k}-\floor{\bs}_{2^{k+1}}\right\|_{\infty}\notag\\
			\le&\ 2^{-(k+1)}+2^{-(k+1)}=2^{-k}. \label{inclusion}
		\end{align}
		Hence, $ \scanset_i^{k+1}\left(\floor{\bs}_{2^{k+1}}\right)\subset\scanset_i^{k}\left(\floor{\bs}_{2^{k}}\right) $ and thus it holds by Lemma \ref{lem_covering_fattening} (b) with a suitable constant $ C>0 $ that
		\begin{align*}
			&\ \abs{\geschweift{\bk\in\Z^p:\frac{\bk}{T}\in\scanset_i^{k}\left(\floor{\bs}_{2^{k}}\right)\setminus \scanset_i^{k+1}\left(\floor{\bs}_{2^{k+1}}\right)}}\\
			\le&\
			\abs{\geschweift{\bk\in\Z^p:\frac{\bk}{T}\in\scanset_i^{k}\left(\floor{\bs}_{2^{k}}\right)\setminus \scanset_i\left(\floor{\bs}_{2^{k}}\right)}}\le CT^p2^{-k}.
		\end{align*}
		
		Furthermore, for $ x>0 $, if $ \delta $ is small enough (and therefore, $ k_0 $ large enough), $ x\ge \sum_{k\ge k_0} 2^{k(p/(2r)-1/4)} $ since $ p/(2r)-1/4<0 $ by $ r>2p $.
		Analogously to \eqref{danobi_term1}, we obtain by $ \sigma $-subadditivity, Markov's inequality and \eqref{moments} with suitable constants $ C_1,C_2>0 $ that
		\begin{align}
			&\ \Prob{\supp{\bs\in[0,1]^p}\abs{\sum_{k=k_0}^{m-1} \left(\frac{1}{T^{p/2}}\danobisum_{\scanset_i^{k+1}}\left(\floor{\bs}_{2^{k+1}}\right)-\frac{1}{T^{p/2}}\danobisum_{\scanset_i^{k}}\left(\floor{\bs}_{2^k}\right)\right)}\ge x}\notag \\
			\le &\ \Prob{\sum_{k=k_0}^{m-1} \supp{\bs\in[0,1]^p}\abs{ \frac{1}{T^{p/2}}\danobisum_{\scanset_i^{k+1}}\left(\floor{\bs}_{2^{k+1}}\right)-\frac{1}{T^{p/2}}\danobisum_{\scanset_i^{k}}\left(\floor{\bs}_{2^k}\right)}\ge \sum_{k=k_0}^{m-1} 2^{k\left(\frac{p}{2r}-\frac{1}{4}\right)}}\notag \\
			\le&\ \Prob{\bigcup_{k=k_0}^{m-1}\geschweift{ \supp{\bs\in[0,1]^p}\abs{ \frac{1}{T^{p/2}}\danobisum_{\scanset_i^{k+1}}\left(\floor{\bs}_{2^{k+1}}\right)-\frac{1}{T^{p/2}}\danobisum_{\scanset_i^{k}}\left(\floor{\bs}_{2^k}\right)}\ge 2^{k\left(\frac{p}{2r}-\frac{1}{4}\right)}}}\notag\\
			\le&\ \sum_{k=k_0}^{m-1} \Prob{\supp{\bs\in[0,1]^p}\abs{ \frac{1}{T^{p/2}}\danobisum_{\scanset_i^{k+1}}\left(\floor{\bs}_{2^{k+1}}\right)-\frac{1}{T^{p/2}}\danobisum_{\scanset_i^{k}}\left(\floor{\bs}_{2^k}\right)}\ge 2^{k\left(\frac{p}{2r}-\frac{1}{4}\right)}}\notag\\
			\le&\ \sum_{k=k_0}^{m-1} \sum_{\floor{\bs}_{2^{k+1}}:\ \bs\in[0,1]^p} 
			\Prob{\abs{ \frac{1}{T^{p/2}}\danobisum_{\scanset_i^{k+1}}\left(\floor{\bs}_{2^{k+1}}\right)-\frac{1}{T^{p/2}}\danobisum_{\scanset_i^{k}}\left(\floor{\bs}_{2^k}\right)}\ge 2^{k\left(\frac{p}{2r}-\frac{1}{4}\right)}}\notag\\
			\le&\ \sum_{k=k_0}^{m-1}  \left(2^{k+1}+1\right)^pC_1 \frac{\abs{\geschweift{\bk\in\Z^p:\frac{\bk}{T}\in\scanset_i^{k}\left(\floor{\bs}_{2^k}\right)\setminus\scanset_i^{k+1}\left(\floor{\bs}_{2^{k+1}}\right)}}^{r/2}}{T^{pr/2}2^{kr\left(\frac{p}{2r}-\frac{1}{4}\right)}}\notag\\
			\le&\ C_2 \sum_{k=k_0}^{m-1} 2^{kp}\cdot 2^{-kr/2}\cdot 2^{-kr\left(\frac{p}{2r}-\frac{1}{4}\right)}=C_2 \sum_{k=k_0}^{m-1} 2^{\frac{2p-r}{4}k}\notag\\
			\le&\ C_2 \sum_{k=k_0}^{\infty} 2^{\frac{2p-r}{4}k}\to0\label{danobi_term3}
		\end{align}		
		as $ \delta\to 0 $ (and therefore $ k_0\to\infty $)	since $ r>2p $.\\\\
		For \eqref{step4}, it holds for $ \left\|\bs-\bt\right\|_{\infty}<\delta $ and due to $ \delta\le2^{-k_0} $ by definition of $ k_0 $ that
		\begin{align}
			&\ \left\|\floor{\bs}_{2^{k_0}}-\floor{\bt}_{2^{k_0}}\right\|_{\infty}\le
			\left\|\floor{\bs}_{2^{k_0}}-\bs\right\|_{\infty} + 	\left\|\bs-\bt\right\|_{\infty}+	\left\|\bt-\floor{\bt}_{2^{k_0}}\right\|_{\infty}\le  3\cdot 2^{-k_0}. \label{grids2}
		\end{align}
		since by definition of $ k_0 $, $ 2^{-(k_0+1)}<\delta\le 2^{-k_0} $.
		For $ \bx\in \scanset_i^{k_0}\left(\floor{\bs}_{2^{k_0}}\right) $, there exists $ \by_0\in \scanset_i\left(\floor{\bs}_{2^{k_0}}\right)$ such that $ \left\|\bx-\by_0\right\| _{\infty}\le2^{-k_0}$. For $ \by_1=\by_0+\left(\floor{\bt}_{2^{k_0}}\right)-\left(\floor{\bs}_{2^{k_0}}\right) $ it holds analogously to \eqref{inclusion} that
		\begin{align*}
			&\ 
			\inff{\by\in\scanset_i\left(\floor{\bt}_{2^{k_0}}\right)}\left\|\bx-\by\right\|_{\infty}\le\left\|\bx-\by_1\right\|_{\infty}
			\le \left\|\bx-\by_0\right\|_{\infty}+\left\|\floor{\bt}_{2^k_0}-\floor{\bs}_{2^{k_0}}\right\|_{\infty}\\
			\le&\ 2^{-k_0}+3\cdot2^{-k_0}=2^{-(k_0-2)}. 
		\end{align*}
		Hence, $ \scanset_i^{k_0}\left(\floor{\bs}_{2^{k_0}}\right)\subset \scanset_i^{k_0-2}\left(\floor{\bt}_{2^{k_0}}\right) $ and analogously we obtain that  $ \scanset_i^{k_0}\left(\floor{\bt}_{2^{k_0}}\right)\subset \scanset_i^{k_0-2}\left(\floor{\bs}_{2^{k_0}}\right) $.
		Therefore, it holds by Lemma \ref{lem_covering_fattening} (b) with a suitable constant $ C>0 $ that
		\begin{align*}
			&\ \abs{\geschweift{\bk\in\Z^p:\frac{\bk}{T}\in \scanset_i^{k_0}\left(\floor{\bs}_{2^{k_0}}\right)\setminus \scanset_i^{k_0}\left(\floor{\bt}_{2^{k_0}}\right)}}\\ 
			\le&\  \abs{\geschweift{\bk\in\Z^p:\frac{\bk}{T}\in \scanset_i^{k_0-2}\left(\floor{\bt}_{2^{k_0}}\right)\setminus \scanset_i\left(\floor{\bt}_{2^{k_0}}\right)}}\le CT^p2^{-k_0},\\
			\intertext{and}
			&\ \abs{\geschweift{\bk\in\Z^p:\frac{\bk}{T}\in\scanset_i^{k_0}\left(\floor{\bt}_{2^{k_0}}\right)\setminus \scanset_i^{k_0}\left(\floor{\bs}_{2^{k_0}}\right)}}\\
			\le&\  \abs{\geschweift{\bk\in\Z^p:\frac{\bk}{T}\in \scanset_i^{k_0-2}\left(\floor{\bs}_{2^{k_0}}\right)\setminus \scanset_i\left(\floor{\bs}_{2^{k_0}}\right)}}\le CT^p2^{-k_0}.
		\end{align*}				
		Analogously to \eqref{danobi_term1}, 
		we now consider the supremum over $ \bs,\bt\in[0,1]^p $ with $ \left\|\bs-\bt\right\|_{\infty}<\delta $ of the below expression that only depends on $ \bs,\bt $ via $ \floor{\bs}_{2^{k_0}} $ and $ \floor{\bt}_{2^{k_0}} $. Thus, it is by \eqref{grids2} sufficient to take the supremum over $ \floor{\bs}_{2^{k_0}} $ with $ \bs\in[0,1]^p $ and over all $ \floor{\bt}_{2^{k_0}} $ with $ \left\|\floor{\bs}_{2^{k_0}}-\floor{\bt}_{2^{k_0}}\right\|\le 3\cdot 2^{-k_0} $ for fixed $ \floor{\bs}_{2^{k_0}} $.\\
		The first supremum is taken over the set
		$ \geschweift{\floor{\bs}_{2^{k_0}}|\bs\in[0,1]^p} $, which has cardinality $ (2^{k_0}+1)^p\le2^p2^{pk_0} $. The second supremum is taken over a set that  has finite cardinality as it can be easily verified that the number of points of the form $ \bk2^{-k_0} $, $ \bk\in\Z^p $ in the hypercube
		$ \left[\floor{\bs}_{2^{k_0}}-3\cdot2^{-k_0},\floor{\bs}_{2^{k_0}}+3\cdot2^{-k_0}\right] $ is finite.		
		Thus we obtain for $ x>0 $ with $ \sigma $-subadditivity and suitable constants $ C_1,C_2>0 $ that
		\begin{align}
			&\ \Prob{\supp{\left\|\bs-\bt\right\|_{\infty}<\delta}\abs{\frac{1}{T^{p/2}}\danobisum_{\scanset_i^{k_0}}\left(\floor{\bs}_{2^{k_0}}\right) -\frac{1}{T^{p/2}}\danobisum_{\scanset_i^{k_0}}\left(\floor{\bt}_{2^{k_0}}\right)}\ge x}\notag\\
			\le&\ 		
			\sum_{\floor{\bs}_{2^{k_0}}:\ \bs\in[0,1]^p}\ \sum_{\stackrel{\floor{\bt}_{2^{k_0}}:\ \bt\in[0,1]^p,}{  
					\left\|\floor{\bs}_{2^{k_0}}-\floor{\bt}_{2^{k_0}}\right\|_{\infty}\le 3\cdot2^{-k_0}}}   \Prob{\frac{1}{T^{p/2}}\abs{\danobisum_{\scanset_i^{k_0}}\left(\floor{\bs}_{2^{k_0}}\right) -\danobisum_{\scanset_i^{k_0}}\left(\floor{\bt}_{2^{k_0}}\right)}\ge x}\notag\\
			\le&\  2^{pk_0} C_1 
			\frac{\abs{\geschweift{\bk\in\Z^p:\frac{\bk}{T}\in \scanset_i^{k_0}\left(\floor{\bs}_{2^{k_0}}\right)\setminus \scanset_i^{k_0}\left(\floor{\bt}_{2^{k_0}}\right)}}^{r/2}}	
			{T^{pr/2}x^r}\notag \\
			&+2^{pk_0} C_1 
			\frac{\abs{\geschweift{\bk\in\Z^p:\frac{\bk}{T}\in\scanset_i^{k_0}\left(\floor{\bt}_{2^{k_0}}\right)\setminus \scanset_i^{k_0}\left(\floor{\bs}_{2^{k_0}}\right)}}^{r/2}}	
			{T^{pr/2}x^r}\notag \\
			\le&\ C_2 2^{pk_0}2^{-k_0r/2}=C_2 2^{k_0(p-r/2)}\to 0 \label{danobi_term4}
		\end{align}		
		as $ \delta\to 0 $ (and therefore $ k_0\to\infty $)	since $ r>2p $.\\\\ 
		By subadditivity, \eqref{danobi_term1}, \eqref{danobi_term2}, \eqref{danobi_term3} and \eqref{danobi_term4}, we obtain for any $ x>0 $ that
		\begin{align*}
			&\ \lim\limits_{\delta\to 0}\limsup_{T\to\infty}\Prob{\supp{\left\|\bs-\bt\right\|_{\infty}<\delta}\abs{\frac{1}{T^{p/2}}\danobisum_{\scanset_i}\left(\floor{\bs}_T\right)-\frac{1}{T^{p/2}}\danobisum_{\scanset_i}\left(\floor{\bt}_T\right)}\ge 7x}\\
			\le&\ \lim\limits_{\delta\to 0}\limsup_{T\to\infty} 
			2\Prob{\supp{\bs\in[0,1]^p}\abs{\frac{1}{T^{p/2}}\danobisum_{\scanset_i}\left(\floor{\bs}_T\right)-\frac{1}{T^{p/2}}\danobisum_{\scanset_i}\left(\floor{\bs}_{2^m}\right)}\ge x}\\*
			& + \lim\limits_{\delta\to 0}\limsup_{T\to\infty}  2\Prob{\supp{\bs\in[0,1]^p}\abs{\frac{1}{T^{p/2}}\danobisum_{\scanset_i}\left(\floor{\bs}_{2^m}\right)-\frac{1}{T^{p/2}}\danobisum_{\scanset_i^{m}}\left(\floor{\bs}_{2^m}\right)}\ge x}\\*
			& + \lim\limits_{\delta\to 0}\limsup_{T\to\infty} 
			2\Prob{\supp{\bs\in[0,1]^p} \abs{\sum_{k=k_0}^{m-1} \frac{1}{T^{p/2}}\danobisum_{\scanset_i^{k+1}}\left(\floor{\bs}_{2^{k+1}}\right)-\frac{1}{T^{p/2}}\danobisum_{\scanset_i^{k}}\left(\floor{\bs}_{2^k}\right)}\ge x}\\*
			& + \lim\limits_{\delta\to 0}\limsup_{T\to\infty}  \Prob{\supp{\left\|\bs-\bt\right\|_{\infty}<\delta} \abs{\frac{1}{T^{p/2}}\danobisum_{\scanset_i^{k_0}}\left(\floor{\bs}_{2^{k_0}}\right) -\frac{1}{T^{p/2}}\danobisum_{\scanset_i^{k_0}}\left(\floor{\bt}_{2^{k_0}}\right)}\ge x}=0
		\end{align*}		
		thus showing the assertion.	
	\end{proof}	
\end{Lemma}

	\section{The space $ \cD([0,1]^p) $} \label{sec_D01}	
	\begin{Definition}{(Compare \cite{Wichura}, Definition 1)}\ \label{Wichura_def}\\
%		Let $\mathcal{C}_p$ be the set of all continuous functions from $ [0,1]^p $ to $ \R $. 
		Let $ \cD([0,1]^p) $ be the uniform closure of the set of all step functions that are constant on sets of the form $ \rechtsoffen{a_1,b_1}\times\ldots\times\rechtsoffen{a_p,b_p} $.
		Let $ \mathcal{A} $ be the $ \sigma $-algebra on $ D([0,1]^p) $ generated by the projection mappings $ \pi_t: f\to f(t) $, $ t\in[0,1]^p $.\\\\
		A sequence $ (P_n)_{n\ge1} $ of probability measures on $ (\cD([0,1]^p),\mathcal{A}) $ \emph{converges weakly in the U-topology} to a probability  measure $ P $ on $ (\cD([0,1]^p),\mathcal{A}) $ if for every measurable function $ f: \cD([0,1]^p)\to\R $ that is continuous in the topology of uniform convergence on $ \cD([0,1]^p) $, it holds that
		\[ 
		\int f\ dP_n\to\int f\ dP.
		\]
	\end{Definition}
	\begin{Theorem}(Compare \cite{Wichura}, Theorem 2) \label{convergence_D}\ \\
		A sequence $ \{X_T\} $ of stochastic processes in $ \cD([0,1]^p) $ is weakly convergent to a process $ X $ which a.s.~belongs to the class of uniformly continuous functions on $ [0,1]^p $ if and only if
		\begin{enumerate}[(i)]
			\item 
			\[ 
			\left(X_{\bt_1,T},\ldots,X_{\bt_n,T}\right)^{\prime}\dto\left(X_{\bt_1},\ldots,X_{\bt_n}\right)^{\prime}
			\]
			for all $ n\in\N $, $ \bt_1,\ldots,\bt_n\in I $, where $ I $ is a dense subset of $ [0,1]^p $ and
			\item
			\[ 
			\lim_{\delta\to 0}\limsup_{T\to\infty}\Prob{\supp{\|\bs-\bt\|_{\infty}<\delta}\abs{X_{\bs,T}-X_{\bt,T}}\ge x}=0
			\]
			for all $ x>0 $.
		\end{enumerate}
	\end{Theorem}
	\begin{Lemma}\label{projection_continuous}
		Let $ \mathcal{D}\left([0,1]^p\right) $ be equipped with the topology induced by the maximum norm. For $ f=(f_1,\ldots,f_P)^{\prime}$, $ g=(g_1,\ldots,g_P)^{\prime} \in\left(\mathcal{D}\left([0,1]^p\right)\right)^P $ and $ \bx=(\bx_1,\ldots,\bx_n)^{\prime} $, $ \by=(\by_1,\ldots,\by_n)^{\prime}\in\left(\R^P\right)^n $ let
		\begin{align*}
			d_1(f,g)=&\ \max_{i=1,\ldots,P}\ \supp{\bt\in[0,1]^p}\abs{f_i(\bt)-g_i(\bt)}\\
			d_2(\bx,\by)=&\ \max_{i=1,\ldots,n}\ 
			\left\|\bx_i-\by_i\right\|_{\infty}.
		\end{align*}
		Denote for $n\in\N $, $ \bt_1,\ldots,\bt_n\in[0,1]^p $ by $ \pi_{\bt_1,\ldots,\bt_n}: \left(\mathcal{D}\left([0,1]^p\right)\right)^P\to \left(\R^{P}\right)^n $, $ \pi_{\bt_1,\ldots,\bt_n}(f)=(f(\bt_1),\ldots,f(\bt_n))^{\prime} $
		the projection mapping from $ \left(\mathcal{D}\left([0,1]^p\right)\right)^P $ to $ \left(\R^{P}\right)^n $. If $ \left(\mathcal{D}\left([0,1]^p\right)\right)^P $ is equipped with the metric $ d_1 $ and  $ \left(\R^{P}\right)^n $ is equipped with the metric $ d_2 $  then $ \pi_{\bt_1,\ldots,\bt_n} $ is continuous.
		\begin{proof}
			Let $ f=(f_1,\ldots,f_P)^{\prime}$, $ g=(g_1,\ldots,g_P)^{\prime} \in\left(\mathcal{D}\left([0,1]^p\right)\right)^P $. Then it holds that
			\begin{align*}
				&\ d_2\left( \pi_{\bt_1,\ldots,\bt_n}(f), \pi_{\bt_1,\ldots,\bt_n}(g)\right)= \max_{i=1,\ldots,P}\left\|\pi_{\bt_1,\ldots,\bt_n}(f_i)-\pi_{\bt_1,\ldots,\bt_n}(g_i)\right\|_{\infty}\\
				=&\ \max_{i=1,\ldots,P}\left\|(f_i(\bt_1),\ldots,f_i(\bt_n))^{\prime}-(g_i(\bt_1),\ldots,g_i(\bt_n))^{\prime}\right\|_{\infty}= \max_{i=1,\ldots,P}\max_{j=1\ldots,n}\abs{f_i(\bt_j)-g_i(\bt_j)}\\
				\le&\ \max_{i=1,\ldots,P}\ \supp{\bt\in[0,1]^p}\abs{f_i(\bt)-g_i(\bt)}=d_1(f,g),
			\end{align*}
			thus showing that $ \pi_{\bt_1,\ldots,\bt_n} $ is Lipschitz continuous with Lipschitz constant $ 1 $.
		\end{proof}
	\end{Lemma}
	
	\begin{Definition}\label{Wiener_ppar}(Extension of \cite{CsorgoRevesz}, 1.11)
		Let $(W_{\bt})_{\bt\in[0,1]^p} $ be a $ p $-parameter, real-valued stochastic process. Let $ \bs=(s_1,\ldots,s_p)^{\prime} $, $ \bt=(t_1,\ldots,t_p)^{\prime}\in\R^p $ with $ s_i\le t_i $ for all $ i $. For the hyperrectangle $ R=\rechtsoffen{\bs,\bt} $ define with $\bs\odot\bt=(s_1t_1,\ldots,s_pt_p)^{\prime}$
		\[ 
		W_{R}=\sum_{\bd=(d_1,\ldots,d_p)^{\prime}\in\geschweift{0,1}}
		(-1)^{p-\sum_id_i}W_{\bs+\bd\odot(\bt-\bs)}.
		\]
		 We call $(W_{\bt})_{\bt\in[0,1]^p} $ a \emph{p-parameter Wiener process} if it fulfills
		\begin{enumerate}[(a)]
			\item For any hyperrectangle $ R\subset\rechtsoffen{0,\infty}^p $, it holds that
			\[ 
			W_{R}\sim\mathcal{N}(0,\lambda(R)).
			\]
			\item $ W_{\bt}=0 $ for all $ \bt=(t_1,\ldots,t_p)^{\prime} $ with $ t_i=0 $ for at least one $ i $.
			\item For pairwise disjoint rectangles $ R_1,\ldots,R_n\subset[0,1]^p $, $ W_{R_1},\ldots,W_{R_n} $ are independent.
			\item $(W_{\bt})_{\bt\in[0,1]^p} $ almost surely belongs to the class of uniformly continuous functions on $ [0,1]^p $.
		\end{enumerate} 
		Note that (a)-(c) implies that for $ \bs=(s_1,\ldots,s_p)^{\prime} $, $ \bt=(t_1,\ldots,t_p)^{\prime}\in\R^p $,
		\[ 
		\EW{W_{\bs}W_{\bt}}=\prod_{i=1}^{p}\min\geschweift{s_i,t_i}.
		\]
	\end{Definition}

\section{Algebraic properties of sets}	\label{sec_algebraic_properties}

\subsection{Fattening of sets, covering numbers}
\begin{Definition}\label{def_fattening}
	Let $A\subset[0,1]^p$ be compact and let $\gamma>0$. By 
	\begin{align*}
	\scanset^{(\gamma)}=&\ \bigcup\limits_{\by\in[-\gamma,\gamma]^p} \scanset(\by)=\geschweift{\bx\in\R^p\Big| \inff{\bz\in A}\left\|\bx-\bz\right\|_{\infty}\le \gamma}
%	=\geschweift{\bx\in\R^p\Big|B_{\gamma,\infty}(\bx)\cap A\ne\emptyset }\footnote{\textcolor{blue}{TODO: Check whether last equation should be moved to a remark.}}
\end{align*}we denote the \emph{$\gamma$-fattening} of $A$.
\end{Definition}

\begin{Definition}\label{def_covering_number}
	Let $\left(\mathcal{S},d\right)$ be a metric space and $A\subset\mathcal{S}$. For $\varepsilon>0$ let 
	\begin{align*}
		N(\varepsilon,A,d)=&\ \min\left\{n\in\N: \text{There exist }\bx_1,\ldots,\bx_n\in\mathcal{S}\text{ such that } A\subset\bigcup\limits_{i=1}^n B_{\varepsilon,d}\left(\bx_i\right)
		\right\}
	\end{align*}
be the \emph{covering number} of $A$, where $B_{\varepsilon,d}\left(\bx\right)=\geschweift{\by\in\mathcal{S}| d(\bx,\by)\le\varepsilon} $ denotes the closed ball of radius $\varepsilon$ around $\bx$ in the $d$-metric.
\end{Definition}

\begin{Lemma}\label{boundary_point}
	Let $A\subset[0,1]^p$ and $\gamma>0$. For any $\by\in A^{(\gamma)}\setminus A$, there exists $\by_0\in\partial A$ such that $\left\|\by-\by_0\right\|_{\infty}\le\gamma$.
	\begin{proof}
		For $\by\in A^{(\gamma)}\setminus A$ let $d=d(\by)=\inff{\bx\in A}\left\|\by-\bx\right\|_{\infty}$ such that $d\le\gamma$ by definition of $A^{(\gamma)}$. Furthermore, there exists a sequence $\left(\by_n\right)_{n\in\N}\subset A$ such that $\left\|\by-\by_n\right\|_{\infty}\downarrow d$ for $n\to\infty$. As $\left(\by_n\right)_{n\in\N}$ is a bounded sequence, there exists a subsequence $\left(\by_{a(n)}\right)_{n\in\N}$ with existing limit $\by_0 \in \closure{A}$, which implies $ \left\|\by-\by_0\right\|_{\infty}=d\le \gamma$ by the continuity of the sup-norm.\\
		We complete the proof by contradiction: To this end, assume that $\by_0\in\inner{A}$. Then there exists $d>\varepsilon>0$ such that $B_{\varepsilon,\infty}(\by_0)\subset A$. For $\bx_0=\by_0+\varepsilon/(2d)(\by-\by_0)$ it holds that 
		\begin{align*}
			&\ \left\|\bx_0-\by_0\right\|_{\infty}=\frac{\varepsilon}{2d}\left\|\by-\by_0\right\|_{\infty}=\frac{\varepsilon}{2},\\
			\intertext{hence $\bx_0\in A$, and}&\ 
			\left\|\bx_0-\by\right\|_{\infty}=\left(1-\frac{\varepsilon}{2d}\right)\left\|\by-\by_0\right\|_{\infty}=d-\frac{\varepsilon}{2}<d, 
		\end{align*}
	which is a contradiction to the minimality of $d$. We have thus shown that $\by_0\in\closure{A}\setminus\inner{A}=\partial A$, which completes the proof.
	\end{proof}
\end{Lemma}

\begin{Lemma}\label{covering_number}
		Let $A\subset[0,1]^p$ be compact and $\gamma>0$. Let $ \varepsilon>0 $, $ N(\varepsilon) $, $ \bx_1,\ldots,\bx_{N(\varepsilon)} $ be such that 
	\[ 
	\partial A\subset\bigcup\limits_{i=1}^{N(\varepsilon)}B_{\varepsilon,\infty}(\bx_i).
	\]
	Then it holds that
	\[ 
	A^{(\gamma)}\setminus A\subset \bigcup\limits_{i=1}^{N(\varepsilon)}B_{\gamma+\varepsilon,\infty}(\bx_i).
	\]
	\begin{proof}
		By Lemma \ref{boundary_point}, for $ \by\in A^{(\gamma)}\setminus A$, there exists $ \by_0\in\partial A $ such that $ \left\|\by-\by_0\right\|_{\infty}\le \gamma $. For $ \by_0 $ there exists $ i\in\geschweift{1,\ldots,N(\varepsilon)} $ such that $ \by_0\in B_{\varepsilon,\infty}(\bx_{i}) $. Therefore,
		\[ 
		\left\|\by-\bx_i\right\|_{\infty}\le\left\|\by-\by_0\right\|_{\infty}+\left\|\by_0-\bx_i\right\|_{\infty}\le \gamma+\varepsilon.
		\]
		Thus, $ \by\in B_{\gamma+\varepsilon}(\bx_i) $ and the assertion follows.
	\end{proof}
\end{Lemma}

\begin{Lemma}\label{lem_gridpoints_rectangle}
	Let $ \bs_T=\left(s_{1,T},\ldots,s_{p,T}\right)^{\prime} $, $ \bt_T=\left(t_{1,T},\ldots,t_{p,T}\right)^{\prime}$ be in $[0,1]^p$  with  $ s_{i,T}<t_{i,T}$  and $ \left(\min\limits_{i=1,\ldots,p}\left(t_{i,T}-s_{i,T}\right)\right)^{-1}=o(T) $. Then it holds that
	\[ 
	\abs{\geschweift{\bk\in\Z^p:\frac{\bk}{T}\in\left[\bs_T,\bt_T\right]}}=T^p\lambda\left(\left[\bs_T,\bt_T\right]\right)\left(1+O\left(\frac{1}{T}\left(\min\limits_{i=1,\ldots,p}\left(t_{i,T}-s_{i,T}\right)\right)^{-1}\right)\right).
	\]
	\begin{proof}
		In one dimension, it holds that
		\begin{align*}
	&\ \abs{\geschweift{k\in\Z:\frac{k}{T}\in[s,t]}}=\floor{tT}-\ceil{sT}+1\in\left[(t-s)T-1,(t-s)T+1\right]. \\
	&\ \text{Therefore, }\\
	&\ \abs{\geschweift{\bk\in\Z^p:\frac{\bk}{T}\in\left[\bs_T,\bt_T\right]}}\ge \prod\limits_{i=1}^p\left(\left(t_{i,T}-s_{i,T}\right)T-1\right)\\
	=&\ T^p\prod\limits_{i=1}^p\left(t_{i,T}-s_{i,T}\right)-T^{p-1}\sum\limits_{j=1}^p\prod\limits_{i\ne j}\left(t_{i,T}-s_{i,T}\right)+R_{1,T}\\
	=&\ T^p\lambda\left(\left[\bs_T,\bt_T\right]\right)-\sum\limits_{j=1}^p \frac{T^p\lambda\left(\left[\bs_T,\bt_T\right]\right)}{T\left(t_{j,T}-s_{j,T}\right)}+R_{1,T},\\
	&\  \text{where}\\
	&\ \frac{R_{1,T}}{T^p\lambda\left(\left[\bs_T,\bt_T\right]\right)}=
	O\left( \frac{1}{T^2}\left(\min\limits_{i\ne j} \left(t_{i,T}-s_{i,T}\right)\left(t_{j,T}-s_{j,T}\right)\right)^{-1}\right)= o\left(\frac{1}{T}\left(\min\limits_{i=1,\ldots,p}t_{i,T}-s_{i,T}\right)^{-1}\right)
\end{align*}
due to $ \left(\min\limits_{i=1,\ldots,p}(t_{i,T}-s_{i,T})\right)^{-1}=o(T) $.
Analogously, we obtain that
\begin{align*}
	&\ \abs{\geschweift{\bk\in\Z^p:\frac{\bk}{T}\in\left[\bs_T,\bt_T\right]}}\le 
	T^p\lambda\left(\left[\bs_T,\bt_T\right]\right)+\sum\limits_{j=1}^p \frac{T^p\lambda\left(\left[\bs_T,\bt_T\right]\right)}{T\left(t_{j,T}-s_{j,T}\right)}+R_{2,T},\\
	\intertext{with}&\ \frac{R_{2,T}}{T^p\lambda\left(\left[\bs_T,\bt_T\right]\right)}= o\left(\frac{1}{T}\left(\min\limits_{i=1,\ldots,p}t_{i,T}-s_{i,T}\right)^{-1}\right),
\end{align*}
thus showing the assertion.
	\end{proof}
\end{Lemma}

\begin{Remark} \label{rem_gridpoints_rectangle}
In particular, Lemma \ref{lem_gridpoints_rectangle} implies that for $s_{i,T},t_{i,T}$ not depending on $T$ that there exists a constant $C>0$ not depending  on $T$ such that
	\[ 
\abs{\geschweift{\bk\in\Z^p:\frac{\bk}{T}\in\left[\bs_T,\bt_T\right]}}=
T^p\lambda\left(\left[\bs_T,\bt_T\right]\right)\left(1+O\left(\frac{1}{T}\right)\right)=
T^p\lambda\left(\left[\bs_T,\bt_T\right]\right)\left(1+r_T\left(\bs,\bt\right)\right).
\]
with $\abs{r_T\left(\bs,\bt\right)}\le C/T$. 
\end{Remark}

\subsection{Riemann-integrability and Jordan-measurable sets}	
\begin{Definition} (see \cite{DuistermattKolk}, Definition 6.3.1)\label{def_jordan}
	Let $ A\subset\R^p $ be bounded.
	
	\begin{enumerate}[(a)]
		\item If $ A $ is Jordan-measurable, then
		\[ 
		\vol(A)=\int\limits_{\bx\in\R^p} \mathds{1}_{A}(\bx) dx
		\]	
		is called \emph{Jordan measure}.
		\item $ A $ is \emph{negligible} if $ \vol(A)=0 $.
	\end{enumerate}
\end{Definition}
Note that Jordan-measurable sets are Lebesgue measurable, with the Jordan and the Lebesgue measure coinciding (compare e.g.\ \cite{Salamon}, Theorem 2.24).

\begin{Lemma}\label{lem_jordan_measurable}
	Let $ A,A_1,\ldots,A_n\subset[0,1]^p $ be compact and let $A^{\prime}$ be the complement of $A$ in $[0,1]^p$. 
	\begin{enumerate}[(a)]
		\item (see \cite{DuistermattKolk}, Theorem 6.3.2) $ A $ being Jordan-measurable is equivalent to $ \partial A $ being negligible.
		\item If $ A $ is Jordan-measurable, then $ A^{\prime} $ is Jordan-measurable.
		\item If $A_1,\ldots,A_n$ are Jordan-measurable, then $ A_1\cap\ldots\cap A_n$ is Jordan-measurable, as well. 
		\item (see \cite{MennucciDuci}, Lemma 4.1) If $ A $ is Jordan-measurable, then $ A^{(\gamma)} $ is Jordan-measurable, as well.
	\end{enumerate}
	\begin{proof}
		(b): Since both $ \partial A $ and $ \partial[0,1]^p $ are negligible and $ \partial A^{\prime}=\partial A\cup \partial[0,1]^p $, the assertion follows immediately from (a).\\\\
		(c) For two sets $A,B$ it holds that
		\begin{align*}
			&\ \partial(A\cap B)=\closure{A\cap B}\setminus\inner{(A\cap B)}=\closure{A\cap B}\setminus(\inner A\cap \inner B )
			=\closure{A\cap B}\cap \left(\left(\inner A\right)^{\prime}\cup \left(\inner B\right)^{\prime}\right)\\
			\subset&\ \left(\closure{A}\cap\closure{B}\right)\cap\left(\left(\inner A\right)^{\prime}\cup \left(\inner B\right)^{\prime}\right)\subset \left(\closure{A}\cap\left(\inner A\right)^{\prime}\right)\cup \left(\closure{B}\cap\left(\inner B\right)^{\prime}\right)=\partial A\cup\partial B,
		\end{align*}
	where $A^{\circ}$ and $\closure{A}$ are the interior and the closure of $A$, respectively. 
	It follows inductively that $\partial(A_1\cap\ldots\cap A_n)\subset\partial A_1\cup\ldots\cup \partial A_n$. Hence $\partial(A_1\cap\ldots\cap A_n) $ is negligible and the assertion follows by (a).

	\end{proof}
\end{Lemma}

\begin{Lemma}
	Let $A\subset[0,1]^p$ be a Jordan-measurable set. Then for each $\varepsilon>0$  there exist $N_I(\varepsilon)$, $N_O(\varepsilon)$ and finite partitions of non-degenerate hyperrectangles $ \geschweift{B_{I,1},\ldots,B_{I,N_I(\varepsilon)}} $, $ \geschweift{B_{O,1},\ldots,B_{O,N_O(\varepsilon)}} $ such that
	\begin{align*}
		&\ \sum\limits_{i=1}^{N_I}B_{I,i}\subset A \subset \sum\limits_{i=1}^{N_O}B_{O,i},\ 
		\lambda\left(\sum\limits_{i=1}^{N_O}B_{O,i}\right)-\lambda\left(\sum\limits_{i=1}^{N_I}B_{I,i}\right)<\varepsilon/2.
	\end{align*}
\begin{proof}
	For any set $B$ it holds that
 
\begin{align}
		\supp{\bx\in B}\mathds{1}_A(\bx)=\mathds{1}_{\geschweift{A\cap B\ne\emptyset}},\quad \inff{\bx\in B}\mathds{1}_A(\bx)= \mathds{1}_{\geschweift{B\subset A}}. \label{eq_eins}
\end{align}
	 By the definition of Jordan-measurability, there exists a bounded hyperrectangle $B\subset\R^p$ with $A\subset B$ such that $\mathds{1}_A $ is Jordan-measurable over $B$. Therefore by the definition of Riemann integrability 
%	 (Definition \ref{def_riemann} (c) and (d)) 
	 there exist $N_I(\varepsilon)$, $N_O(\varepsilon)$ and partitions $\mathcal{B}_{I,\varepsilon}=\geschweift{B_{I,1},\ldots,B_{I,N_I(\varepsilon)}}$, 
$\mathcal{B}_{O,\varepsilon}=\geschweift{B_{O,1},\ldots,B_{O,N_O(\varepsilon)}}$ such that
\begin{align*}
	&\ \overline{S}(\mathds{1}_A,\mathcal{B}_{O,\varepsilon})-\lambda(A)<\frac{\varepsilon}{2},\quad \lambda(A)-\underline{S}(\mathds{1}_A,\mathcal{B}_{I,\varepsilon})<\frac{\varepsilon}{2}.
\end{align*}
Let $I_I=\geschweift{i=1,\ldots,N_I(\varepsilon)| B_{I,i}\subset A}$, $I_O=\geschweift{i=1,\ldots,N_O(\varepsilon)|A\cap B_{O,i}\ne\emptyset}$. Then it holds by \eqref{eq_eins} that
\begin{align*}
	&\ \overline{S}(\mathds{1}_A,\mathcal{B}_{O,\varepsilon})=
	\sum\limits_{i=1}^{N_O(\varepsilon)}\supp{\bx\in B_{O,i}}\mathds{1}_A(\bx)\lambda(B_{O,i})=
	\sum\limits \limits_{i\in I_O} \lambda\left(B_{O,i}\right)=
	\lambda\left(\bigcup\limits_{i\in I_O}B_{O,i}\right)
	\\\intertext{and analogously}&\ 
	\underline{S}(\mathds{1}_A,\mathcal{B}_{I,\varepsilon})=\lambda\left(\bigcup\limits_{i\in I_I}B_{I,i}\right).
\end{align*}
Therefore the assertion follows with the partitions $\geschweift{B_{I,i}|i\in I_I}$ and  $\geschweift{B_{O,i}|i\in I_O}$.
\end{proof}
\end{Lemma}

\begin{Lemma}\label{lem_ball_covering}
	Let $A \subset[0,1]^p$ be a compact Jordan-measurable set with $\lambda(A)>0$. For $\varepsilon>0$, there exist $\delta=\delta(\varepsilon)>0$, $N(\delta)$ and $\bx_1,\ldots,\bx_{N(\delta)}$ such that 
	\begin{align*}
		\partial A\subset\bigcup\limits_{i=1}^{N(\delta)} B_{\delta,\infty}(\bx_i),\quad\lambda\left(\bigcup\limits_{i=1}^{N(\delta)} B_{\delta,\infty}(\bx_i)\right)<\varepsilon.
	\end{align*}
	\begin{proof}
		By \cite{DuistermattKolk}, Sec.\ 6.3, there exists a partitioning with hyperrectangles $\geschweift{B_{i},i=1,\ldots,N(\varepsilon,\partial A)}$  such that
		\begin{align}
			\partial A\subset\bigcup_{i=1}^{N(\varepsilon,\partial A)}B_i,\quad\lambda\left(\bigcup_{i=1}^{N(\varepsilon,\partial A)}B_i\right)<\frac{\varepsilon}{2^p}.
			\label{ball_covering0}
		\end{align}		 Let $\delta=\delta(\varepsilon)$ be the minimal side length of any of the $B_i$. There exist $N(\delta)$ and $\bx_1,\ldots,\bx_{N(\delta)}$ such that
		\begin{align}
			\bigcup\limits_{i=1}^{N(\delta)} B_{\delta,\infty}(\bx_i)\supset \bigcup_{i=1}^{N(\varepsilon,\partial A)}B_i,\quad \lambda\left(\bigcup\limits_{i=1}^{N(\delta)} B_{\delta,\infty}(\bx_i)\right)\le 2^p\lambda\left(\bigcup_{i=1}^{N(\varepsilon,\partial A)}B_i\right): \label{ball_covering}
		\end{align}		 
		In dimension $1$, it is possible to cover an interval $[a,a+\beta]$ by intervals $[b_i,b_i+\delta]$ with $\delta\le\beta$ such that $ \lambda\left(\bigcup\limits_j [b_j,b_j+\delta]\right)\le 2\beta$: One can e.g.\ choose $b_j=a+j\delta$ with $j=0,\ldots,\ceil{\beta/\delta}-1$. Then
		\[ 
		\bigcup\limits_{j=0}^{\ceil{\beta/\delta}-2}[b_j,b_j+\delta]\subset[a,a+\beta]\subset\bigcup\limits_{j=0}^{\ceil{\beta/\delta}-1}[b_j,b_j+\delta]=[a,a+\ceil{\beta/\delta}\delta]
		\]
		and since $\delta<\beta$, it holds that $\ceil{\beta/\delta}\delta\le2\beta $. Analogously, we obtain \eqref{ball_covering}.\\
		By combining \eqref{ball_covering0} and \eqref{ball_covering}, the assertion follows.
	\end{proof}
\end{Lemma}

\begin{Lemma}\label{lem_measure_fattening}
	Let $A\subset[0,1]^p$ be a compact Jordan-measurable set.  For $\gamma_T\to 0$ it holds that $\lambda\left(A^{(\gamma_T)}\setminus A\right)\to 0$ as $T\to\infty$.
	\begin{proof}
		By Lemma \ref{lem_jordan_measurable} (d), $A^{(\gamma)}$ is Jordan-measurable, hence $\lambda\left(A^{(\gamma_T)}\setminus A\right)$ is well-defined. For $\varepsilon>0$, by Lemma \ref{lem_ball_covering} there exist $\delta=\delta(\varepsilon)$ and a covering $ \geschweift{B_{\delta,\infty}(\bx_i),i=1,\ldots,N(\delta)}$ with balls of $\partial A$ such that
		\[ 
		\partial A\subset\bigcup\limits_{i=1}^{N(\delta)} B_{\delta,\infty}(\bx_i)\quad\text{and }\lambda\left(\bigcup\limits_{i=1}^{N(\delta)} B_{\delta,\infty}(\bx_i)\right)<\varepsilon/2^p.
		\]
		By Lemma \ref{covering_number} it holds for $T$ sufficiently large (and thus $\gamma_T$ sufficiently small) that
		\[ 
		A^{(\gamma_T)}\setminus A\subset \bigcup\limits_{i=1}^{N(\delta)} B_{\delta+\gamma_T,\infty}(\bx_i)\subset \bigcup\limits_{i=1}^{N(\delta)} B_{2\delta,\infty}(\bx_i).
		\]	
		Therefore it follows that
		\begin{align*}
			\lambda\left(A^{(\gamma_T)}\setminus A\right)\le \lambda\left(\bigcup\limits_{i=1}^{N(\delta)} B_{2\delta,\infty}(\bx_i)\right)<\varepsilon.
		\end{align*}
		As $\varepsilon$ was chosen arbitrarily, the assertion follows.
	\end{proof}
	
\end{Lemma}

\begin{Lemma} \label{lem_jordan_assumptions}
	Let $ A,A_1,\ldots,A_n\subset[0,1]^p $ be compact Jordan-measurable sets not depending on $T$ with $ \lambda(A)>0 $. Then it holds that
\begin{align*}
	\intertext{(a)}
			\abs{\geschweift{\bk\in\Z^p:\frac{\bk}{T}\in A}}=&\ T^p\lambda(A)(1+o(1)),\\ 
%			\label{set}\\
			\intertext{(b)}
			\abs{\geschweift{\bk\in\Z^p:\frac{\bk}{T}\in \bigcap\limits_{i=1}^n A_i}}=&\ \begin{cases}
				T^p\lambda\left(\bigcap\limits_{i=1}^n A_i\right)(1+o(1)),&\text{ if $ \lambda\left(\bigcap\limits_{i=1}^n A_i\right)>0 $}\\
				o(T^p),&\text{ if $ \lambda\left(\bigcap\limits_{i=1}^n A_i\right)=0, $}\end{cases}\\
%				\label{intersection}\\
			\intertext{(c)} 
			\abs{\geschweift{\bk\in\Z^p:\frac{\bk}{T}\in A^{(\gamma_T)}\setminus A}}=&\ o(T^p) 
%			\label{extension2}
		\end{align*}	
	for $\gamma_T\to 0$ as $T\to\infty$.
		
	\begin{proof}
	(a) 
	We need to show that 
\begin{align*}
		&\ \abs{\frac{\abs{\geschweift{\bk\in\Z^p:\frac{\bk}{T}\in A}}}{T^p}-\lambda(A)}=o(1).
\end{align*}
	By the definition of Riemann-integrability, for $ A $ and $ \varepsilon>0 $, there exist $N_I=N_I(\varepsilon)$, $N_O=N_O(\varepsilon)$ and finite partitions of non-degenerate hyperrectangles $ \geschweift{B_{I,1},\ldots,B_{I,N_I}} $, $ \geschweift{B_{O,1},\ldots,B_{O,N_O}} $ such that
		\begin{align*}
		&\ \sum\limits_{i=1}^{N_I}B_{I,i}\subset A \subset \sum\limits_{i=1}^{N_O}B_{O,i},\ 
			\lambda\left(\sum\limits_{i=1}^{N_O}B_{O,i}\right)-\lambda\left(\sum\limits_{i=1}^{N_I}B_{I,i}\right)<\varepsilon/2.
			\end{align*}
		It holds that
		\begin{align*}
		&\ \abs{\geschweift{\bk\in\Z^p:\frac{\bk}{T}\in A}}\le
			\abs{\geschweift{\bk\in\Z^p:\frac{\bk}{T}\in \sum\limits_{i=1}^{N_O}B_{O,i}}}=\sum\limits_{i=1}^{N_O} \abs{\geschweift{\bk\in\Z^p:\frac{\bk}{T}\in B_{O,i}}}. 
\end{align*}
By Lemma \ref{lem_gridpoints_rectangle}, Remark \ref{rem_gridpoints_rectangle} and since there are finitely many $B_{O,i}$ there exists $C_{\varepsilon}>0$ not depending on $T$ such that
\begin{align*}
	&\ \abs{\frac{\abs{\geschweift{\bk\in\Z^p:\frac{\bk}{T}\in B_{O,i}}}}{T^p}-\lambda(B_{O,i})}\le\frac{C_{\varepsilon}}{T},
\end{align*}
hence for $\varepsilon>0$ there exists a $T_{\varepsilon}>0$ uniform in $i$ such that
\begin{align*}
		&\ \abs{\frac{\abs{\geschweift{\bk\in\Z^p:\frac{\bk}{T}\in B_{O,i}}}}{T^p}-\lambda(B_{O,i})}<\frac{\varepsilon}{2N_O}
\end{align*}
for $T\ge T_{\varepsilon}$ as $N_O$ is finite.
Therefore it holds for $T\ge T_{\varepsilon}$ that
\begin{align*}
	&\ \frac{\abs{\geschweift{\bk\in\Z^p:\frac{\bk}{T}\in A}}}{T^p}\le \frac{\sum\limits_{i=1}^{N_O} \abs{\geschweift{\bk\in\Z^p:\frac{\bk}{T}\in B_{O,i}}} }{T^p}<\sum\limits_{i=1}^{N_O} \lambda\left(B_{O,i}\right)+\varepsilon/2<\lambda(A)+\varepsilon.\\
	\intertext{Analogously it holds that}
	&\ \frac{\abs{\geschweift{\bk\in\Z^p:\frac{\bk}{T}\in A}}}{T^p}\ge\lambda(A)-\varepsilon
\end{align*}
for $T$ sufficiently large, thus showing the assertion.\\\\		
		(b): 
		By Lemma \ref{lem_jordan_measurable} (c), $ A_1\cap\ldots\cap A_n $ is Jordan-measurable hence $\lambda\left(A_1\cap\ldots\cap A_n\right)$ is well-defined and by the definition of Riemann-integrability, for $\varepsilon>0$ there exist
		$N_I=N_I(\varepsilon)$, $N_O=N_O(\varepsilon)$ and finite partitions of hyperrectangles $ \geschweift{B_{I,1},\ldots,B_{I,N_I}} $, $ \geschweift{B_{O,1},\ldots,B_{O,N_O}} $ such that
		\begin{align}
			&\ \sum\limits_{i=1}^{N_I}B_{I,i}\subset \bigcap\limits_{i=1}^n A_i \subset \sum\limits_{i=1}^{N_O}B_{O,i},\ 
			\lambda\left(\sum\limits_{i=1}^{N_O}B_{O,i}\right)-\lambda\left(\sum\limits_{i=1}^{N_I}B_{I,i}\right)<\varepsilon/2. \label{cover_jordan_intersect}
		\end{align} 
	  If $ \lambda\left(\bigcap\limits_{i=1}^n A_i\right)=0 $, \eqref{cover_jordan_intersect}, implies that for $ \varepsilon>0 $ there exists a covering $ \bigcup\limits_{i=1}^{N(\varepsilon)}B_i \supset\bigcap\limits_{i=1}^n A_i $ with hyperrectangles $ B_i $ such that
		$ 		\lambda\left(\bigcup\limits_{i=1}^{N(\varepsilon)}B_i\right)<\varepsilon/2.$
		Analogously to the proof of (a) there exists $T_{\varepsilon}$ such that
		\begin{align*}
			&\ \abs{\frac{\abs{\geschweift{\bk\in\Z^p:\frac{\bk}{T}\in B_{i}}}}{T^p}-\lambda(B_{i})}<\frac{\varepsilon}{2N(\varepsilon)}
		\end{align*}
		for $T\ge T_{\varepsilon}$. Therefore it holds for $T\ge T_{\varepsilon}$ that
		\begin{align*}
		\frac{\abs{\geschweift{\bk\in\Z^p:\frac{\bk}{T}\in \bigcap\limits_{i=1}^n A_i}}}{T^p}\le \frac{\sum\limits_{i=1}^{N(\varepsilon)}\abs{\geschweift{\bk\in\Z^p:\frac{\bk}{T}\in B_i}}}{T^p}<\sum\limits_{i=1}^{N(\varepsilon)}\lambda\left(B_i\right)+\varepsilon/2<\varepsilon.
			\end{align*}
		Since $ \varepsilon>0 $ was chosen arbitrarily, the assertion follows.\\\\
		If $ \lambda\left(\bigcap\limits_{i=1}^n A_i\right)>0 $, the assertion immediately  follows by (a).\\\\
		(c): By Lemma \ref{lem_jordan_measurable} (d), $A^{(\gamma_T)}$ is Jordan-measurable hence $\lambda\left(A^{(\gamma_T)}\setminus A\right)$ is well-defined. 
		Analogously to the proof of Lemma \ref{lem_measure_fattening}, for $\varepsilon>0$ there exist $\delta=\delta(\varepsilon)$, $N(\delta)$ and $\bx_1,\ldots,\bx_{N(\delta)}$ such that
		\[ 
		A^{(\gamma_T)}\setminus A\subset  \bigcup\limits_{i=1}^{N(\delta)} B_{\delta,\infty}(\bx_i)\text{ and }\lambda\left(\bigcup\limits_{i=1}^{N(\delta)} B_{\delta,\infty}(\bx_i)\right)<\varepsilon/2.
		\]
		Analogously to (a) there exists $T_{\varepsilon}>0$ such that
		\begin{align*}
	&\ \abs{\frac{\abs{\geschweift{\bk\in\Z^p:\frac{\bk}{T}\in B_{\delta,\infty}(\bx_i)}}}{T^p}-\lambda(B_{\delta,\infty}(\bx_i))}<\frac{\varepsilon}{2N(\delta)}
\end{align*}
		for $T\ge T_{\varepsilon}$. Therefore it holds for $T\ge T_{\varepsilon}$ that
		\begin{align*}
		&\ \frac{\abs{\geschweift{\bk:\frac{\bk}{T}\in A^{(\gamma_T)}\setminus A}}}{T^p}\le \sum\limits_{i=1}^{N(\delta)}\frac{\abs{\geschweift{\bk:\frac{\bk}{T}\in  B_{\delta,\infty}(\bx_i)}}}{T^p}
			< \sum\limits_{i=1}^{N(\delta)}\lambda\left(B_{\delta,\infty}(\bx_i)\right)+\varepsilon/2<\varepsilon.
			\end{align*}
		As $\varepsilon$ was chosen arbitrarily, the assertion follows.
		\end{proof}
\end{Lemma}

\subsection{Sets with piecewise Lipschitz boundary}
\begin{Lemma}(compare \cite{Alexdiss}, Lemma B.32 (b)) \label{lem_Alex}
	Let $ \mathcal{S}_1,\mathcal{S}_2 $ be sets and $ d_1,d_2 $ be semimetrices on  $ \mathcal{S}_1$ and $\mathcal{S}_2 $, respectively. Let $ f:\mathcal{S}_1\to\mathcal{S}_2 $ be surjective such that there exist $ a,b>0 $ with
	\[ 
	d_2(f(x),f(y))\le b^ad_1^a(x,y)\quad\text{for all }x,y\in\mathcal{S}_1. 
	\]
	Then 
	\[ 
	N\left(\varepsilon,\mathcal{S}_2,d_2\right)\le N\left(\frac{\varepsilon^{1/a}}{b},\mathcal{S}_1,d_1\right).
	\]
\end{Lemma}

\begin{Lemma}\label{covering_number_ass}
	If $A\subset[0,1]^p$ fulfills Assumption \ref{ass_sets}, then there exists a suitable constant $C>0$ such that
	\[ 
	N\left(\varepsilon,\partial A\right)\le\frac{C}{\varepsilon^{p-1}}.
	\]
	\begin{proof}
		Obviously, $N\left(\varepsilon,[0,1]^{p-1}\right)\le 1/\varepsilon^{p-1}$ as
		\begin{align}
			[0,1]^{p-1}\subset\bigcup\limits_{\bk\in\geschweift{1,\ldots,\floor{1/\varepsilon}}^{p-1}}B_{\varepsilon,\infty}\left(\bk\varepsilon\right).
%			\label{covering_unit_square}
		\end{align}
		By assumption there exist
		Lipschitz continuous functions $ f_1,\ldots,f_K $ with Lipschitz constant $L$ such that $\partial A\subset\bigcup\limits_{i=1}^K f_i\left([0,1]^{p-1}\right)$.  By Lemma \ref{lem_Alex} we obtain with a suitable constant $C>0$ only depending on $K$ and $L$ that
		\begin{align*}
			N(\varepsilon,\partial A)\le  K N\left(\varepsilon/L,[0,1]^{p-1}\right)\le \frac{C}{\varepsilon^{p-1}},
		\end{align*}
		thus showing the assertion.
	\end{proof}
	
\end{Lemma}

\begin{Lemma} \label{ass21_implies_jordan}
	If $A\subset[0,1]^p$ fulfills Assumption \ref{ass_sets}, $A$ is Jordan-measurable.
	\begin{proof}
		For $\varepsilon>0$, by Lemma \ref{covering_number_ass} there exists a covering
		\begin{align*}
			&\ \partial A\subset\bigcup\limits_{i=1}^{N(\varepsilon,\partial A)}B_{\varepsilon,\infty}(\bx_i)\\
			\intertext{with}
			&\ 	\lambda\left(\partial A\right)\le N(\varepsilon,\partial A)\varepsilon^p\le C\varepsilon.
		\end{align*}
		As $\varepsilon>0$ was chosen arbitrarily, it follows that $\lambda\left(\partial A\right)=0 $. Therefore, the assertion follows by Lemma \ref{lem_jordan_measurable} (a).
	\end{proof}
\end{Lemma}

\begin{Lemma}\label{convex_jordan}
	If $A\subset\R^p$ is a bounded convex set with $0<\lambda(A)<\infty$, then $A$ is Jordan-measurable.
	\begin{proof}
		By \cite{Widmer}, Lemma 3.1, $A$ fulfills Assumption \ref{ass_sets}. Hence the assertion follows by Lemma \ref{ass21_implies_jordan}.
	\end{proof}
\end{Lemma}

\begin{Lemma}\label{lem_covering_fattening}
	If $ A\subset[0,1]^p $ fulfills Assumption \ref{ass_sets}, it holds with a suitable constant $C>0$ that
\begin{align*}
	\intertext{(a)}
	\lambda\left(A^{(\gamma_T)}\setminus A\right)\le&\ C\gamma_T,
	\intertext{(b)}
		\abs{\geschweift{\bk\in\Z^p:\frac{\bk}{T}\in A^{(\gamma_T)}\setminus A}}\le&\ CT^p \gamma_T
\end{align*}	for $ \gamma_T^{-1}=o(T)$, but bounded.
	\begin{proof}
		(a)
		By assumption there exist Lipschitz continuous functions $ f_1,\ldots,f_K $ with Lipschitz constant $L$ such that $A\subset\bigcup\limits_{i=1}^K f_i([0,1]^p)$.  By Lemma \ref{covering_number_ass} we obtain with a suitable constant $C>0$ that
		$	N(\gamma_T)=N(\gamma_T,\partial A)\le C/\gamma_T^{p-1}.	$	
		By Lemma \ref{covering_number}, we obtain that for a covering of $\partial A$  
		\begin{align*}
			&\ \partial A\subset\bigcup\limits_{i=1}^{N(\gamma_T)}B_{\gamma_T,\infty}(\bx_i),\\
			\intertext{it holds that}
			&\ A^{(\gamma_T)}\setminus A\subset \bigcup\limits_{i=1}^{N(\gamma_T)}B_{2\gamma_T,\infty}(\bx_i).\\
			\intertext{Furthermore, for each $\bx_i$, there exists $\bk_i\in\Z^p$ such that $\left\|\bx_i-\bk_i/T\right\|_{\infty}\le 1/T$, hence for $T$ sufficiently large,}
			&\ A^{(\gamma_T)}\setminus A\subset \bigcup\limits_{i=1}^{N(\gamma_T)}B_{3\gamma_T,\infty}(\bk_i/T)
		\end{align*}
		as $\gamma_T^{-1}=o(T)$. 
		Therefore it holds with a suitable constant $C_1>0$ that
		\begin{align*}
			\lambda\left(A^{(\gamma_T)}\setminus A\right)\le N(\gamma_T)(3\gamma_T)^p\le C_1\gamma_T.
		\end{align*}
		(b) Analogously to (a) it holds by Lemma \ref{lem_gridpoints_rectangle}  with suitable constants $C_1,C_2>0$ that
		\begin{align*}
			&\ \abs{\geschweift{\bk\in\Z^p:\frac{\bk}{T}\in A^{(\gamma_T)}\setminus A}}\le \abs{\geschweift{\bk\in\Z^p:\frac{\bk}{T}\in \bigcup\limits_{i=1}^{N(\gamma_T)}B_{3\gamma_T,\infty}(\bk_i/T)}}\\
			\le&\ \sum\limits_{i=1}^{N(\gamma_T)}\abs{\geschweift{\bk\in\Z^p:\frac{\bk}{T}\in B_{3\gamma_T,\infty}(\bk_i/T)}}=N(\gamma_T)\abs{\geschweift{\bk\in\Z^p:\frac{\bk}{T}\in B_{3\gamma_T,\infty}(0)}}\\
			\le&\ C_1 N(\gamma_T)T^p(3\gamma_T)^p\le C_2T^{p}\gamma_T, 
		\end{align*}
		thus showing the assertion.
	\end{proof}
\end{Lemma}

%\section*{References}	
%\newpage

\bibliography{Bib_Paper_Danobi} 
\addcontentsline{toc}{chapter}{References}

	\end{document}